\documentclass[a4paper,review,1p]{elsarticle}

\usepackage[utf8]{inputenc}
\usepackage[english]{babel}
\usepackage{bbm} %
\usepackage[tbtags]{mathtools}
\usepackage{amssymb,amsmath,amsthm} %
\usepackage{color} %
\usepackage{graphicx} %
\usepackage[]{hyperref} %
\usepackage{cleveref}
\usepackage{booktabs} %
\usepackage{dsfont} %
\usepackage{tikz} %
\usepackage{a4wide}

\usepackage{siunitx}

\usepackage{pgfplots} 
\pgfplotsset{compat=newest}

\usepackage{subcaption}
\usepgfplotslibrary{fillbetween}
\usepgfplotslibrary{colorbrewer}

\newcommand{\bfp}{\mathbf{p}}

\newcommand{\bfP}{\mathbf{P}}

\newcommand{\bbR}{\mathbb{R}}

\renewcommand{\div}{\operatorname{div}}

\newcommand{\dif}{ \ \mathrm{d}}

\DeclareMathOperator{\diag}{diag}
\DeclareMathOperator{\E}{\mathbb{E}}
\DeclareMathOperator{\Var}{Var}
\DeclareMathOperator{\Cov}{Cov}

\newcommand{\LandauO}{\mathcal{O}}

\newcommand{\nKL}{M}

\newcommand{\vt}{t}

\renewcommand{\forall}{\text{for all }}

\newcommand{\dd}{\mathrm{d}}

\newcommand{\ddt}{\frac{\dd}{\dd t}}

\DeclareMathOperator{\tr}{tr}
\newcommand{\D}{\mathrm{d}}

\newcommand{\R}{\ensuremath{\mathbb{R}}}

\newcommand{\Hzerocurl}{\ensuremath{H_0\left(\textrm{curl};\domain\right)}}
\newcommand{\HzerocurlDt}{\ensuremath{H_0\left(\textrm{curl};\domain_t\right)}}
\newcommand{\HzerocurlDtomega}{\ensuremath{H_0\left(\textrm{curl};\domain_t(\omega)\right)}}
\newcommand{\HzerocurlDzero}{\ensuremath{H_0\left(\textrm{curl};\domain_0\right)}}

\newcommand{\domain}{\mathrm{D}}
\newcommand{\boundary}{\partial \domain}

\renewcommand{\vec}[1]{\ensuremath{\boldsymbol{#1}}}
\newcommand{\curl}[1]{\ensuremath{\nabla \times #1}}
\renewcommand{\div}[1]{\ensuremath{\nabla \cdot #1}}

\newcommand{\fembasis}{\ensuremath{N}}

\newcommand{\stiff}{\ensuremath{\mathbf{K}}}
\newcommand{\mass}{\ensuremath{\mathbf{M}}}

\newcommand{\coeffstiff}{\mathbf{C}}
\newcommand{\coeffmass}{\mathbf{A}}
\newcommand{\coeff}{\mathbf{B}}

\newcommand{\freq}{f}

\newcommand{\maptil}{\ensuremath{{\mathbf{G}}}}

\newcommand{\mapnurbs}{\ensuremath{{\mathbf{g}_i}}}
\newcommand{\mapf}{\ensuremath{{\mathbf{f}}}}

\newcommand{\samplespace}{\Omega}
\newcommand{\eventspace}{\mathcal{F}}
\newcommand{\probmeasure}{\mathbb{P}}
\newcommand{\event}{\omega}

\newcommand{\rand}{z}

\newtheorem{theorem}{Theorem}
\newtheorem{remark}[theorem]{Remark}
\newtheorem{lemma}[theorem]{Lemma}

\begin{document}

\begin{frontmatter}
\title{Shape uncertainty quantification of Maxwell eigenvalues and -modes with application to TESLA cavities\tnoteref{t1}}
\tnotetext[t1]{The work of SSc and AZ is partially supported by the Graduate School CE within the Centre for Computational Engineering at TU Darmstadt. The work of JD and DE was partially funded by the Deutsche Forschungsgemeinschaft (DFG, German Research Foundation) – project number 501419255. JD and DE also thankfully acknowledge the support by the DFG under Germany’s Excellence Strategy – project number 390685813. We thank Jacopo Corno for providing us with the code and data to analyze the misaligned TESLA cavities. We also thank Christian Schmitt for his frequent support with running our code on the compute server. The authors gratefully acknowledge the granted access to the Bonna cluster hosted by the University of Bonn.}

\author[1]{Jürgen Dölz}
\ead{doelz@ins.uni-bonn.de}
\author[1]{David Ebert\corref{cor1}}
\ead{ebert@ins.uni-bonn.de}
\author[2]{Sebastian Schöps}
\ead{schoeps@temf.tu-darmstadt.de}
\author[2]{Anna Ziegler}
\ead{anna.ziegler@tu-darmstadt.de}

\cortext[cor1]{Corresponding author.}

\affiliation[1]{organization={Institute for Numerical Simulation, University of Bonn},
    addressline={Friedrich-Hirzebruch-Allee~7},
    postcode={53115},
    city={Bonn},
    country={Germany}}
\affiliation[2]{organization={Computational Electromagnetics Group, Technische~Universität~Darmstadt},
    addressline={Schloßgartenstr.~8},
    postcode={64289},
    city={Darmstadt},
    country={Germany}}

\begin{abstract}
We consider Maxwell eigenvalue problems on uncertain shapes with perfectly conducting TESLA cavities being the driving example. 
Due to the shape uncertainty the resulting eigenvalues and eigenmodes are also uncertain and it is well known that the eigenvalues may exhibit crossings or bifurcations under perturbation.
We discuss how the shape uncertainties can be modelled using the domain mapping approach and how the deformation mapping can be expressed as coefficients in Maxwell's equations. 
Using derivatives of these coefficients and derivatives of the eigenpairs, we follow a perturbation approach to compute approximations of mean and covariance of the eigenpairs.
For small perturbations these approximations are faster and more accurate than sampling or surrogate model strategies.
For the implementation we use an approach based on isogeometric analysis, which allows for straightforward modelling of the domain deformations and computation of the required derivatives.
Numerical experiments for a three-dimensional 9-cell TESLA cavity are presented.
\end{abstract}


\end{frontmatter}

\section{Introduction} \label{sec:intro}

\subsection{Motivation}

Radio-frequency resonators are devices in which electromagnetic fields oscillate at gigahertz frequencies. They are used, for example, to accelerate particles, where the aim is to increase the magnitude of the excitation.
A specific example which will be the driving force of our article is the superconducting TESLA cavity, following the design by~\cite{Aune_2000aa}.
The TESLA cavity, designed for linear accelerators, is a 9-cell superconducting cavity, manufactured from niobium sheets.
It is cooled with superfluid helium to \SI{2}{\kelvin} in operation and the accelerating eigenmode resonates at \SI{1.3}{\giga \hertz}.
The standing wave structure has a total length of approximately \SI{1}{\meter} and consists of nine elliptic cells, i.e. is designed from circles and ellipses, with an equator diameter of approximately \SI{21}{\centi\meter}, see \cref{fig:tesla}.

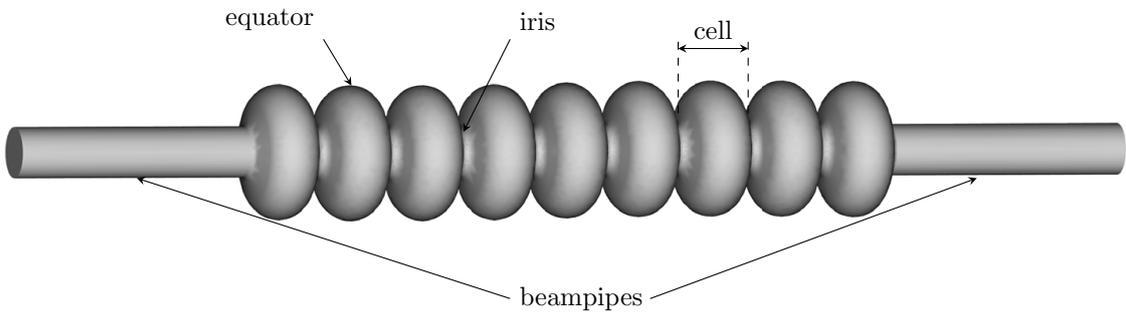
\begin{figure}[htb]
    \centering
        \def\myw{1/12.15*\textwidth}
    \begin{tikzpicture}
        \node[inner sep=0pt, anchor = south west] (tesla) at (0,0)     {\includegraphics[width=\textwidth]{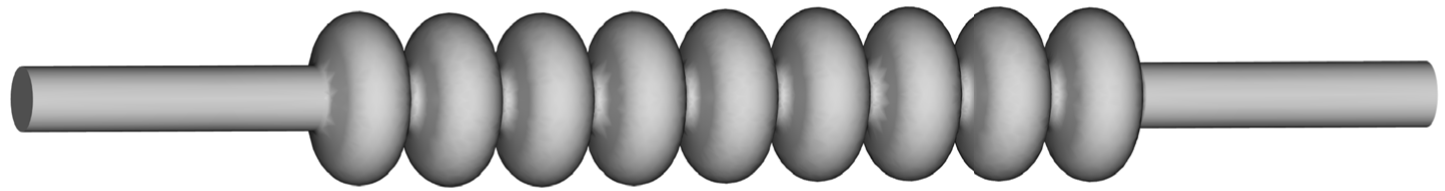}};
        
        \draw[black, -stealth] (7.0*\myw,-0.8*\myw) --  (10.5*\myw,0.5*\myw);
        \draw[black, -stealth] (5.5*\myw,-0.8*\myw) --  (1.5*\myw,0.5*\myw) node[pos = 0, anchor = west] {beampipes};
        
        \draw[black, -stealth] (3.5*\myw,2*\myw) -- (3.8*\myw, 1.5*\myw) node[pos = 0, anchor = south east] {equator};
        
        \draw[black, -stealth] (5.5*\myw,2*\myw) -- (5*\myw, 1*\myw) node[pos = 0, anchor = south west] {iris};
        
        \draw[black, dashed] (7.3*\myw,1.2*\myw) --  (7.3*\myw,2*\myw);
        \draw[black, dashed] (8.05*\myw,1.2*\myw) --  (8.05*\myw,2*\myw);
        
        \draw[black, stealth-stealth] (7.3*\myw,1.9*\myw) -- (8.05*\myw, 1.9*\myw) node[pos = 0.5, anchor = south] {cell};
        
    \end{tikzpicture}
    \caption{The 9-cell TESLA cavity with attached elongated beampipes.}
    \label{fig:tesla}
\end{figure}

During operation, a harmonic time-varying electric field is excited in the cavity whose field distribution and resonating frequency are governed by Maxwell’s equations.
Considering the source-free, time harmonic and lossless case, the magnetic field strength~$\vec{H}$ and the electric field strength~$\vec{E}$ are related by Faraday's law
\begin{equation*}
    \nabla \times \vec{E} = -j\omega\mu_0 \vec{H},
\end{equation*}
where~$\mu_0$ denotes the magnetic permeability of vacuum, and by Ampère's law
\begin{equation*}
    \nabla \times \vec{H} = j\omega\varepsilon_0 \vec{E},
\end{equation*}
with $\varepsilon_0$ indicating the electric permittivity in vacuum. In both equations $\omega$ denotes the unknown angular resonance frequency of the cavity.
The magnetic field strength $\vec{H}$ can be eliminated to obtain the \emph{Maxwell eigenvalue problem} in terms of the electric field strength
\begin{equation*}
    \curl(\curl\vec{E})=\lambda\vec{E}
\end{equation*}
with $\lambda=\omega^2\mu_0\varepsilon_0$ and complemented with the boundary conditions of a perfect electric conductor $\vec{n}\times\vec{E}=0$, with \vec{n} being the outward pointing normal vector. Thus, determining the eigenpair $(\lambda, \vec{E})$, i.e., solving the eigenvalue problem allows to determine the resonance frequencies and their modes.

Cavities are carefully designed to accelerate bunched particles in the longitudinal direction by transferring energy to the bunch, which travels at velocities close to the speed of light.
A precise synchronisation between the electric field and the particle bunch is necessary for the successful and efficient energy transmission and specific eigenmodes of the cavity need to be excited.
The electromagnetic field is characterised by the resonating frequency. 
For the accelerating mode, we also consider the field quality, indicating an even distribution of the field within the cavity.
Amongst other factors, the eigenfrequency and the field distribution are strongly coupled to the geometry of the domain and even small deviations and uncertainties can have a non-negligible effect on the field.
Thus, deformations influencing the distribution and quality of the accelerating mode significantly impact the final efficiency of the acceleration.
Since manufacturing inaccuracies or electromagnetic pressure on the domain wall in operation inevitably result in deviations from the design geometry and thus shift the resonating frequency and affect the performance, a careful investigation of sensitivity and a \emph{shape uncertainty quantification} of the eigenmodes and fields of the eigenvalue problem is necessary.

\subsection{Related Work}

\subsubsection{Uncertainty quantification for Eigenvalue problems}
The mathematical theory for uncertainty quantification of eigenvalue problems seems to be still a relatively little explored field. Early works mainly were developed in the fields of structural analysis and aerospace engineering, see, e.g., the review paper~\cite{AF2007} for investigations in structural dynamics. 
These works are mostly limited to eigenvalues of single multiplicity and rather small algebraic systems. 
To overcome the computational challenges of systems involving partial differential equations, \cite{AndreevSchwab} propose a sparse grid approach and \cite{GGK+2019a,GS2024,GS2024a,Ngu2022,CL2023} apply quasi-Monte Carlo methods. 
While these works are restricted to the case of a single, isolated eigenvalue which does not cross, the trajectories of several eigenvalues may generally show crossings or bifurcations when being perturbed \cite{Kat1995}. 
Rellich showed in \cite{Rel1969} that eigenpairs remain locally analytic even in this case if the perturbation depends on a single, real parameter and gave a characterization of the first (and higher) derivatives. 
This characterization relies on second derivatives to deal with the occurring rank-defects in the characterization and was later exploited in the algorithmic approach of Nelson and Dailey \cite{Dai1989,Nel1976}. Their algorithm was extended in~\cite{Jorkowski_2020aa} for higher derivatives of degenerate eigenpairs.
These derivatives were used in a tracking algorithm to identify eigenvalues where the crossings occur only in one parameter's dimension in a stochastic collocation scheme in \cite{GACS2019}. For eigenvalue problems depending on more than a single scalar parameter, Rellich showed in \cite{Rel1969} that the eigenpairs belonging to an eigenvalue of higher multiplicity are not Fréchet differentiable in general. 
However, \cite{GSH2023} and \cite{DE2024} showed that the eigenspaces remain locally analytic. In \cite{GSH2023}, this was employed in the uncertainty quantification of eigenspaces based on a stochastic collocation approach. 
In \cite{DE2024}, a linear representation of the first Fr\'echet derivative of eigenspaces without the need for second derivatives was derived and exploited for a perturbation-based uncertainty quantification approach.

However, to the best of our knowledge, the influence of shape uncertainties onto the eigenpairs has not been systematically investigated in the mathematical community. This is in contrast to shape uncertainties onto the solution of partial differential equations, where essentially two approaches have emerged. 
The first, the domain mapping approaches, see \cite{CNT16, HPS16, HSSS2018, XT06}, transfer the uncertainty in the domain onto a partial differential equation with random coefficients on a fixed reference domain and are able to deal with large deformations. 
For the computation of statistical quantities of interest such as the mean or the variance of the solutions, high-dimensional quadrature rules such as Monte Carlo are usually employed. The second class of approaches are perturbation approaches, see \cite{BN2014, CS2013, Dol2020, HSS08, JS16}, which use local sensitivity analyses of solutions to partial differential equations to compute statistical quantities of interest. 
Generally speaking the domain mapping approach is best suited for large perturbations, whereas the perturbation or sensitivity-based approach is best suited for smaller perturbations.

\subsubsection{Uncertainty Quantification for Cavities}
Within the electromagnetics community, the effect of uncertainties in cavities on the field quality or resonating frequency have been studied in several works, such as~\cite{Corno_2020, georg_uncertainty_2019, Jacopo, Zadeh, Xiao, Schmidt}.
Deformations arising in cavities include deviations from the design parameters, non axis-symmetric deformations, such as bumps or kinks or mechanical deformations due to Lorentz forces~\cite{Jacopo}.
Furthermore, the welding causes a shrinkage at the welding point and a random misalignment of the cells with respect to the ideal cavity axis~\cite{Corno_2020, Corno_thesis}. Sensitivity analyses are carried out on cavities in order to state the impact of deviations from the design introduced in the production process.
For the TESLA cavity, details of the manufacturing procedure are given in~\cite{Aune_2000aa, Corno_2020}, describing the effects of the deep drawing of the half-cells from niobium sheets, the machining, welding and trimming.
Furthermore, the processes to control and balance misalignment due to welding and the necessity for cleaning, chemical treatments, grinding and final tuning are illustrated.
Additionally,~\cite{Aune_2000aa} state the requirements on the niobium sheets.
The works~\cite{Jacopo, Corno_2020} investigate an uncertainty quantification for the cavity design parameters.
Also in~\cite{Schmidt}, the authors assume deviations of $13$ uncertain design parameters of the TESLA mid-cell and compare a Monte Carlo simulation and uni-variate and multi-variate generalized Polynomial Chaos expansions to estimate some stochastic properties. 
In~\cite{Brackebusch_2012,Bra2016}, a series expansion from the eigenmodes of an unperturbed geometry based on matrix perturbation theory is proposed, in order to avoid computationally expensive parameter studies for cavity perturbations.
Investigating the impact of parameter uncertainties on the dipole mode frequencies and the external quality factor, the authors of~\cite{Xiao_2007} use a mesh distortion method.
A stochastic response surface model is proposed in \cite{Deryckere_2012}, using the derivatives of the system matrices with respect to geometric parameters.
In \cite{Zadeh}, a perturbation is applied to the system matrices of the eigenvalue problem and eigenpair derivatives at the unperturbed geometry are calculated. 
These are employed for an efficient calculation of the eigenvalues and eigenvectors of the perturbed geometry.

Aside from production inaccuracies, \cite{Corno_2016} investigates the effects due to electromagnetic radiation pressure and considers the frequency shift of the accelerating mode in the pillbox cavity and in the 1-cell TESLA cavity applying a linear elasticity problem on the cavity walls.
The authors of~\cite{georg_uncertainty_2019} extend uncertainty investigations due to manufacturing imperfections to the more realistic case of misalignment of the cells due to welding.
A truncated Karhunen–Loève expansion is performed on measurement data of perturbed geometries and a stochastic collocation method based on sparse grids at predefined collocation points is employed.
For this approach, an eigenvalue tracking procedure is proposed to ensure consistency of the solutions.

\subsection{Contributions}
As shape uncertainties in TESLA cavities are usually small, a perturbation-based uncertainty quantification approach can be expected to be beneficial \cite{Brackebusch_2011aa, Brackebusch_2012, Brackebusch_2013aa,Bra2016}. 
The purpose of this article is to derive such an approach for the Maxwell eigenvalue problem. 
The contributions are as follows:
\begin{enumerate}
\item We provide a characterization of the shape sensitivities of the eigenpairs for the Maxwell eigenvalue problem in terms of partial differential equations. In contrast to previous approaches based on matrix perturbation theory, this provides larger flexibility for the used solvers.
\item We extend the perturbation approach to uncertainty quantification for eigenvalue problems from \cite{DE2024} to the case of random domains on the example of the Maxwell eigenvalue problem.
\item We provide closed-form characterizations of the deformation coefficients arising from discretized domain deformations using the isogeometric framework from \cite{ziegler_computation_2022}.
\item We demonstrate the feasibility and effectiveness of our approach on TESLA cavities with deformation fields obtained from real world data.
\end{enumerate}

\subsection{Outline}
The paper is structured as follows. In \cref{sec:deterministicProblem}, we introduce the problem formulation and our method of modelling the geometry deformations.
We review the formulation of the domain mappings as well as the computation of their derivatives.
On this basis, in \cref{sec:UQ}, we present the stochastic model and the efficient computation of the covariances of the derivatives. 
In \cref{sec:disc}, we state the discretized version of the Maxwell eigenvalue problem, of its derivative and of the covariance equation.
We discretize the probability space using a truncated Karhunen-Loève style expansion in \cref{sec:computation} and specify our chosen discretization method, the isogeometric analysis.
We demonstrate our work in \cref{sec:num_examples}, where we perform an uncertainty quantification for the TESLA cavity. 
To this end, we assume a misalignment of the cavity cells with respect to the ideal axis.
Finally, we conclude our work in \cref{sec:conclusion}.

\section{Shape perturbations of Maxwell's eigenvalue problem} \label{sec:deterministicProblem}

\subsection{Maxwell's Eigenvalue Problem}
We consider the Maxwell eigenvalue problem on the bounded Lipschitz domain~$\domain\subset\bbR^3$ with Dirichlet boundary data in the source-free and homogeneous case with lossless materials, i.e. the problem to find
$({\lambda},\vec{E}) \in \R \times \Hzerocurl, \vec{E} \neq 0$
such that
\begin{subequations}
\begin{align} 
    \curl (\curl \vec{E} ) &= {\lambda} \vec{E} && \text{in } \domain , \\
    \div \vec{E} &= 0 && \text{in } \domain , \\
    \vec{E} \times \vec{n} &= 0 && \text{on } \boundary , \\
    \big(\vec{E}, \vec{E}\big)_\domain &= 1 ,
\end{align}\label{eq:evp}
\end{subequations}%
with $\vec{n}$ being the outward pointing normal vector and $(\cdot,\cdot)_\domain$ denoting the canonical inner product in $\big[L^2(\domain)\big]^3$.
There are countably infinitely many positive eigenvalues, which we sort in ascending order 
\begin{align*}
    0 < \lambda^{(1)} \le \lambda^{(2)} \le \ldots ,
\end{align*}
and the eigenmodes $\vec{E}$ of this eigenvalue problem are unique up to orientation and linear combination within the eigenspace, cf. \cite{boffi_2010}. 
The eigenmodes $\vec{E}$ are the amplitudes of electric field strength, which oscillates with the resonant frequency $\freq$ related to the eigenvalues $\lambda$ via
\begin{align} 
    \freq = \frac{\sqrt{\lambda}c_0}{2\pi},
\end{align}
where $c_0$ is the speed of light in vacuum.
\Cref{fig:E-modes} shows the first nine nonzero eigenvalues and the corresponding normalized magnitude of the electrical field strength of a 9-cell TESLA cavity. 
The accelerating mode is the ninth eigenmode, it is displayed in \cref{fig:E-mode9} for a tuned cavity \cite{Corno_2020}.
We note that the first nine eigenvalues each are of single multiplicity.

\begin{figure}
    \begin{subfigure}[c]{.8\textwidth}
    \centering
    \includegraphics[trim = 200 310 200 300, clip, width = .65\textwidth]{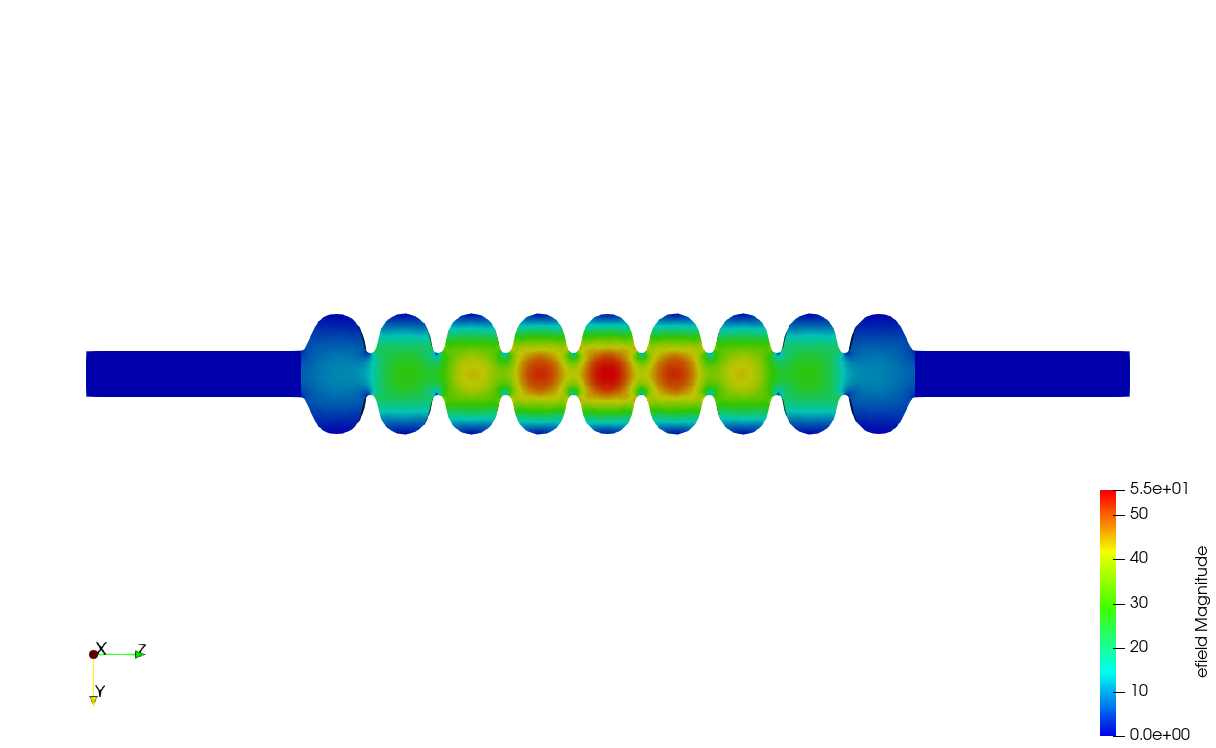}
    \caption{$f = \SI{1.2777}{\giga \hertz}$}
\end{subfigure} \vspace{0.1em}

\begin{subfigure}[c]{.8\textwidth}\centering
    \includegraphics[trim = 200 310 200 300, clip, width = .65\textwidth]{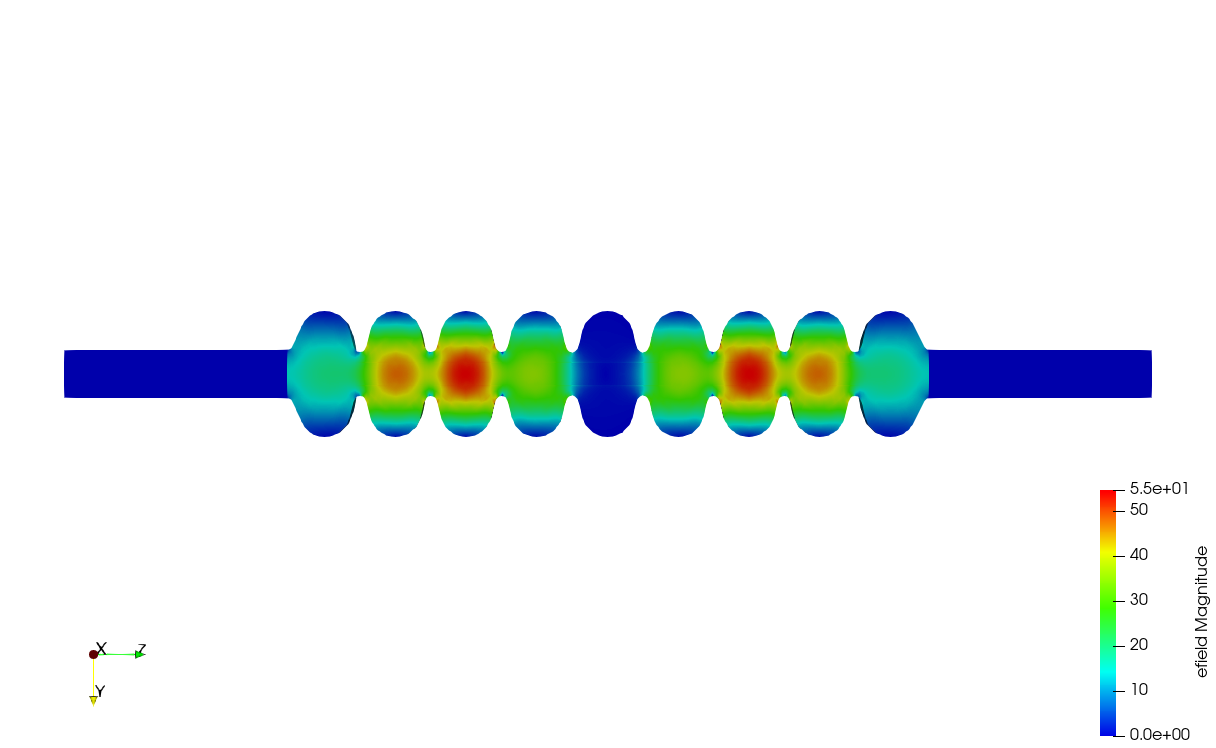}
    \caption{$f = \SI{1.2798}{\giga \hertz}$}
    \end{subfigure}\vspace{0.1em}

\begin{subfigure}[c]{.8\textwidth}\centering
    \includegraphics[trim = 200 310 200 300, clip, width = .65\textwidth]{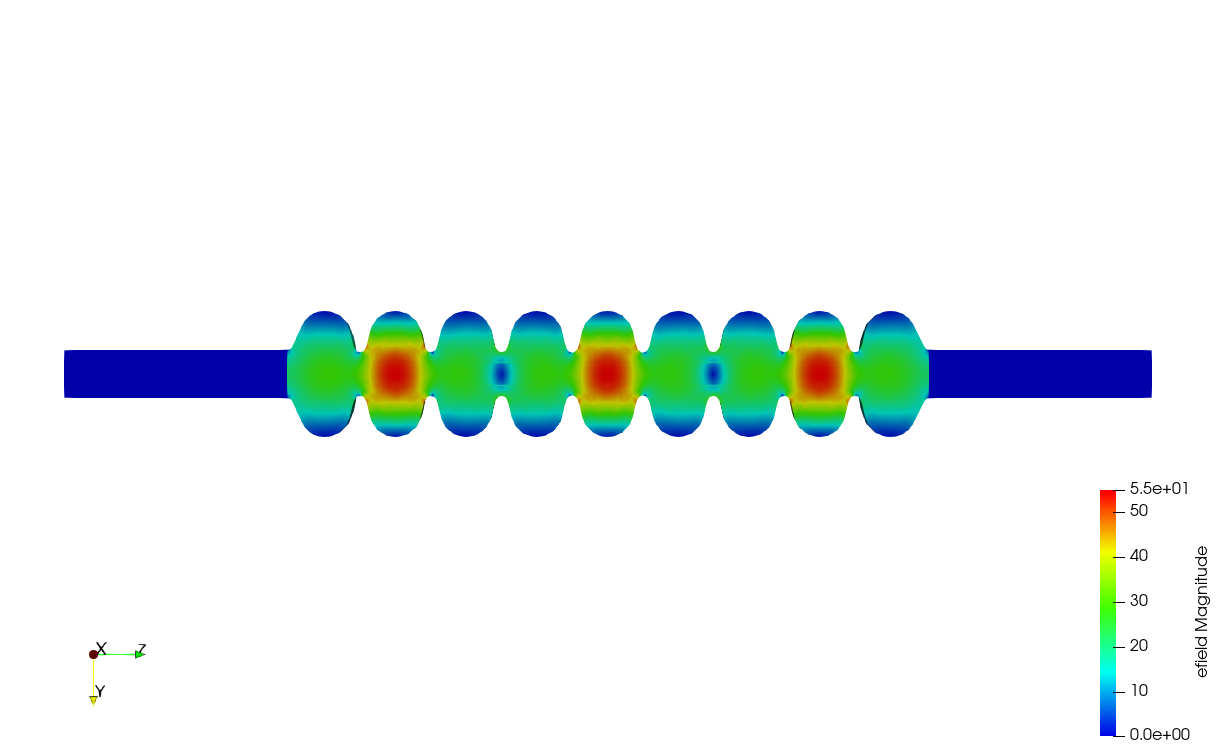}
    \caption{$f = \SI{1.2830}{\giga \hertz}$}
    \end{subfigure}\vspace{0.1em}

\begin{subfigure}[c]{.8\textwidth}\centering
    \includegraphics[trim = 200 310 200 300, clip, width = .65\textwidth]{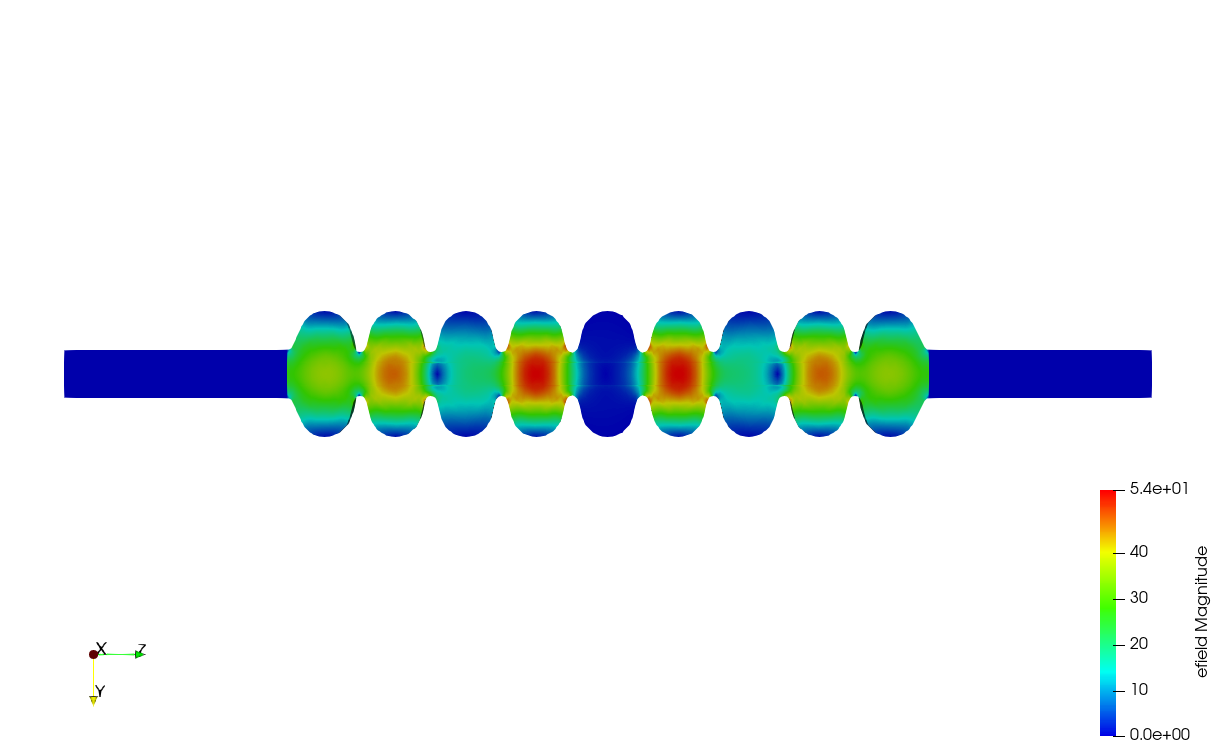}
    \caption{$f = \SI{1.2870}{\giga \hertz}$}
    \end{subfigure}\vspace{0.1em}

\begin{subfigure}[c]{.8\textwidth}\centering
    \includegraphics[trim = 200 310 200 300, clip, width = .65\textwidth]{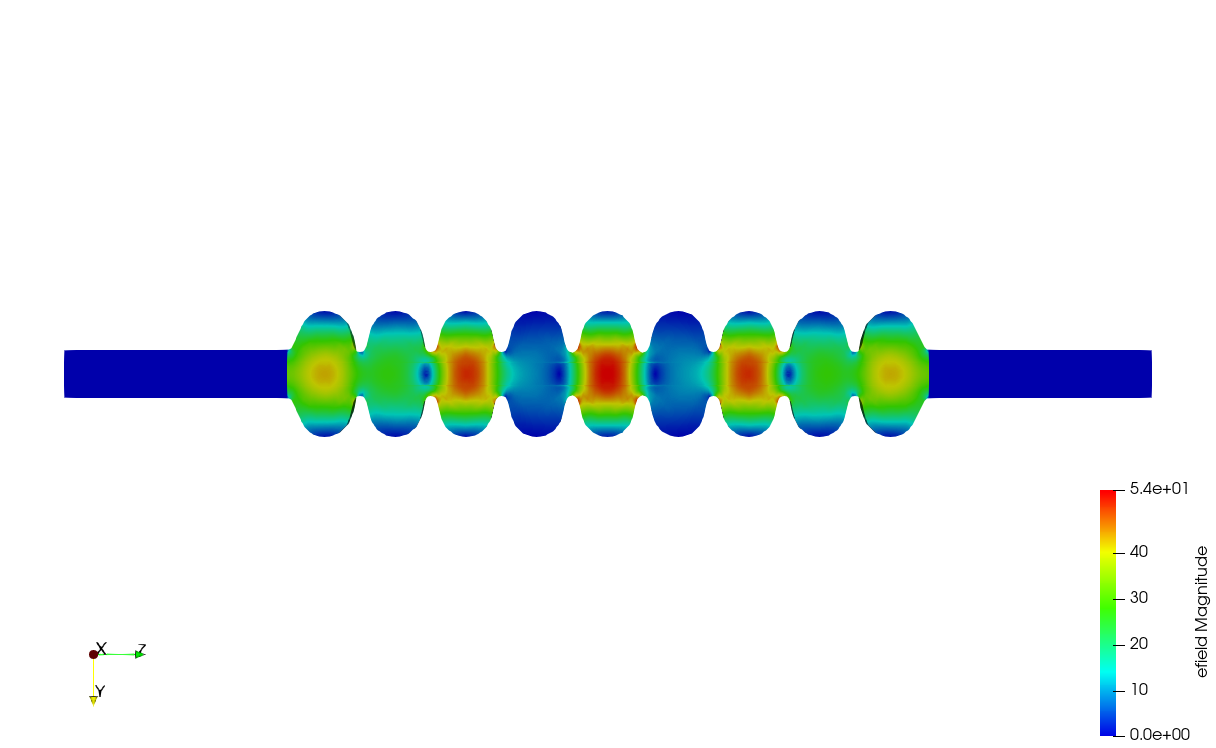}
    \caption{$f = \SI{1.2912}{\giga \hertz}$}
    \end{subfigure}\vspace{0.1em}

\begin{subfigure}[c]{.8\textwidth}\centering
    \includegraphics[trim = 200 310 200 300, clip, width = .65\textwidth]{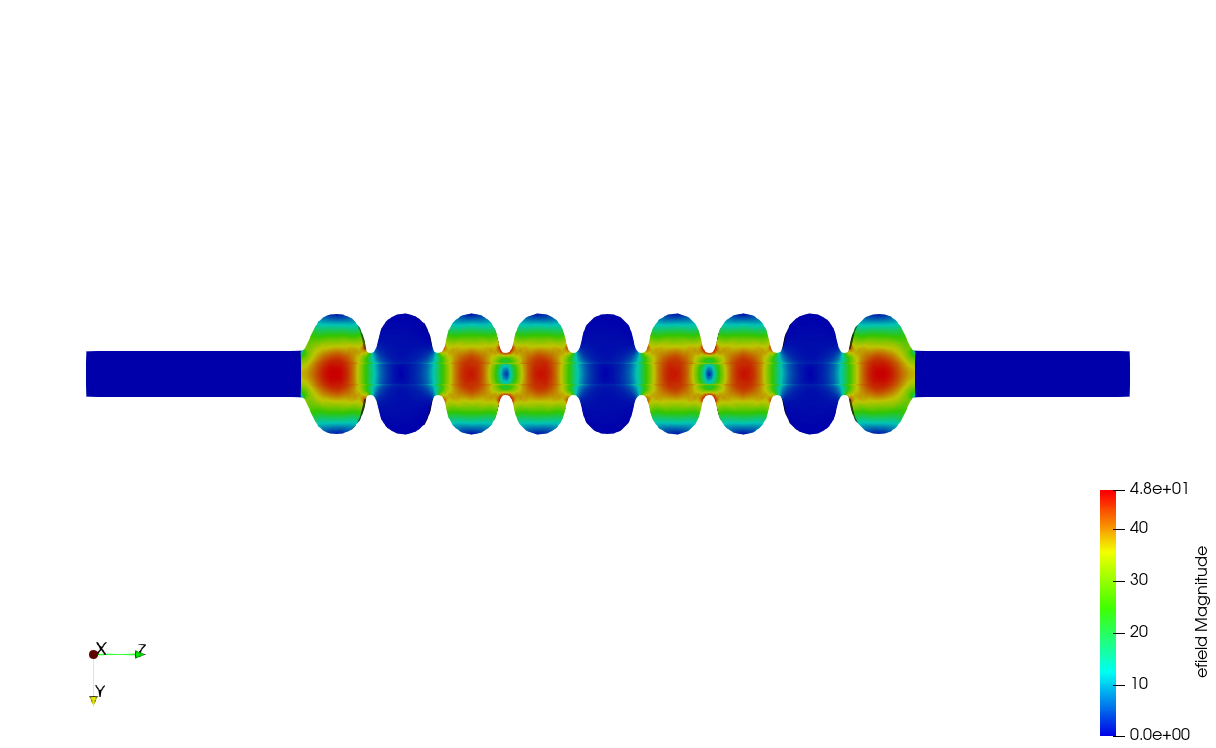}
    \caption{$f = \SI{1.2952}{\giga \hertz}$}
    \end{subfigure}\vspace{0.1em}

\begin{subfigure}[c]{.8\textwidth}\centering
    \includegraphics[trim = 200 310 200 300, clip, width = .65\textwidth]{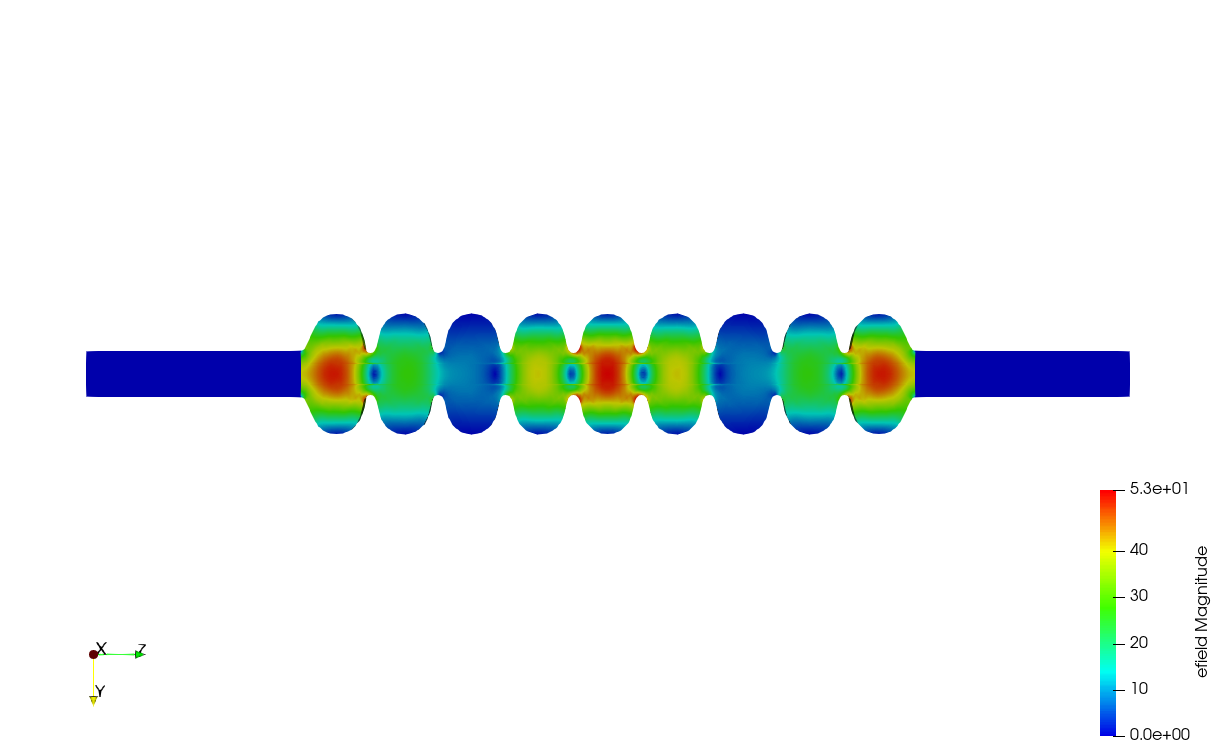}
    \caption{$f = \SI{1.2985}{\giga \hertz}$}
    \end{subfigure}\vspace{0.1em}

\begin{subfigure}[c]{.8\textwidth}\centering
    \includegraphics[trim = 200 310 200 300, clip, width = .65\textwidth]{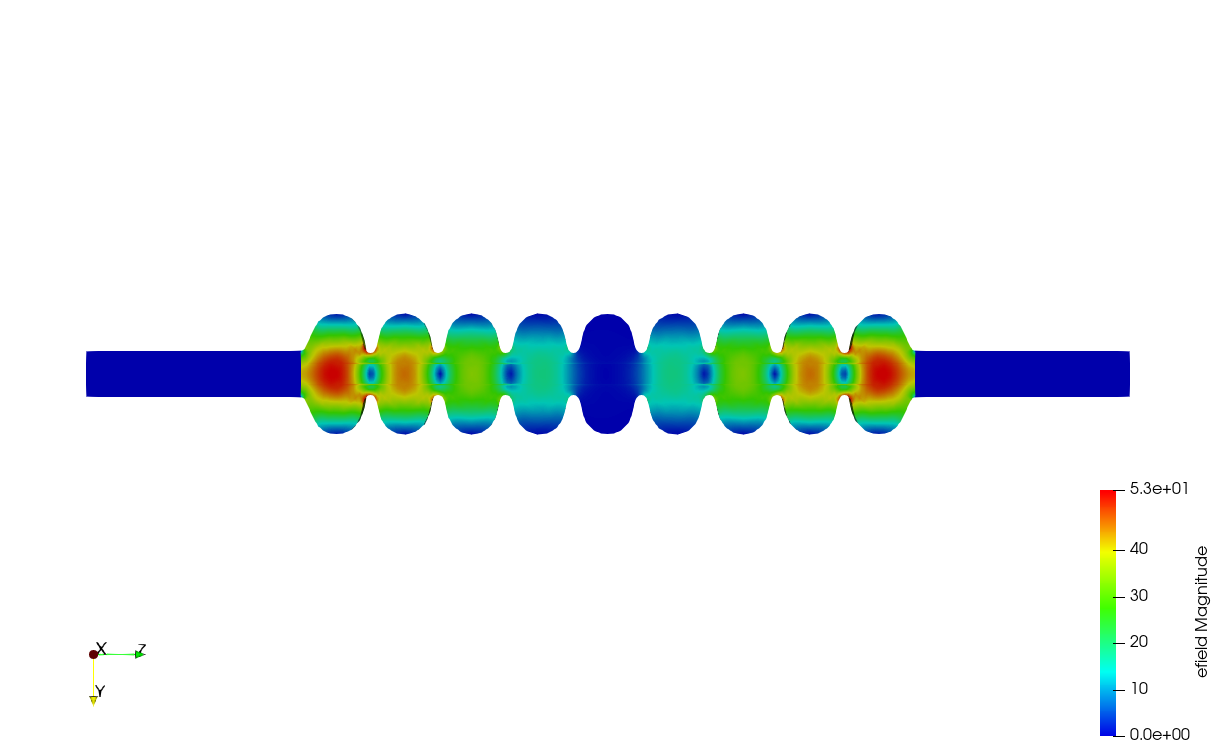}
    \caption{$f = \SI{1.3006}{\giga \hertz}$}
    \end{subfigure}\vspace{0.1em}

\begin{subfigure}[c]{.8\textwidth}\centering
    \includegraphics[trim = 200 310 200 300, clip, width = .65\textwidth]{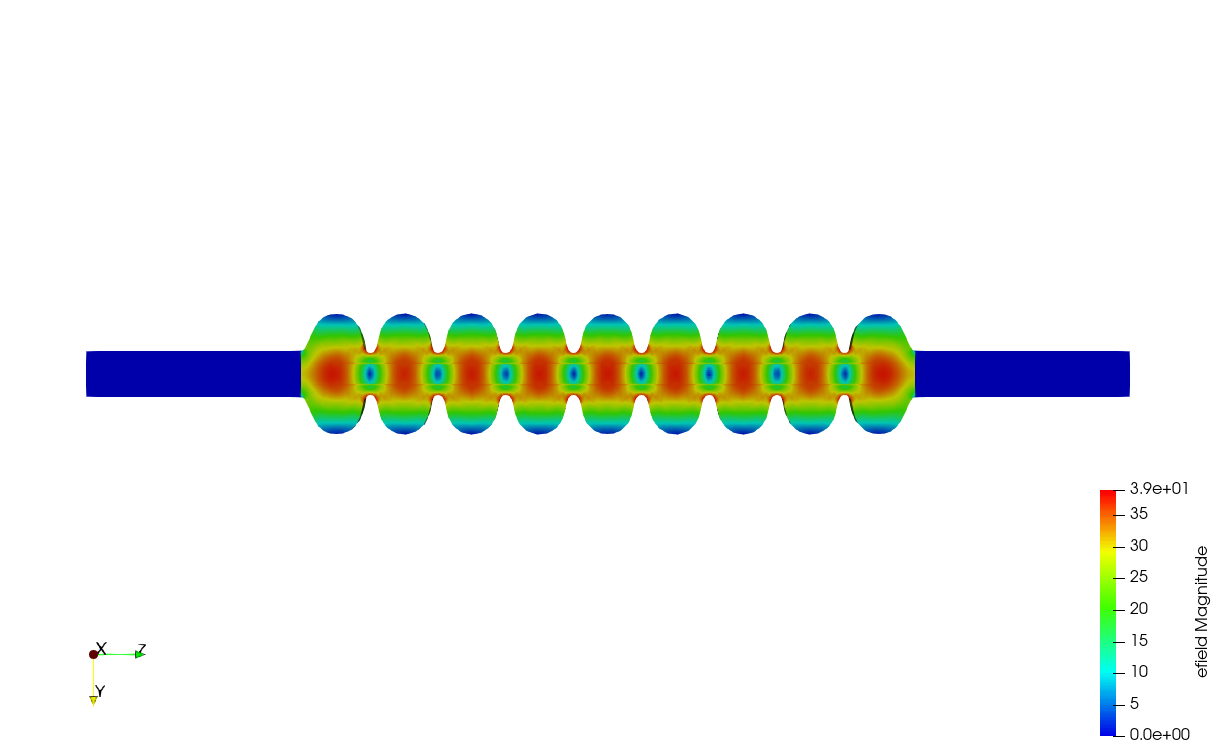}
    \caption{$f = \SI{1.3014}{\giga \hertz}$}
    \label{fig:E-mode9}
\end{subfigure}

\begin{tikzpicture}[overlay, xshift = 0.9\textwidth, yshift = 0.15\textwidth]
 \def\x{0.15}
 \def\h{-1.52}
 \def\d{1.486}
     \node[anchor=south] at (0,-1.65cm) {\includegraphics [width=0.4cm, height = 1.5cm] {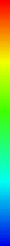}};
     \foreach \y in {0,0.2,0.4,0.6,0.8,1} \draw (\x,\h+\d*\y)--(\x+0.1,\h+\d*\y) node [anchor=west] {\footnotesize $\y$};
     \node[] at (0.1,0.35) {\footnotesize $|\vec{E}|$ $(\SI{}{\volt/\meter})$};
\end{tikzpicture}

    \caption{Nine smallest eigenvalues of a 9-cell TESLA cavity and respective eigenmodes (cross-section along the longitudal direction, color map according to the normalized magnitude of the electric field strength).}
    \label{fig:E-modes}
\end{figure}
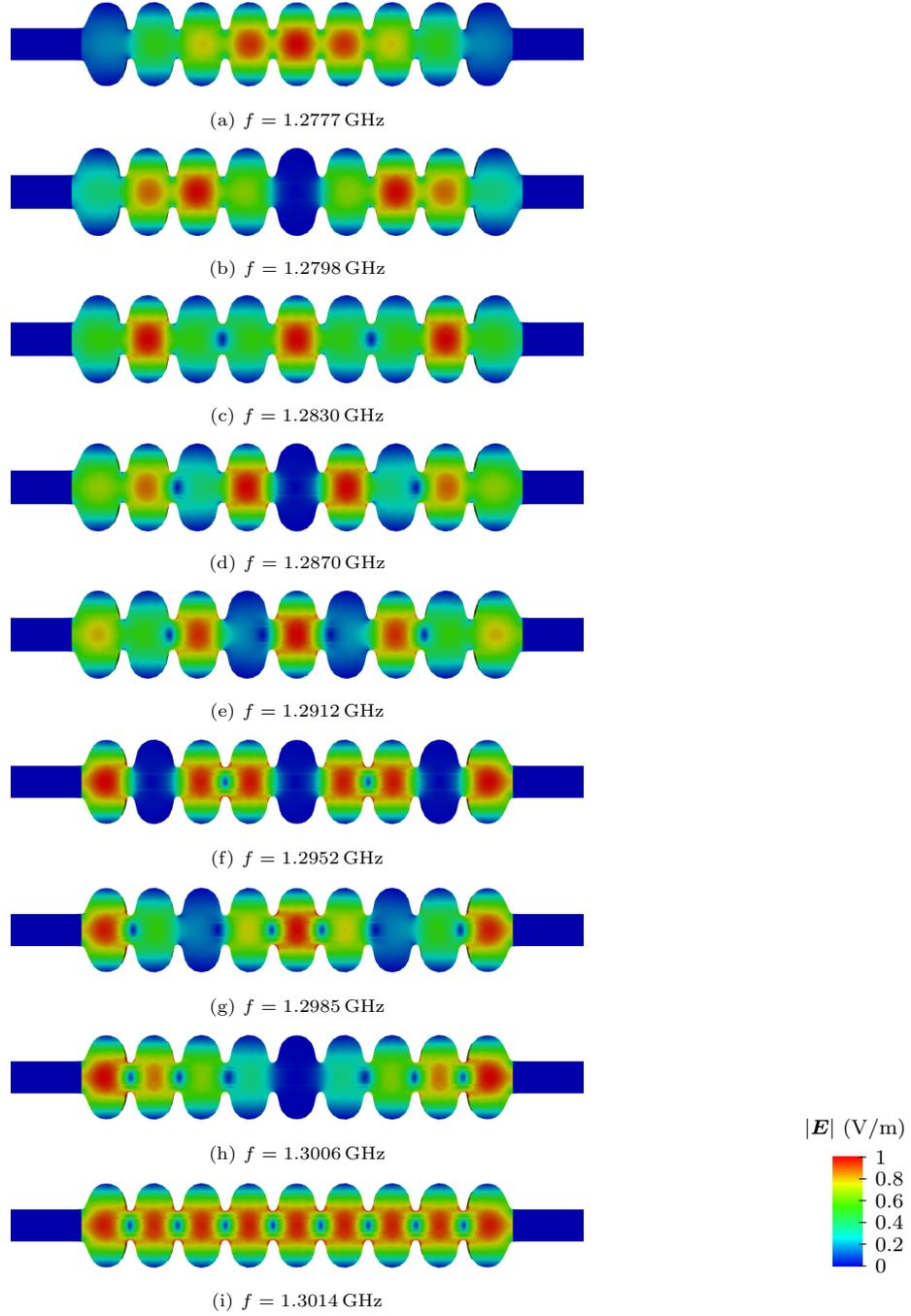

\subsection{Shape Perturbation Model}
We model the possible shape perturbations by means of the domain mapping approach with respect to a physical reference domain $\domain_0 \in\R^3$. We assume that the deformations are parametrized by parameters $\vec{\rand} = [\rand_1,\ldots,\rand_\nKL] \in [-1,1]^\nKL=:\mathbb{U}$, where we allow $\nKL = \infty$, and a scaling parameter ${t}\in \R$, which scales the amplitude of the deformation.
A perturbed domain $\domain_{\vt}$ is then given by
\begin{equation}\label{eq:newdomain}
\domain_{\vt}=\maptil_{\vt}(\domain_0;\vec{\rand}),
\end{equation}
with
\begin{align} \label{eq:deformation}
    \maptil_{\vt}: \domain_0 \to \R^3,
    \qquad
    \vec{x}\mapsto\maptil_{\vt}(\vec{x};\vec{\rand}),
\end{align}
and 
\begin{equation}\label{eq:maptilass}
\maptil_{0}(\vec{x};\vec{\rand}) = \vec{x} = \maptil_{\vt}(\vec{x};\vec{0}).
\end{equation}
To avoid technicalities we assume that the \emph{deformation field} $\maptil_{\vt}$ is continuously differentiable in $\vec{x}$ with the derivative in $\vec{x}$, i.e. $\partial \maptil_\vt$, being smooth in $\vt$. In practice, the domain deformations are often modelled by means of Karhunen-Lo\`eve type expansions as discussed in \cref{sec:KL}. These models satisfy the above assumption \eqref{eq:maptilass}, but our approach can also deal with more general domain deformation models.

An example of such a deformation for the TESLA cavity from \cref{fig:tesla} can be seen in \cref{fig:deformation}, where the deformation is visualized amplified by a factor $250$ and the middle axis of the cavity is depicted in red. 
We note that we define the reference geometry as the sample mean geometry of a set of measured cavities.
This is the reason why we observe a slightly curved central axis in \cref{fig:refDomain} which obviously deviates from the nominal design. 
The example deformation in \cref{fig:deformation} is obtained by generating a deformation direction using the Karhunen-Lo\`eve type expansion with the random variables $\vec{\rand}$. 
For visualisation purposes, this is amplified by a factor $250$.
After determining this displacement direction, it was scaled by $t = \{0.5,1\}$ and applied to the reference geometry to illustrate the domains in \cref{fig:Domain05} and \cref{fig:Domain1}, respectively.
For a more detailed description of this setting, we refer the reader to \cref{sec:num_example_setup}.

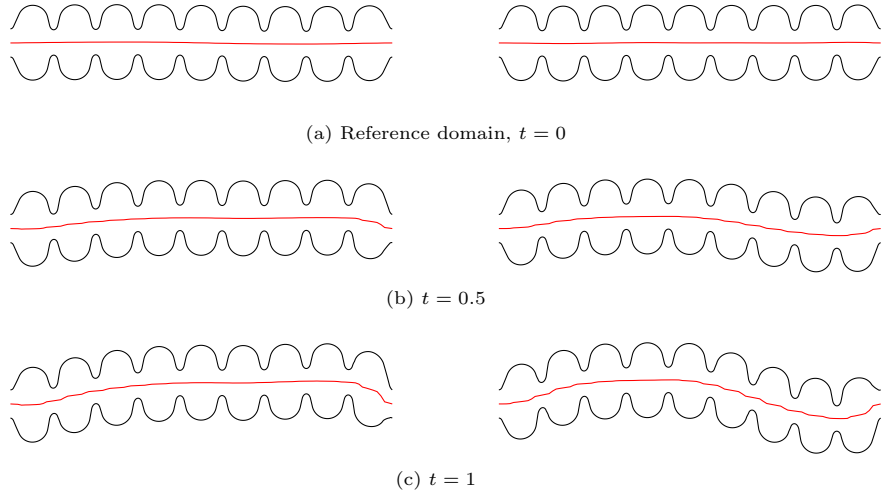
\begin{figure}
    \centering
    {\pgfplotsset{
    axis line style={draw=none},
    tick style={draw=none},
    yticklabels={,,},
    xticklabels={,,},
    width = 0.6\textwidth,
    axis equal image,
    }

    \centering
    \def\minleft{-0.5224}
    \def\maxright{0.5240}
    \def\shift{0.25}
\begin{subfigure}[c]{.85\textwidth}
\begin{tikzpicture}
\begin{axis}[
ymin = -0.11,
ymax = 0.11,
	]
\addplot[color=black] table[x=z,y=y,col sep=comma]{data/t_0_iptc_2_yz.csv};
\addplot[color=black] table[x=z,y=y,col sep=comma]{data/t_0_iptc_5_yz.csv};
\addplot[color = red] table[x=z,y=y,col sep=comma]{data/t_0_axis_yz.csv};
\end{axis}
\end{tikzpicture}~
\begin{tikzpicture}
\begin{axis}[
ymin = -0.11,
ymax = 0.11,
	]
\addplot[color=black] table[x=z,y=x,col sep=comma]{data/t_0_iptc_2_xz.csv};
\addplot[color=black] table[x=z,y=x,col sep=comma]{data/t_0_iptc_5_xz.csv};
\addplot[color = red] table[x=z,y=x,col sep=comma]{data/t_0_axis_xz.csv};
\end{axis}
\end{tikzpicture}
\caption{Reference domain, $t=0$}
\label{fig:refDomain}
\end{subfigure}
\vspace{1em}

\begin{subfigure}[c]{.85\textwidth}
\begin{tikzpicture}
\begin{axis}[
ymin = -0.125,
ymax = 0.14,
	]
\addplot[color=black] table[x=z,y=y,col sep=comma]{data/t_0.5_iptc_2_yz.csv};
\addplot[color=black] table[x=z,y=y,col sep=comma]{data/t_0.5_iptc_5_yz.csv};
\addplot[color = red] table[x=z,y=y,col sep=comma]{data/t_0.5_axis_yz.csv};
\end{axis}
\end{tikzpicture}~
\begin{tikzpicture}
\begin{axis}[
ymin = -0.125,
ymax = 0.14,
	]
\addplot[color=black] table[x=z,y=x,col sep=comma]{data/t_0.5_iptc_2_xz.csv};
\addplot[color=black] table[x=z,y=x,col sep=comma]{data/t_0.5_iptc_5_xz.csv};
\addplot[color = red] table[x=z,y=x,col sep=comma]{data/t_0.5_axis_xz.csv};
\end{axis}
\end{tikzpicture}\vspace{-1em}
\caption{$t=0.5$}
\label{fig:Domain05}
\end{subfigure}
\vspace{1em}

\begin{subfigure}[c]{.85\textwidth}
\begin{tikzpicture}
\begin{axis}[
ymin = -0.14,
ymax = 0.17,
	]
\addplot[color=black] table[x=z,y=y,col sep=comma]{data/t_1_iptc_2_yz.csv};
\addplot[color=black] table[x=z,y=y,col sep=comma]{data/t_1_iptc_5_yz.csv};
\addplot[color = red] table[x=z,y=y,col sep=comma]{data/t_1_axis_yz.csv};

\end{axis}
\end{tikzpicture}~
\begin{tikzpicture}
\begin{axis}[
ymin = -0.14,
ymax = 0.17,
	]
\addplot[color=black] table[x=z,y=x,col sep=comma]{data/t_1_iptc_2_xz.csv};
\addplot[color=black] table[x=z,y=x,col sep=comma]{data/t_1_iptc_5_xz.csv};
\addplot[color = red] table[x=z,y=x,col sep=comma]{data/t_1_axis_xz.csv};

\end{axis}
\end{tikzpicture}\vspace{-1em}
\caption{$t=1$}
\label{fig:Domain1}
\end{subfigure}}
    \caption{Physical reference domain of the TESLA cavity (top) and domain according to a sampled deformation direction $\vec{\rand}$, ($t = \{0,.5,1\}$). On the left, the deformation in the $yz$-plane. On the right, the deformation in the $xz$-plane. The deformations are displayed increased by factor $250$. The red line indicates the deformation of the central axis.}
    \label{fig:deformation}
\end{figure}
In the following section, we show how the transformation can be translated into an eigenvalue problem involving coefficients.

\subsection{Maxwell's Eigenvalue Problem on Perturbed Domains}
We start by formulating the eigenvalue equation \eqref{eq:evp} on the perturbed domain~$\domain_\vt$ to obtain
\begin{subequations}\label{eq:MWErandomdomain}
\begin{align} 
    \curl (\curl \vec{E}_{\vt} ) &= {\lambda}_\vt \vec{E}_{\vt} && \text{in } \domain_\vt , \\
    \div \vec{E}_{\vt} &= 0 && \text{in } \domain_\vt , \\
    \vec{E}_{\vt} \times \vec{n}_{\vt} &= 0 && \text{on } \boundary_{\vt} , \\
    \big(\vec{E}_{\vt}, \vec{E}_{\vt}\big)_{\domain_\vt} &= 1 ,
\end{align}
\end{subequations}
with $\vec{n}_{\vt}$ being the outward pointing normal vector to $\domain_\vt$ from \cref{eq:newdomain}.
Going to the weak formulation, using the deformation mapping~\eqref{eq:deformation}, substitution rules, cf. \cite[Corollary 3.58]{monk_2003}, to pull the problem back to~$\domain_0$, and the isomorphy between $\HzerocurlDt$ and $\HzerocurlDzero$ \cite[Lemma 2.2.]{jerez-hanckes_2017}, we recover a strong formulation on the physical reference domain
\begin{subequations}\label{eq:MWErefdomain}
\begin{align}
    \curl ( \coeffstiff_{\vt} \curl \vec{E}_{0,\vt} ) 
    &= {\lambda}_{0,\vt} \coeffmass_{\vt} \vec{E}_{0,\vt} 
    && \text{in } \domain_0 , \\
    \div (\coeffmass_{\vt} \vec{E}_{0,\vt})  &= 0 
    && \text{in } \domain_0 , \\
    \vec{E}_{0,\vt}  \times \vec{n}_{0} &= 0 \label{eq:evp_coeff:normal}
    && \text{on } \boundary_0 , \\
    \big(\coeffmass_{\vt} \vec{E}_{0,\vt}, \vec{E}_{0,\vt}\big)_{\domain_0} &= 1,
\end{align}
with coefficients
\begin{align} 
    \coeffstiff_{\vt} &= \frac{1}{\det(\partial \maptil_{\vt})}\partial \maptil_{\vt}^\top \partial \maptil_{\vt} , \label{eq:defcoeffC}\\
    \coeffmass_{\vt}  &= \det(\partial \maptil_{\vt}) \partial \maptil_{\vt}^{-1} \partial \maptil_{\vt}^{-\top}\label{eq:defcoeffA},
\end{align}
\end{subequations}
and where $\partial$ denotes the derivative in the spatial coordinates $\vec{x}\in \domain_0$. We remark that an eigenvalue to \eqref{eq:MWErandomdomain} is also an eigenvalue to \eqref{eq:MWErefdomain} and vice versa, i.e., it holds $\lambda_\vt=\lambda_{0,\vt}$. Moreover, each eigenfunction of  \eqref{eq:MWErandomdomain} corresponds to an eigenfunction of \eqref{eq:MWErefdomain} and the correspondence is given through $\vec{E}_{0,\vt}(\vec{x}_0) = \vec{E}_{\vt}(\maptil_{\vt}(\vec{x}_0;\vec{\rand}))$ for all $\vec{x}_0\in\domain_0$ and $\vec{\rand}\in\mathbb{U}$.

This transformation allows us to analyze the deformations of the domain as variations of coefficients of the eigenvalue problem, allows to compare eigenpairs of different domain realizations, and enables us to compute derivatives. 
However, for a full characterization of the eigenpair derivatives, we need to have the derivatives of the domain transformation coefficients, i.e.\ $\coeffstiff_\vt$ and $\coeffmass_\vt$, to our disposal. These derivatives are available, as we assume $\partial \maptil_{\vt}$ to be smooth in  $\vt$, see \cref{eq:deformation,eq:maptilass}.

\subsection{Derivatives of Domain Transformation Coefficients}
For fixed $\vec{\rand}\in\mathbb{U}$, the first and higher directional derivatives of $\coeffstiff_{\vt}$ and $\coeffmass_{\vt}$ in $\vt$ have been computed in \cite{ziegler_computation_2022}. Referring to this article, the first directional derivatives are given as
\begin{subequations} \label{eq:ddt_CA}
\begin{align} \label{eq:ddt_C}
    \begin{aligned}
        \ddt \coeffstiff_{\vt} 
        ={}& - \tr \left(\left(\ddt \partial \maptil_{\vt}\right) \partial \maptil_{\vt}^{-1} \right) \coeffstiff_{t} 
        + \frac{1}{\det(\partial \maptil_{\vt})} \left(\ddt \partial \maptil_{\vt}\right)^\top \partial \maptil_{\vt}  \\
	    &\quad+ \frac{1}{\det(\partial \maptil_{\vt})}  \left( \left(\ddt \partial \maptil_{\vt}\right)^\top \partial \maptil_{\vt} \right)^\top  ,
    \end{aligned} 
\end{align}  
and
\begin{align}
    \begin{aligned} \label{eq:ddt_A}
        \ddt \coeffmass_{t} 
        ={}&  \tr \left( \left(\ddt \partial \maptil_{\vt} \right) \partial \maptil_{\vt}^{-1} \right) \coeffmass_{t} 
        - (\partial \maptil_{\vt})^{-1} \left(\ddt \partial \maptil_{\vt} \right) \coeffmass_{t} \\
        &\quad- \left( (\partial \maptil_{\vt})^{-1} \left(\ddt\partial \maptil_{\vt}\right) \coeffmass_{t} \right)^\top .
    \end{aligned}
\end{align}
\end{subequations}
To keep notation simple, for a given differentiable function $\mathbf{B}_{\vt}$, we note the derivative evaluated at the reference $t=0$ as
\[
\D_{\vt}\mathbf{B}_{\vt} = \ddt \mathbf{B}_{\vt}\bigg|_{\vt=0}.
\]
Thus, for any deformation parameter $\vec{\rand}\in \mathbb{U}$ and $t\in\R$, we can linearize $\coeffstiff_{\vt}$ and $\coeffmass_{\vt}$ as
\begin{align*}
    \coeffstiff_{\vt}
    &=\coeffstiff_{\vec{0}} + \vt (\D_{\vt}\coeffstiff_{\vt}) + \LandauO(\vt^2),\\
    \coeffmass_{\vt}
    &=\coeffmass_{\vec{0}} + \vt (\D_{\vt}\coeffmass_{\vt}) + \LandauO(\vt^2).
\end{align*}
As we will see in a moment, these linearizations are precisely the entities required to compute derivatives of eigenpairs.

\subsection{Derivatives of Eigenpairs}
The derivatives of eigenpairs of single multiplicity where originally characterized in \cite{Rel1940}. While 
the characterization for eigenpairs of higher multiplicity in \cite{Rel1940} was a nonlinear mapping,
\cite{DE2024} resolved this anomaly by deriving a linear characterization which is consistent to the case of
single multiplicity. Moreover, it was discussed that for perturbation theories
of eigenpairs to eigenvalues with multiplicity $m>1$ one must resort to a perturbation theory of the full
eigenspace, i.e., we need to consider the eigenvalue problem
\begin{subequations}\label{eq:MWErefdomainmultiple}
\begin{align}
    \curl ( \coeffstiff_{\vt} \curl \vec{E}_{0,\vt} ) 
    &= \coeffmass_{\vt} \vec{E}_{0,\vt} \vec{\lambda}_{0,\vt}
    && \text{in } \domain_0 , \\
    \div (\coeffmass_{\vt} \vec{E}_{0,\vt})  &= 0 
    && \text{in } \domain_0 , \\
    \vec{E}_{0,\vt}  \times \vec{n}_{0} &= 0
    && \text{on } \boundary_0 , \\
    \big(\coeffmass_{\vt} \vec{E}_{0,\vt}, \vec{E}_{0,\vt}\big)_{\domain_0} &= \vec{I},
\end{align}
\end{subequations}
where by a slight abuse of notation we write $\vec{E}_{0,t}=\big[[\vec{E}_{0,t}]_1,\ldots,[\vec{E}_{0,t}]_m\big]\in\big[\HzerocurlDzero\big]^m$ and $\vec{\lambda}_{0,t}\in\R^{m\times m}$. We remark that, for $m=1$, \eqref{eq:MWErefdomainmultiple} clearly coincides with the original problem \eqref{eq:MWErefdomain}.
Moreover, it is easy to check that, for fixed $t$ and $\vec{\rand}$, any solution to \eqref{eq:MWErefdomainmultiple} corresponds to exactly $m$ solutions to \eqref{eq:MWErefdomain} and vice versa, see, e.g., \cite[Lemma 2.3]{DE2024} for a proof.

In the following, we will assume that $\lambda_0=\lambda_{0,0}$ is an eigenvalue of multiplicity $m$ with corresponding orthogonal eigenfunctions $\vec{E}_{0,t}=\big[[\vec{E}_{0,t}]_1,\ldots,[\vec{E}_{0,t}]_m\big]\in\big[\HzerocurlDzero\big]^m$. Thus, \eqref{eq:MWErefdomainmultiple} holds with $\vec{\lambda}_{0,0}=\lambda_{0,0}\vec{I}$. 
Following \cite{DE2024}, the derivative of $(\vec{\lambda}_{0,\vt},\vec{E}_{0,t})$ in $t$ at $t=0$ is given as the solution of $m$ saddle point problems
\begin{subequations}\label{eq:derivative}
\begin{align}
    \curl (\coeffstiff_0 \curl \D_\vt [\vec{E}_{0,\vt}]_i)
    &- \coeffmass_0 (\D_\vt [\vec{E}_{0,\vt}]_i){\lambda}_{0,0}
    - \sum_{j=1}^m\coeffmass_0 [\vec{E}_{0,0}]_j ([\D_\vt \vec{\lambda}_{\vt}]_{ji})\label{eq:der_eigpair:main1}\\
    &= - (\D_\vt \coeffstiff_{\vt}) [\vec{E}_{0,0}]_i + (\D_\vt \coeffmass_{\vt}) [\vec{E}_{0,0}]_i {\lambda}_{0,0}
    && \text{in } \domain_0, \label{eq:der_eigpair:main2} \\
    - \div \big(\coeffmass_0 (\D_\vt[\vec{E}_{0,\vt}]_i)\big) &= \div \big((\D_\vt\coeffmass_{\vt} )[\vec{E}_{0,0}]_i\big)
    && \text{in } \domain_0, \label{eq:der_eigpair:div} \\
    (\D_\vt [\vec{E}_{0,\vt}]_i) \times \vec{n}_0 
    &= 0 
    && \text{on } \boundary_0, \label{eq:der_eigpair:boundarynormal}\\
    - \Big( \coeffmass_0 (\D_\vt [\vec{E}_{0,\vt}]_i), [\vec{E}_{0,0}]_j\Big)_{\domain_0}
    &= \frac{1}{2}\Big( (\D_\vt \coeffmass_{\vt} ) [\vec{E}_{0,0}]_i, [\vec{E}_{0,0}]_i\Big)_{\domain_0} \label{eq:der_eigpair:norm} \\
    - \Big( \coeffmass_0 (\D_\vt [\vec{E}_{0,\vt}]_i), [\vec{E}_{0,0}]_j \Big)_{\domain_0}
    &= 0 \label{eq:der_eigpair:orth}
    &&\hspace{-2cm}i\neq j;\ j=1,\ldots,m,
\end{align} \label{eq:der_eigpair}
\end{subequations}%
for $i=1,\ldots,m$ and where we abbreviate $\D_\vt\vec{E}_{0,\vt}=\big[\D_\vt [\vec{E}_{\vt}]_1,\ldots,\D_\vt [\vec{E}_{\vt}]_m\big]$.
The system of equations is of saddle point form, thus unique solvability can be shown using a Ladyzhenskaya-Babuška-Brezzi (LBB)-condition, see \cite{DE2024} for a proof.
We will refer to \eqref{eq:der_eigpair:main1} and \eqref{eq:der_eigpair:main2} as the \emph{main condition}, \eqref{eq:der_eigpair:norm} as the \emph{normalization constraint}, and \eqref{eq:der_eigpair:orth} as the \emph{orthogonality constraint}. 
In its discretization, \eqref{eq:der_eigpair:norm} and \eqref{eq:der_eigpair:orth} are combined into a matrix-valued \emph{orthonormality constraint}.
\eqref{eq:der_eigpair:div} is the \emph{divergence constraint} and \eqref{eq:der_eigpair:boundarynormal} is the \emph{boundary normal constraint}. To shorten notation, we may define
\begin{equation}\label{eq:fancydifferential}
\vec{\mathfrak{A}}
=
\begin{bmatrix}
\curl (\coeffstiff_0 \curl \cdot)-{\lambda}_{0,0}\coeffmass_0 &
-\coeffmass_0 \vec{E}_{0,0}\\
- \div \big(\coeffmass_0\cdot\big)& 0\\
\cdot \times \vec{n}_0 & 0\\
- \big( \coeffmass_0 \cdot, \vec{E}_{0,0}\big)_{\domain_0}& 0
\end{bmatrix}
\end{equation}
to abbreviate \eqref{eq:derivative} as
\begin{equation}\label{eq:shortderivative}
\vec{\mathfrak{A}}
\begin{bmatrix}
\D_\vt\vec{E}_{0,\vt}\\
\D_\vt\vec{\lambda}_{0,\vt}
\end{bmatrix}
=
\begin{bmatrix}
 - (\D_\vt \coeffstiff_{\vt}) \vec{E}_{0,0} + (\D_\vt \coeffmass_{\vt}) \vec{E}_{0,0} {\lambda}_{0,0}\\
 \div (\D_\vt\coeffmass_{\vt} \vec{E}_{0,0})\\
 0\\
 \diag \limits_{i=1,\ldots,m}\frac{1}{2}\Big( (\D_\vt \coeffmass_{\vt}) [\vec{E}_0]_i, [\vec{E}_0]_i \Big)_{\domain_0}
\end{bmatrix}.
\end{equation}

Since $(\vec{\lambda}_{0,\vt},\vec{E}_{0,\vt})$ depends analytically on $\coeffstiff_\vt$ and $\coeffmass_\vt$, see \cite{DE2024}, the assumptions on the deformation field yield that $(\vec{\lambda}_{0,\vt},\vec{E}_{0,\vt})$ satisfies the expansions
\begin{subequations}
\begin{align*}
    \vec{E}_{0,\vt} &= \vec{E}_{0,0} + t (\D_\vt \vec{E}_{0,\vt}) + \LandauO(t^2),\\
    \vec{\lambda}_{0,\vt} &= \vec{\lambda}_{0,0} + t (\D_\vt \vec{\lambda}_{0,\vt}) + \LandauO(t^2).
\end{align*}
\end{subequations}
Most notably, $(\vec{\lambda}_{0,\vt},\vec{E}_{0,\vt})$ satisfies \eqref{eq:MWErefdomainmultiple} in a neighbourhood of $t=0$, with $\vec{\lambda}_{0,\vt}$ not necessarily being diagonal for $t\neq 0$.

In the next section, we shall consider a stochastic model of the geometric deformations. 
We can then use the derivatives of the eigenpair to approximate the mean and covariance of the eigenpair.

\section{Shape Uncertainty Quantification of Eigenvalues and -modes} \label{sec:UQ}

\subsection{Maxwell's Eigenvalue Problem on Randomly Deformed Shapes}

We now go on to consider the case where the deformations are realizations of a stochastic model.
To this end, let $(\samplespace,\eventspace,\probmeasure)$ be a probability space and let $\vec{\rand}(\omega)\in\mathbb{U}$, $\omega\in\Omega$, be a random variable. In complete analogy to \eqref{eq:newdomain}, $\domain_{\vt}(\omega)=\maptil_{\vt}(\domain_0,\vec{\rand}(\omega))$ defines a randomly perturbed domain. Thus, we obtain a new eigenvalue problem, i.e., find $(\lambda_t(\omega),\vec{E}_t(\omega)) \in (\R \times \HzerocurlDtomega)$, such that
\begin{subequations}
\begin{align*} 
    \curl (\curl \vec{E}_\vt(\omega) ) &= {\lambda}_\vt(\omega) \vec{E}_\vt(\omega) && \text{in } \domain_\vt(\omega) , \\
    \div \vec{E}_\vt(\omega) &= 0 && \text{in } \domain_\vt(\omega) , \\
    \vec{E}_\vt(\omega) \times \vec{n}_\vt(\omega) &= 0 && \text{on } \boundary_\vt(\omega) , \\
    \big( \vec{E}_\vt(\omega), \vec{E}_\vt(\omega)\big)_{\domain_\vt(\omega)} &= 1.
\end{align*}
\end{subequations}
As before, the new eigenproblem can be pulled back to the physical reference domain, where the task is to find $(\lambda_{0,\vt}(\omega),\vec{E}_{0,\vt}(\omega)) \in (\R \times \HzerocurlDzero)$, such that
\begin{subequations}\label{eq:randeigref}
\begin{align} 
    \curl \big(\coeffstiff_{\vt}(\omega) \curl \vec{E}_{0,\vt}(\omega) \big) &= {\lambda}_{0,\vt}(\omega) \coeffmass_{\vt}(\omega) \vec{E}_{0,\vt}(\omega) && \text{in } \domain_0 , \\
    \div \coeffmass_{\vt}(\omega) \vec{E}_{0,\vt}(\omega) &= 0 && \text{in } \domain_0 , \\
    \vec{E}_{0,\vt}(\omega) \times \vec{n}_0 &= 0 && \text{on } \boundary_0 , \\
    \big( \coeffmass_{\vt}(\omega) \vec{E}_{0,\vt}(\omega), \vec{E}_{0,\vt}(\omega)\big)_{\domain_0} &= 1.
\end{align}
\end{subequations}
Following \eqref{eq:shortderivative}, the derivatives are characterized by
\begin{equation}\label{eq:randderivative}
\vec{\mathfrak{A}}
\begin{bmatrix}
\D_\vt\vec{E}_{0,\vt}(\omega)\\
\D_\vt\vec{\lambda}_{0,\vt}(\omega)
\end{bmatrix}
=
\begin{bmatrix}
 - \big(\D_\vt \coeffstiff_{\vt}(\omega)\big) \vec{E}_{0,0} + \big(\D_\vt \coeffmass_{\vt}(\omega)\big) \vec{E}_{0,0} {\lambda}_{0,0}\\
 \div \big(\D_\vt\coeffmass_{\vt}(\omega) \vec{E}_{0,0}\big)\\
 0\\
 \diag \limits_{i=1,\ldots,m}\frac{1}{2}\Big(\big(\D_\vt \coeffmass_{\vt}(\omega)\big) [\vec{E}_{0,0}]_i, [\vec{E}_{0,0}]_i \Big)_{\domain_0}
\end{bmatrix},
\end{equation}
with the very same $\vec{\mathfrak{A}}$ as in \eqref{eq:fancydifferential}. We note that only the right-hand side and the derivatives depend on the random parameter $\omega$, but not the differential operator $\vec{\mathfrak{A}}$.
Thus, we arrive at approximations
\begin{align*}
    \vec{\lambda}_{0,\vt}(\omega) 
    &= \vec{\lambda}_{0,0} 
    + t (\D_\vt \vec{\lambda}_{0,\vt})(\omega) 
    + \LandauO(t^2;\omega) 
    && \event \in \samplespace , \\
    \vec{E}_{0,\vt}(\omega) 
    &= \vec{E}_{0,0}
    + t(\D_\vt \vec{E}_{0,\vt})(\omega) 
    + \LandauO(t^2;\omega) 
    && \event \in \samplespace,
\end{align*}
where $\LandauO(\cdot;\omega)$ denotes the usual Landau notation with the involved constant possibly depends on $\omega\in \samplespace$.

\subsection{Expansions for Expected Value and Covariance}

In the following, we will use the derived expansions to approximate the mean and covariance of eigenpairs.
Our goal is to find mean and covariance of $\vec{u} \in \{\vec{\lambda}_{0,\vt},\vec{E}_{0,\vt}\}$, i.e.,
\begin{align}
    &\E[\vec{u}](\vec{x}) = \int \limits_{\samplespace} \vec{u}(\vec{x}) \dif \probmeasure , 
    && \Cov[\vec{u}](\vec{x},\vec{y}) 
    = \int \limits_{\samplespace} 
    \big( \vec{u}(\vec{x}) - \E[\vec{u}](\vec{x}) \big) \otimes \big( \vec{u}(\vec{y}) - \E[\vec{u}](\vec{y}) \big) \dif \probmeasure,
\label{eq:E_Cov}
\end{align}
given a stochastic model of geometric deformations. 
To this end, the following preliminary lemma on the derivatives of the domain transformation coefficients is helpful.
\begin{lemma} \label{lem:GtoCA}
    Assume that $\D_\vt \maptil_{\vt}(\vec{x};\vec{\rand})$ is linear in $\vec{\rand}$ and $\vec{\rand}$ is centered, i.e. $\E[\vec{\rand}] = \vec{0}$, then the mean of the derivatives of the deformation coefficients evaluated at the reference $t=0$ is also centered, i.e.,
    \begin{align*}
        \E[\D_\vt \coeffstiff_{\vt}] = \E[\D_\vt \coeffmass_{\vt}] = 0 .
    \end{align*}
\end{lemma}
\begin{proof}
    We first note that the assumptions \eqref{eq:maptilass} on the domain transformations, i.e., $\maptil_0(\vec{x};\vec{z})\!=\! \vec{x}$ yield $\partial\maptil_0(\vec{x};\vec{z}) = \vec{I}$ and $\coeffstiff_0, \coeffmass_0 = \vec{I}$.
    So when considering \eqref{eq:ddt_CA} at the reference $\vt=0$, we get
    \begin{subequations}\label{eq:ddt_CA0}
    \begin{align}
        \D_\vt \coeffstiff_{\vt} 
        &= - \tr \left(\left(\D_\vt \partial \maptil_{\vt}\right) \vec{I} \right) \vec{I} 
        + \frac{1}{\det(\vec{I})} \left(\D_\vt \partial \maptil_{\vt}\right)^\top \vec{I} 
        + \frac{1}{\det(\vec{I})} \left( \left(\D_\vt \partial \maptil_{\vt}\right)^\top \vec{I} \right)^\top\\
        &=
        -\tr\left(\D_\vt\partial\maptil_{\vt}\right)\vec{I}
        +\left(\D_\vt \partial \maptil_{\vt}\right)^\top
        +\D_\vt \partial \maptil_{\vt},
    \end{align}
    and
    \begin{align}
        \D_\vt \coeffmass_{t} 
        &= \tr \left( \left(\D_\vt \partial \maptil_{\vt} \right) \vec{I} \right) \vec{I} 
        - \vec{I} \left(\D_\vt \partial \maptil_{\vt} \right) \vec{I} 
        - \left( \vec{I} \left(\D_\vt\partial \maptil_{\vt}\right) \vec{I} \right)^\top\\
        &=
        \tr\left(\D_\vt\partial\maptil_{\vt}\right)\vec{I} 
        - \D_\vt \partial \maptil_{\vt}
        - \left(\D_\vt \partial \maptil_{\vt}\right)^\top\\
        &=
        - \D_\vt \coeffstiff_{\vt},
    \end{align}
    \end{subequations}
    implying
    \begin{align*}
        \E[\D_\vt \coeffstiff_{\vt}] 
        = - \tr \left(\E[\D_\vt \partial \maptil_{\vt}] \right) 
        + \left(\E[\D_\vt \partial \maptil_{\vt}]\right)^\top  
        + \E[\D_\vt \partial \maptil_{\vt}]
        =
        - \E[\D_\vt \coeffmass_{t}].
    \end{align*}
    Moreover, $\maptil_{\vt}(\vec{x};\vec{0}) =\vec{x}$ from the assumptions \eqref{eq:maptilass} and total differentiability imply
    \begin{align*}
        \D_\vt \partial \maptil_{\vt} 
        = \partial \D_\vt \maptil_{\vt}
        = \partial( \D_\vt \maptil_{\vt} |_{\rand=0})
        + \partial (D_{\vec{\rand}} \D_\vt \maptil_{\vt}) \cdot \vec{\rand}
        = \partial (D_{\vec{\rand}} \D_\vt \maptil_{\vt}) \cdot \vec{\rand}.
    \end{align*}
    Taking the mean yields $\E[\D_{\vt}\partial \maptil_{\vt}] = \partial (\D_{\vec{\rand}} \D_\vt \maptil_{\vt}) \cdot \E[\vec{\rand}] = 0$ and thus the assertion.
\end{proof}

\begin{lemma} \label{lem:excov}
    Given the assumptions of \cref{lem:GtoCA} and that the covariance $\Cov[\partial \maptil_\vt(\vec{x};\cdot)]$ exists, the following moments are finite 
    \begin{align*}
        \E[\D_{\vt} \coeffstiff_\vt \otimes \D_{\vt} \coeffstiff_\vt], \  
        \E[\D_{\vt} \coeffstiff_\vt \otimes \D_{\vt} \coeffmass_\vt], \ 
        \E[\D_{\vt} \coeffmass_\vt \otimes \D_{\vt} \coeffmass_\vt] < \infty .
    \end{align*}
\end{lemma}
\begin{proof}
    We find that $\D_{\vt} \coeffstiff_\vt, \D_{\vt} \coeffmass_\vt$ are both linear in $\partial \maptil_\vt$, cf. \eqref{eq:ddt_CA0}, thus the above moments are finite.
\end{proof}

\begin{lemma} \label{lem:Gtoeig}
Let $\lambda_{0,0}$ be an eigenvalue of multiplicity $m$ of \eqref{eq:randeigref} at $t=0$ with orthonormal eigenbasis $\vec{E}_{0,0}$.
Assume that $\D_\vt \maptil_{\vt}(\vec{x};\vec{\rand})$ is linear in $\vec{\rand}$ and $\vec{\rand}$ is centered, i.e. $\E[\vec{\rand}] = \vec{0}$ and assume that $\Cov[\partial \maptil_\vt(\vec{x};\vec{z})]$ exists.
Let $({\vt},\vec{\rand}) \mapsto (\vec{E}_{0,\vt},\vec{\lambda}_{0,\vt})$ be the unique local, analytic trajectory for which it holds $(\vec{E}_{0,\vt},\vec{\lambda}_{0,\vt}) = (\vec{E}_{0,0},\lambda_{0,0} \vec{I})$ at $\vt=0$. 
Then it holds
\begin{subequations} \label{eq:series_stoch}
\begin{align}
    \E[\vec{\lambda}_{0,\vt}] &= \vec{\lambda}_{0,0} + \LandauO(t^2) , \label{eq:series_stoch:lexp}\\
    \E[\vec{E}_{0,\vt}] &= \vec{E}_{0,0} + \LandauO(t^2) , \label{eq:series_stoch:uexp}\\
    \Cov[\vec{\lambda}_{0,\vt}] &= t^2 \Cov\Big[ \D_\vt \vec{\lambda}_{0,\vt} \Big] + \LandauO(t^3) , \label{eq:series_stoch:lcov}\\
    \Cov[\vec{E}_{0,\vt}] &= t^2 \Cov\Big[ \D_\vt \vec{E}_{0,\vt} \Big] + \LandauO(t^3) \label{eq:series_stoch:ucov}.
\end{align}
\end{subequations}
\end{lemma}
\begin{proof}
Lemma~\ref{lem:GtoCA} yields that the first derivatives of the coefficients are centered, i.e., it holds
\begin{align*}
    \mathbb{E}[\D_\vt \coeffstiff_t] = 0 , \qquad
    \mathbb{E}[\D_\vt \coeffmass_t] = 0.
\end{align*}
The rest of the proof follows \cite[Theorem 3.4]{DE2024} with only minor modifications.
\end{proof}

\begin{remark}\label{lem:order4}
If we additionally assume $\mathbb{E}[\vec{\rand} \otimes \vec{\rand} \otimes \vec{\rand}] = \vec{0}$, one can show along the lines of \cite[Lemma 2.3]{HP2018} that the residue term for the covariances improves to $\mathcal{O}(t^4)$.%
\end{remark}

The approximations of the covariances in \eqref{eq:series_stoch:lcov}, \eqref{eq:series_stoch:ucov} involve the covariances of the first derivatives of the eigenpair. 
To this end, those derivatives can be characterized by tensorizing \eqref{eq:randderivative} and taking the mean to obtain
\begin{equation}
    \begin{aligned}
    \big(\vec{\mathfrak{A}}\otimes\vec{\mathfrak{A}}\big)
    \Cov
    \begin{bmatrix}
    \delta_\vt \vec{E}_{0,\vt} \\ \D_\vt \vec{\lambda}_{0,\vt}
    \end{bmatrix}
    =
    \Cov
    \begin{bmatrix}
    - (\D_\vt \coeffstiff_{\vt}) \vec{E}_{0,0} + (\D_\vt \coeffmass_{\vt}) \vec{E}_{0,0} \lambda_0 \\
    \div (\D_\vt \coeffmass_{\vt}) \vec{E}_{0,0} \\
    0 \\
    \diag \limits_{i=1,\ldots,m}\frac{1}{2}\Big( (\D_\vt \coeffmass_{\vt}) [\vec{E}_{0,0}]_i, [\vec{E}_{0,0}]_i \Big)_{\domain_0}
    \end{bmatrix},
    \end{aligned} \label{eq:cov_eq}
\end{equation}
where the right-hand side exists due to \cref{lem:excov}.
We note that, in contrast to \cite{DE2024}, $\D_\vt \coeffstiff_{\vt}$ and $\D_\vt \coeffmass_{\vt}$ are correlated, as they both depend on the random deformation field $\maptil_{\vt}$, cf.~\eqref{eq:ddt_CA}. Also, we note that the numerical solution of \eqref{eq:cov_eq} is well understood nowadays and we will comment on it later.

\section{Discretization} \label{sec:disc}

\subsection{Conforming Discretization of Maxwell's Eigenvalue Problem}
Starting from the problem formulation~\eqref{eq:MWErefdomain}, let us derive the discrete Maxwell eigenvalue problem. To this end, we pursue a Ritz-Galerkin discretization, but any other discretization method with similar properties could also be used. The weak formulation of the Maxwell eigenvalue problem reads: find $\lambda \in \mathbb{R}$ and $\vec{E}_{0,\vt} \in \HzerocurlDzero$, such that
\begin{equation}\label{eq:randeigref:disc}
    \Big(\coeffstiff_t(\omega)\curl  \vec{E}_{0,\vt}(\omega), \curl \vec{\fembasis}\Big)_{\domain_0} = {\lambda}_{0,t}(\omega)\Big(\coeffmass_t(\omega)\vec{E}_{0,\vt}(\omega),\vec{\fembasis}\Big)_{\domain_0} \quad \forall\vec{\fembasis} \in \HzerocurlDzero.
\end{equation}
Choosing an $N$-dimensional subspace $\vec{W}\subset\HzerocurlDzero$, $\vec{W}=\operatorname{span}_{i=1,\ldots,N}\{\vec{\fembasis}_i\}$, for discretization of the electromagnetic field, a Galerkin discretization of \eqref{eq:randeigref:disc} yields a discrete eigenvalue problem
\begin{equation}\label{eq:disceigprob}
\stiff\big(\coeffstiff_\vt(\omega)\big)\vec{e}_{0,\vt}(\omega)={\lambda}_{0,\vt}^h(\omega)\mass\big(\coeffmass_\vt(\omega)\big)\vec{e}_{0,\vt}(\omega),
\end{equation}
where $\vec{e}_{0,\vt}(\omega)\in\mathbb{R}^N$,
\begin{equation}
    \vec{E}_{0,\vt}^h(\omega) = \sum_j \big[\vec{e}_{0,\vt}(\omega)\big]_{j} \vec{\fembasis}_j\in\vec{W},
    \label{eq:weak_evp}
\end{equation}
and
\begin{align*}
    \stiff\big(\coeffstiff_\vt(\omega)\big)
    =& \Big[\Big(\coeffstiff_{\vt}(\omega) \curl \vec{\fembasis}_j, \curl \vec{\fembasis}_i\Big)_{\domain_0}\Big]_{i,j=1}^N,\\
    \mass\big(\coeffmass_\vt(\omega)\big)
    =& \Big[\Big(\coeffmass_{\vt}(\omega)  \vec{\fembasis}_j,  \vec{\fembasis}_i\Big)_{\domain_0}\Big]_{i,j=1}^N.
\end{align*}
\begin{remark}
The choice of appropriate subspaces is well understood, see, e.g., \cite{monk_2003}. For our numerical examples, we will choose the isogeometric analysis discretization using B-spline spaces, which we will introduce in \cref{sec:computation}.
\end{remark}

\subsection{Discretization of the Derivatives of Maxwell's Eigenvalue Problem}

Given parameterized stiffness matrices $\stiff(\coeffstiff_{\vt})$ and mass matrices $\mass(\coeffmass_{\vt})$, the discrete form of the analytic derivatives \eqref{eq:der_eigpair} reads
\begin{align}
    \begin{bmatrix}
    \stiff(\coeffstiff_0) - {\lambda}_{0,0}^h \mass(\coeffmass_0) & -\mass(\coeffmass_0) \vec{e}_{0,0} \\
    -\vec{e}_{0,0}^\top \mass^\top (\coeffmass_0) & 
    \end{bmatrix}
    \begin{bmatrix}
    \D_\vt \vec{e}_{0,\vt} \\ \D_\vt \vec{\lambda}^h_{0,\vt}
    \end{bmatrix} 
    = 
    \begin{bmatrix}
    -\stiff(\D_\vt\coeffstiff_{\vt}) \vec{e}_{0,\vt} + {\lambda}_{0,0}^h \mass(\D_\vt\coeffmass_{\vt})\vec{e}_{0,\vt} \\
    \diag \limits_{i=1,\ldots,m} \frac{[\vec{e}_{0,\vt}]_i^\top \mass(\D_\vt\coeffmass_{\vt}) [\vec{e}_{0,\vt}]_i}{2}
    \end{bmatrix}. \label{eq:der_eigpair:disc}
\end{align}
As discussed in \cite{DE2024}, up to a consistency error which needs to be analysed, it is equivalent to first consider the derivatives of the continuous problem \eqref{eq:MWErefdomain} and discretize the saddle point problem or to consider the derivatives of \eqref{eq:disceigprob}. Unique solvability of this system of equations can be shown in complete analogy to the continuous case. We also note that the system matrix in \eqref{eq:der_eigpair:disc} is symmetric.

\subsection{Discretization of the Covariance Equation}
A discrete version of the covariance equation \eqref{eq:cov_eq} can be obtained by employing a Galerkin discretization of tensor product finite element spaces or, leading to the very same results, by tensorizing \eqref{eq:der_eigpair:disc} and taking the mean, to obtain
\begin{align}
\begin{aligned}
    \begin{bmatrix}
    \stiff(\coeffstiff_0) - {\lambda}^h_{0,0} \mass(\coeffmass_0) & -\mass(\coeffmass_0) \vec{e}_{0,0} \\
    -\vec{e}_{0,0}^\top \mass^\top(\coeffmass_0) & 
    \end{bmatrix}
    \Cov
    \begin{bmatrix}
    \D_\vt \vec{e}_{0,\vt} \\ \D_\vt \vec{\lambda}^h_{0,\vt}
    \end{bmatrix}
    &
    \begin{bmatrix}
    \stiff(\coeffstiff_0) - {\lambda}^h_{0,0} \mass(\coeffmass_0) & -\mass(\coeffmass_0) \vec{e}_{0,0} \\
    -\vec{e}_{0,0}^\top \mass^\top(\coeffmass_0) & 
    \end{bmatrix}\\
    &=
    \Cov
    \begin{bmatrix}
    -\stiff(\D_\vt \coeffstiff_{\vt}) \vec{e}_{0,0} + {\lambda}^h_{0,0} \mass(\D_\vt\coeffmass_{\vt})\vec{e}_{0,0} \\
    \diag \limits_{i=1,\ldots,m} \frac{[\vec{e}_{0,0}]_i^\top \mass(\D_\vt \coeffmass_{\vt}) [\vec{e}_{0,0}]_i}{2}
    \end{bmatrix} .
\end{aligned} \label{eq:cov_eq:disc}
\end{align}
Unfortunately, the solution of such a system of equations can be computationally infeasible since covariance matrices are generally dense, and their storage requirements grow prohibitively. To this end, the solution of such covariance equations has been the topic of many articles, see \cite{BG04b,DHS17,HPS12,HSS2008,vPS06} to mention a few. 
To keep the implementation simple, we outline in the next section a strategy which is based on a (possibly approximate) Karhunen-Lo\`eve type representation of the deformation field.

\section{Computational Considerations}\label{sec:computation}
\subsection{Closed-form Representations for Karhunen-Loève type Deformation Fields}\label{sec:KL}
In order to compute the derivatives of the deformation coefficients efficiently, we follow the domain mapping approach to model the transformation $\maptil_{\vt}$ in the form of 
\begin{align}\label{eq:specificdeformationfield}
    \maptil_{\vt}(\vec{x};\vec{\rand}) =  \vec{x} + t \vec{V}(\vec{x};\vec{\rand}) ,
\end{align}    
where $\vec{x} \in \domain_0$, with a sufficiently smooth deformation field $\vec{V}$ depending on $\vec{\rand}$.
For the purpose of this article, we follow a Karhunen-Lo\`eve type approach, i.e., we assume that the deformation field can be decomposed as
\begin{align} \label{eq:KLE:G}
    \vec{V}(\vec{x};\vec{z}) = \sum \limits_{i=1}^{\nKL} \rand_i \vec{V}_i(\vec{x})
\end{align}
and $\rand_i$ uniformly distributed i.i.d.\ random variables on $[-1,1]$. 
For $M=\infty$ we assume that the $\vec{V}_i$ are such that the sum converges absolutely in $C^{1}$, yielding a continuously differentiable vector field. This implies existence of its covariance function and that the deformation field $\maptil_\vt$ is continuously differentiable in $\vec{x}$ with $\partial \maptil_\vt$ smooth in $\vt$ as required in \cref{eq:deformation,eq:maptilass}. Let us remark that weaker but more technical assumptions can be made as long as all of the required derivatives exist.

The following lemma gives closed-form expressions for the parametric derivative of the deformation coefficients for this case, which happen to be again of Karhunen-Lo\`eve type.
\begin{lemma}
Let $\maptil_{\vt}$ be given through \eqref{eq:specificdeformationfield} and \eqref{eq:KLE:G}. Then the parametric derivatives of the deformation coefficients \eqref{eq:defcoeffC} and \eqref{eq:defcoeffA} are given through
\begin{equation*}
\D_\vt \coeffstiff_{\vt} 
=
-
\sum \limits_{i=1}^{\nKL} \rand_i \Big[ \D_\vt \coeff_{\vt} \Big]_i,
\qquad
\D_\vt \coeffmass_{\vt} 
=
\sum \limits_{i=1}^{\nKL} \rand_i \Big[ \D_\vt \coeff_{\vt} \Big]_i,
\end{equation*}
where
\begin{equation}\label{eq:B}
\Big[ \D_\vt \coeff_{\vt} \Big]_i
=
\tr(\partial \vec{V}_i ) 
    - \partial \vec{V}_i 
    - \partial \vec{V}_i ^\top.
\end{equation}
\end{lemma}
\begin{proof}
The specific form of $\maptil_{\vt}$ due to \eqref{eq:specificdeformationfield} and \eqref{eq:KLE:G} yields
\begin{equation*}
\D_\vt \partial \maptil_{\vt}
=
\sum \limits_{i=1}^{\nKL} \rand_i\partial \vec{V}_i,
\end{equation*}
which, inserted into \cref{eq:ddt_CA0} yields the assertion.
\end{proof}
Considering system of equations \eqref{eq:der_eigpair} we also observe that $\D_\vt \coeffstiff, \D_\vt \coeffmass$ are only plugged in on the right-hand side of the equation. Since \eqref{eq:der_eigpair} is linear, by plugging in $[\D_\vt \coeffstiff]_i, [\D_\vt \coeffmass]_i$ instead, we obtain derivatives $\big([\D_\vt \vec{\lambda}_{0,\vt}]_i,[\D_\vt \vec{E}_{0,\vt}]_i\big)$ and find that the first derivative of the eigenpair is of Karhunen-Lo\`eve type as well, i.e.,
\begin{align*}
    \D_\vt \vec{\lambda}_{0,\vt} 
    = \sum \limits_{i=1}^{\nKL} \rand_i \Big[\D_\vt \vec{\lambda}_{0,\vt}\Big]_i ,\qquad
    \D_\vt \vec{E}_{0,\vt} 
    = \sum \limits_{i=1}^{\nKL} \rand_i \Big[\D_\vt \vec{E}_{0,\vt}\Big]_i .
\end{align*}
Since the $z_i$ are stochastically independent, considering the expected value of the tensor product of the derivatives with themselves, we get a deterministic series representation
\begin{align}
    \Cov
    \begin{bmatrix}
    \D_\vt \vec{\lambda}_{0,\vt} \\ \D_{\vt} \vec{E}_{0,\vt}
    \end{bmatrix}
    &= \frac{1}{3}\sum \limits_{i=1}^{\nKL} 
    \begin{bmatrix}
    \D_\vt \vec{\lambda}_{0,\vt} \\ \D_\vt \vec{E}_{0,\vt}
    \end{bmatrix}_i
    \otimes
    \begin{bmatrix}
    \D_\vt \vec{\lambda}_{0,\vt} \\ \D_\vt \vec{E}_{0,\vt}
    \end{bmatrix}_i,\label{eq:KLE:tensor_lu}
\end{align}
i.e., we have a solution representation for \eqref{eq:cov_eq}. This way of construction transfers to the discrete case such that we have an explicit way to solve the covariance equation \eqref{eq:cov_eq:disc}. Since we observe in practice $M\ll\infty$ without significant loss in accuracy this solution approach is also efficient, but requires the computation of the $[\D_{\vt}\coeff_{\vt}]_i$-terms in the discrete setting.

\subsection{Geometry Description}
A common challenge when dealing with random domains is that the geometry perturbations need to be realized, e.g. for sampling.
This is not the case in our derivative-based approach. 
Here, the challenge is to determine the deformation coefficients, for which we need to compute the $[\D_{\vt}\coeff_{\vt}]_i$-terms given in \eqref{eq:B}. 
Only to validate our approach, we compare with a surrogate model which requires actual geometry deformations.
The framework of isogeometric analysis (IGA) \cite{Hughes_2005} has been successfully used for geometry deformations in the past for this purpose, see, e.g., \cite{DHJM2022,georg_uncertainty_2019,ziegler_mode_2022}, so we will also follow it here. 
To this end, for the geometry perturbations we use the same parametric framework as also used for the representation of geometries in computer-aided design.

For the geometry description itself, we start with the one-dimensional B-spline spaces of piecewise polynomial degree $p\geq 0$. For $k>p$, a \emph{locally quasi uniform $p$-open knot vector} is a tuple
$
\Xi = \big[{\xi_0 = \cdots =\xi_{p}}\leq
\cdots \leq{\xi_{k}=\cdots =\xi_{k+p}}\big]\in[0,1]^{k+p+1}
$
with $\xi_0 = 0$ and $\xi_{k+p}=1$
such that there exists a constant $\theta\geq 1$ with
 $\theta^{-1}\leq h_j\cdot h_{j+1}^{-1} \leq \theta$
 for all $p\leq j < k$, 
  where  $h_j=  \xi_{j+1}-\xi_{j}$.
The B-spline basis $ \lbrace B_{j,p} \rbrace_{0\leq j< k}$ is then recursively defined according to
\begin{align*}
B_{j,p}(x) & =\begin{cases}
\mathbbm{1}_{[\xi_j,\xi_{j+1})}&\text{ if }p=0,\\
\frac{x-\xi_j}{\xi_{j+p}-\xi_j}b_j^{p-1}(x) +\frac{\xi_{j+p+1}-x}{\xi_{j+p+1}-\xi_{j+1}}b_{j+1}^{p-1}(x) & \text{ else,}
\end{cases}
\end{align*}
where $\mathbbm{1}_A$ refers to the indicator function of the set $A$. 
For further concepts and algorithmic realization of B-splines, we refer to \cite{de-Boor_2001aa} and the references therein.

We assume the usual isogeometric setting for the domain representation of the physical reference domain, i.e., denoting $\square=[0,1]^3$, we assume that the reference domain can be decomposed into several smooth \emph{patches}
\[
\domain_0=\bigcup_{i=1}^M\domain_0^{(i)}.
\]
Therein, the intersection $\domain_0^{(i)}\cap \domain_0^{(j)}$ consists at most of a common vertex, a common edge, or a common face for $i\neq j$. In particular, each patch is given as $\domain_0^{(i)} = \mapnurbs(\square)$, where $\mapnurbs$ is an invertible \emph{NURBS mapping} $\mapnurbs\colon\square\to \domain_0^{(i)}$ given through
\[
\mapnurbs(x,y,z)
=
\sum_{i_1=0}^{k_1}
\sum_{i_2=0}^{k_2}
\sum_{i_3=0}^{k_3}
\frac{B_{i_1,p_1}(x)B_{i_2,p_2}(y)B_{i_3,p_3}(z)w_{i_1,i_2,i_3}}{\sum_{j_1=0}^{k_1}\sum_{j_2=0}^{k_2}\sum_{j_3=0}^{k_3}B_{j_1,p_1}(x)B_{j_2,p_2}(y)B_{j_3,p_3}(z)w_{j_1,j_2,j_3}}\bfP_{i_1,i_2,i_3},
\]
for \emph{control points} $\bfP_{i_1,i_2,i_3}\in\bbR^3$ and \emph{weights} $w_{i_1,i_2,i_3}\in\bbR$.

\subsection{Geometry Deformation and Spline Spaces}
In most cases, the geometry deformation as given through \eqref{eq:specificdeformationfield} and \eqref{eq:KLE:G} will not be available in closed-form representations. Instead, the $\vec{V}_i$ will need to be approximated numerically, for which we need to define suitable finite-dimensional function spaces on $\partial\domain_0$. To this end, we define the spline spaces spanned by the one-dimensional B-splines $\mathbb{S}^p(\Xi)=\operatorname{span}(\lbrace B_{j,p}\rbrace_{j <k})$ constructed from a given knot vector $\Xi$. Spline spaces in higher spatial dimensions are defined through tensor-product constructions, i.e., for a tuple of knot vectors $\boldsymbol{\Xi} =(\Xi_1,\Xi_2,\Xi_3)$ and polynomial degrees $\bfp=(p_1,p_2,p_3)$, we define
\[
\mathbb{S}_0^{\bfp}(\boldsymbol{\Xi})= \mathbb{S}^{p_1}(\Xi_1)\otimes\mathbb{S}^{p_2}(\Xi_2)\otimes\mathbb{S}^{p_3}(\Xi_3).
\]
One of the requirements to obtain a reasonably deformed domain is that the deformation field is at least continuous. Thus, given a tuple of knot vectors $\boldsymbol{\Xi}$ and polynomial degrees $\bfp$, a natural choice for the discretization of the perturbation field $\vec{V}$ is thus given by the vector valued spline space
\[
\pmb{\mathbb{S}}_0^{\bfp}(\domain_0)
=
\big[\mathbb{S}_0^{\bfp}(\domain_0)\big]^3
\]
where
\[
\mathbb{S}_0^{\bfp}(\domain_0)
=
\left\lbrace \mapf\in
C(\domain_0)\colon \mapf\circ \mapnurbs\in \mathbb{S}_0^{\bfp}(\boldsymbol{\Xi})
\text{ for }1\leq i\leq M\right\rbrace.
\]
Of course, the knot vectors and polynomial degrees could vary in each
component and on each patch, but for better readability we opt for using the
same knots and degrees. For remarks on the approximation property of these spaces we refer to \cite{BDK+2020}.
An efficient approach to compute a discrete approximation to a
Karhunen-Lo\`eve expansion with vanishing mean and prescribed
covariance was described in \cite{DHJM2022}.

Given an approximation of \eqref{eq:KLE:G} in $\pmb{\mathbb{S}}_0^{\bfp}(\domain_0)$  we can represent the deformed random domains $\domain_{\vt}(\omega)=\maptil_{\vt}(\vec{x},\vec{\rand}(\omega))$ as $\domain_{\vt}=\cup_{i=1}^M\domain_{\vt}^{(i)}(\omega)$, where $\domain_{\vt}^{(i)}(\omega) = \mapnurbs(\square,\omega)$ and 
\[
\mapnurbs(\square,\omega)
=
\maptil_{\vt}(\mapnurbs(\square),\vec{\rand}(\omega))
=
\mapnurbs(\square)
+
\vt\sum \limits_{i=1}^{\nKL} \rand_i \vec{V}_i(\mapnurbs(\square))
\]
With $\mapnurbs(\cdot)$ being a NURBS mapping, $\vec{V}_i(\mapnurbs(\cdot))$ being a B-spline space, and the sum of these being a NURBS mapping, we note that $\mapnurbs(\cdot,\omega)$ is again a NURBS mapping. As such, the random domains $\domain_{\vt}(\omega)$ can be dealt with in the same isogeometric framework as $\domain_0$.
We remark that, despite small deformations being the primary interest of this work, this modelling approach also permits for large shape deformations, see, e.g., \cite{DHJM2022,DHM2023,georg_uncertainty_2019,ziegler_mode_2022}.

\subsection{Isogeometric discretization of Maxwell's eigenvalue problem}
For the isogeometric $\HzerocurlDtomega$-conforming discretization of the Maxwell eigenvalue problem we follow \cite{Buffa_2011aa} to define
\begin{align*}
\pmb{\mathbb{S}}_1^{\bfp}(\boldsymbol{\Xi})
&=
\mathbb{S}^{p_1-1}(\Xi_1')\otimes\mathbb{S}^{p_2}(\Xi_2)\otimes\mathbb{S}^{p_3}(\Xi_3)\\
&\qquad\times
\mathbb{S}^{p_1}(\Xi_1)\otimes\mathbb{S}^{p_2-1}(\Xi_2')\otimes\mathbb{S}^{p_3}(\Xi_3)\\
&\qquad\times
\mathbb{S}^{p_1}(\Xi_1)\otimes\mathbb{S}^{p_2}(\Xi_2)\otimes\mathbb{S}^{p_3-1}(\Xi_3'),
\end{align*}
where $\Xi_i'$, $i=1,2,3$, denotes the truncated knot vector $\Xi_i$, i.e., $\Xi_i$ without the first and the last element. This allows to define $\HzerocurlDtomega$-conforming spaces by
\[
\pmb{\mathbb{S}}_1^{\bfp}(\domain_0)
=
\left\lbrace \mapf\in\HzerocurlDtomega\colon \mapf\circ \mapnurbs\in \mathbb{S}_1^{\bfp}(\boldsymbol{\Xi})
\text{ for }1\leq i\leq M\right\rbrace.
\]
For approximation properties of this space we refer to \cite{BDK+2020}.

\section{Numerical Examples} \label{sec:num_examples}
\subsection{Example Problem: Shape Deformation of a TESLA Cavity} \label{sec:num_example_setup}

We demonstrate our results investigating the shape deformation of the 9-cell TESLA cavity, where the cell centers are misaligned with respect to the ideal axis cavity. We focus on the issue of misaligned cells with respect to the ideal cavity axis, which are inherently introduced during the production processes of the dumbbells and at the welding at the equators, see \cref{fig:tesla_manufacturing}.
\def\diameterlengthratio{0.1991}
\def\arrowheight{0.15}
\def\weldingarrow{-\arrowheight}
\def\eqradius{0.1033}
\def\shift{0.4}
\newcommand{\dbcol}{black!20}
\newcommand{\dbcoll}{black!20}
\newcommand{\egcol}{teal!30}

\def\centerone{-0.4620}
\def\centertwo{-0.3467}
\def\centerthree{-0.2313}
\def\centerfour{-0.1159}
\def\centerfive{-0.0005}
\def\centersix{0.1149}
\def\centerseven{0.2303}
\def\centereight{0.3457}
\def\centernine{0.4610}
\def\minleft{-0.5180}
\def\maxright{0.5180}

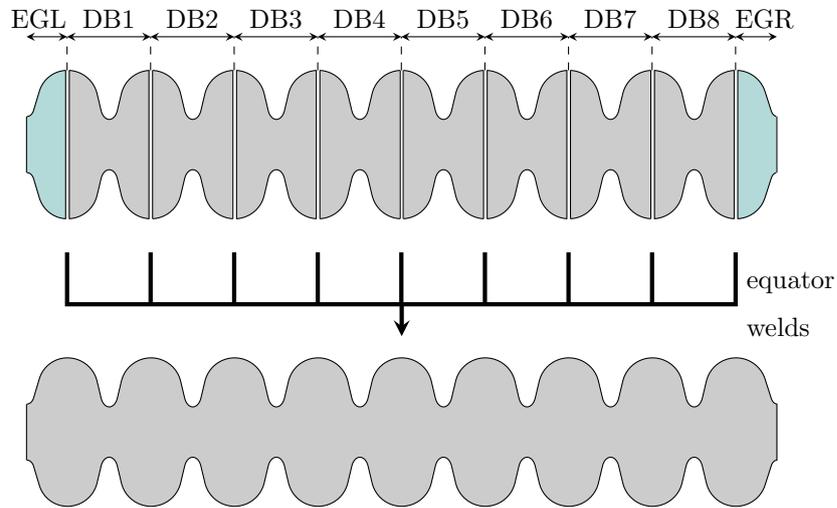
\begin{figure}[htbp]
    \centering
\hspace{-5em}\begin{tikzpicture}
\begin{axis}[
        axis line style={draw=none},
      tick style={draw=none},
      yticklabels={,,},
      xticklabels={,,},
       width = 0.9\textwidth,
       axis equal,
]

\draw[black, stealth-stealth] (axis cs:\minleft,\arrowheight) --  (axis cs:\centerone,\arrowheight) node[pos = 0.3, anchor = south] {EGL};
\draw[black, stealth-stealth] (axis cs:\centerone,\arrowheight) --  (axis cs:\centertwo,\arrowheight) node[midway, anchor = south] {DB1};
\draw[black, dashed] (axis cs: \centerone,\eqradius) --  (axis cs:\centerone,\arrowheight);
\draw[black, dashed] (axis cs: \centertwo,\eqradius) --  (axis cs:\centertwo,\arrowheight);
\draw[black, stealth-stealth] (axis cs:\centertwo,\arrowheight) --  (axis cs:\centerthree,\arrowheight) node[midway, anchor = south] {DB2};

\draw[black, dashed] (axis cs: \centerthree,\eqradius) --  (axis cs:\centerthree,\arrowheight);
\draw[black, stealth-stealth] (axis cs:\centerthree,\arrowheight) --  (axis cs:\centerfour,\arrowheight) node[midway, anchor = south] {DB3};

\draw[black, dashed] (axis cs: \centerfour,\eqradius) --  (axis cs:\centerfour,\arrowheight);
\draw[black, stealth-stealth] (axis cs:\centerfour,\arrowheight) --  (axis cs:\centerfive,\arrowheight) node[midway, anchor = south] {DB4};

\draw[black, dashed] (axis cs: \centerfive,\eqradius) --  (axis cs:\centerfive,\arrowheight);
\draw[black, stealth-stealth] (axis cs:\centerfive,\arrowheight) --  (axis cs:\centersix,\arrowheight) node[midway, anchor = south] {DB5};

\draw[black, dashed] (axis cs: \centersix,\eqradius) --  (axis cs:\centersix,\arrowheight);
\draw[black, stealth-stealth] (axis cs:\centersix,\arrowheight) --  (axis cs:\centerseven,\arrowheight) node[midway, anchor = south] {DB6};

\draw[black, dashed] (axis cs: \centerseven,\eqradius) --  (axis cs:\centerseven,\arrowheight);
\draw[black, stealth-stealth] (axis cs:\centerseven,\arrowheight) --  (axis cs:\centereight,\arrowheight) node[midway, anchor = south] {DB7};

\draw[black, dashed] (axis cs: \centereight,\eqradius) --  (axis cs:\centereight,\arrowheight);
\draw[black, stealth-stealth] (axis cs:\centereight,\arrowheight) --  (axis cs:\centernine,\arrowheight) node[midway, anchor = south] {DB8};

\draw[black, dashed] (axis cs: \centernine,\eqradius) --  (axis cs:\centernine,\arrowheight);
\draw[black, stealth-stealth] (axis cs:\centernine,\arrowheight) --  (axis cs:\maxright,\arrowheight) node[pos = 0.7, anchor = south] {EGR};

\addplot[name path = eglA, color=black] table[x=z,y=y,col sep=comma, restrict x to domain = \minleft: \centerone]{data/tesla_geo_1.csv};
\addplot[name path = eglB, color=black] table[x=z,y=y,col sep=comma, restrict x to domain = \minleft: \centerone]{data/tesla_geo_2.csv};
\draw[black] (axis cs: \minleft,0.039) --  (axis cs:\minleft,-0.039);
\draw[black] (axis cs: -0.4640,\eqradius) --  (axis cs:-0.4640,-\eqradius);

\addplot[name path = db1A, color=black] table[x=z,y=y,col sep=comma, restrict x to domain = \centerone: \centertwo]{data/tesla_geo_1.csv};
\addplot[name path = db1B, color=black] table[x=z,y=y,col sep=comma, restrict x to domain = \centerone: \centertwo]{data/tesla_geo_2.csv};
\draw[black] (axis cs: -0.459,\eqradius) --  (axis cs:-0.459,-\eqradius);
\draw[black] (axis cs: -0.349,\eqradius) --  (axis cs:-0.349,-\eqradius);

\addplot[name path = db2A, color=black] table[x=z,y=y,col sep=comma, restrict x to domain = \centertwo: \centerthree]{data/tesla_geo_1.csv};
\addplot[name path = db2B, color=black] table[x=z,y=y,col sep=comma, restrict x to domain = \centertwo: \centerthree]{data/tesla_geo_2.csv};
\draw[black] (axis cs: -0.344,\eqradius) --  (axis cs:-0.344,-\eqradius);
\draw[black] (axis cs: -0.233,\eqradius) --  (axis cs:-0.233,-\eqradius);

\addplot[name path = db3A, color=black] table[x=z,y=y,col sep=comma, restrict x to domain = \centerthree: \centerfour]{data/tesla_geo_1.csv};
\addplot[name path = db3B, color=black] table[x=z,y=y,col sep=comma, restrict x to domain = \centerthree: \centerfour]{data/tesla_geo_2.csv};
\draw[black] (axis cs: -0.228,\eqradius) --  (axis cs:-0.228,-\eqradius);
\draw[black] (axis cs: -0.118,\eqradius) --  (axis cs:-0.118,-\eqradius);

\addplot[name path = db4A, color=black] table[x=z,y=y,col sep=comma, restrict x to domain = \centerfour: \centerfive]{data/tesla_geo_1.csv};
\addplot[name path = db4B, color=black] table[x=z,y=y,col sep=comma, restrict x to domain = \centerfour: \centerfive]{data/tesla_geo_2.csv};
\draw[black] (axis cs: -0.113,\eqradius) --  (axis cs:-0.113,-\eqradius);
\draw[black] (axis cs: -0.002,\eqradius) --  (axis cs:-0.002,-\eqradius);

\addplot[name path = db5A, color=black] table[x=z,y=y,col sep=comma, restrict x to domain = \centerfive: \centersix]{data/tesla_geo_1.csv};
\addplot[name path = db5B, color=black] table[x=z,y=y,col sep=comma, restrict x to domain = \centerfive: \centersix]{data/tesla_geo_2.csv};
\draw[black] (axis cs: 0.113,\eqradius) --  (axis cs:0.113,-\eqradius);
\draw[black] (axis cs: 0.002,\eqradius) --  (axis cs:0.002,-\eqradius);

\addplot[name path = db6A, color=black] table[x=z,y=y,col sep=comma, restrict x to domain = \centersix: \centerseven]{data/tesla_geo_1.csv};
\addplot[name path = db6B, color=black] table[x=z,y=y,col sep=comma, restrict x to domain = \centersix: \centerseven]{data/tesla_geo_2.csv};
\draw[black] (axis cs: 0.228,\eqradius) --  (axis cs:0.228,-\eqradius);
\draw[black] (axis cs: 0.118,\eqradius) --  (axis cs:0.118,-\eqradius);

\addplot[name path = db7A, color=black] table[x=z,y=y,col sep=comma, restrict x to domain = \centerseven: \centereight]{data/tesla_geo_1.csv};
\addplot[name path = db7B, color=black] table[x=z,y=y,col sep=comma, restrict x to domain = \centerseven: \centereight]{data/tesla_geo_2.csv};
\draw[black] (axis cs: 0.344,\eqradius) --  (axis cs:0.344,-\eqradius);
\draw[black] (axis cs: 0.233,\eqradius) --  (axis cs:0.233,-\eqradius);

\addplot[name path = db8A, color=black] table[x=z,y=y,col sep=comma, restrict x to domain = \centereight :   \centernine]{data/tesla_geo_1.csv};
\addplot[name path = db8B, color=black] table[x=z,y=y,col sep=comma, restrict x to domain = \centereight :   \centernine]{data/tesla_geo_2.csv};
\draw[black] (axis cs: 0.459,\eqradius) --  (axis cs:0.459,-\eqradius);
\draw[black] (axis cs: 0.349,\eqradius) --  (axis cs:0.349,-\eqradius);

\addplot[name path = egrA, color=black] table[x=z,y=y,col sep=comma, restrict x to domain = \centernine: \maxright]{data/tesla_geo_1.csv};
\addplot[name path = egrB, color=black] table[x=z,y=y,col sep=comma, restrict x to domain = \centernine: \maxright]{data/tesla_geo_2.csv};
\draw[black] (axis cs: 0.4640,\eqradius) --  (axis cs:0.4640,-\eqradius);
\draw[black] (axis cs: \maxright,0.039) --  (axis cs:\maxright,-0.039);

\addplot [\egcol] fill between [of = eglA and eglB, soft clip={domain=\minleft:\centerone}];
\addplot [\dbcoll] fill between [of = db1A and db1B, soft clip={domain=\centerone:\centertwo}];
\addplot [\dbcol] fill between [of = db2A and db2B, soft clip={domain=\centertwo:\centerthree}];
\addplot [\dbcoll] fill between [of = db3A and db3B, soft clip={domain=\centerthree: \centerfour}];
\addplot [\dbcol] fill between [of = db4A and db4B, soft clip={domain= \centerfour:\centerfive}];
\addplot [\dbcoll] fill between [of = db5A and db5B, soft clip={domain=\centerfive:\centersix}];
\addplot [\dbcol] fill between [of = db6A and db6B, soft clip={domain=\centersix:\centerseven}];
\addplot [\dbcoll] fill between [of = db7A and db7B, soft clip={domain=\centerseven:\centereight }];
\addplot [\dbcol] fill between [of = db8A and db8B, soft clip={domain=\centereight :\centernine}];
\addplot [\egcol] fill between [of = egrA and egrB, soft clip={domain=\centernine:\maxright}];

\addplot[name path = A, color=black] table[x=z,y expr=\thisrow{y} - \shift,col sep=comma]{data/tesla_geo_1.csv};
\addplot[name path = B, color=black] table[x=z,y expr=\thisrow{y} - \shift ,col sep=comma]{data/tesla_geo_2.csv};
\draw[black] (axis cs: \maxright,0.039-\shift) --  (axis cs:\maxright,-0.039-\shift);
\draw[black] (axis cs: \minleft,0.039-\shift) --  (axis cs:\minleft,-0.039-\shift);
\addplot [\dbcoll] fill between [of = A and B];

\draw[ultra thick] (axis cs:\centerone,\weldingarrow) |- (axis cs:0,-\shift/1.8);
\draw[ultra thick] (axis cs:\centernine,\weldingarrow) |- (axis cs:0,-\shift/1.8) node[pos = 0.5, anchor = west, align = left] {equator\\ welds};
\draw[ultra thick] (axis cs:\centertwo,\weldingarrow) -- (axis cs:\centertwo,-\shift/1.8);
\draw[ultra thick] (axis cs:\centerthree,\weldingarrow) -- (axis cs:\centerthree,-\shift/1.8);
\draw[ultra thick] (axis cs:\centerfour,\weldingarrow) -- (axis cs:\centerfour,-\shift/1.8);
\draw[ultra thick] (axis cs:\centerfive,\weldingarrow) -- (axis cs:\centerfive,-\shift/1.8);
\draw[ultra thick] (axis cs:\centersix,\weldingarrow) -- (axis cs:\centersix,-\shift/1.8);
\draw[ultra thick] (axis cs:\centerseven,\weldingarrow) -- (axis cs:\centerseven,-\shift/1.8);
\draw[ultra thick] (axis cs:\centereight,\weldingarrow) -- (axis cs:\centereight,-\shift/1.8);
\draw[ultra thick, -stealth] (axis cs:0,-\shift/1.8) -- (axis cs:0,-\shift/1.5);

\end{axis}
\end{tikzpicture}\hspace{-5em}
\vspace{-3em}

    \caption{The TESLA cavity and the manufacturing from dumbbells (DB) and endgroups (EGL = endgroup left, EGR = endgroup right).}
    \label{fig:tesla_manufacturing}
\end{figure}

For the shape deformations, we use a Karhunen-Lo\`eve decomposition with respect to the mean geometry (including the deviations from the nominal design) such that the random realizations have zero mean. The decomposition uses $\nKL = 7$ terms which was derived in \cite{georg_uncertainty_2019} from real-world data from the Deutsches Elektronen-Synchrotron (DESY) database~\cite{Yasar2013}. To this end, the authors represent the displacement of the nine cell centers in the transversal planes, i.e. in the $x$- and $y$-direction, while keeping the position along the longitudinal $z$-axis fixed. The deformation between the cell centers is smoothly interpolated by splines. 
The modes identified by the authors of~\cite{georg_uncertainty_2019} in the Karhunen-Lo\`eve decomposition, which describe the displacement of the central axis, are referred to as deformation modes. %
These should not be confused with the eigenmodes that we obtain by solving the eigenvalue problem in the cavity~\eqref{eq:weak_evp}.
The values of the displacement for each cell center in $x$- and $y$-direction for the mean geometry and each deformation mode are shown in~\cref{tab:displacements}.
A Matlab code for generating random realizations is provided in~\cite{Repo_Demo}.

\begin{table}[htb]
    \centering
    
    \begin{tabular}{l|r|rrrrrrr}
       &   & \multicolumn{7}{c}{Deformation modes}\\ 
       & mean & 1 & 2 & 3 & 4 & 5 & 6 & 7 \\\hline
       $x_1$ & -0.0203 & 0.0200 & -0.0022 & 0.0231 & 0.0377 & -0.0061 & -0.0370 & 0.0131 \\
       $x_2$ & -0.0297 & 0.0748 & -0.0172 & 0.0238 & 0.0524 & 0.0164 & -0.0358 & 0.0098 \\
       $x_3$ & 0.0012  & 0.0957 & -0.0198 & 0.0307 & 0.0593 & 0.0007 & -0.0099 & -0.0104 \\
    $x_4$ & 0.0193 & 0.1168 & -0.0264 & 0.0251 & 0.0360 & -0.0111 & 0.0234 & -0.0143 \\
    $x_5$ & 0.0046 & 0.1251 & -0.0309 & 0.0133 & 0.0030 & 0.0000 & 0.0343 & 0.0014 \\
    $x_6$ & 0.0074 & 0.1261 & -0.0359 & -0.0052 & -0.0310 & 0.0039 & 0.0225 & 0.0136 \\
    $x_7$ & 0.0003 & 0.1144 & -0.0343 & -0.0191 & -0.0517 & -0.0002 & -0.0072 & 0.0106 \\
    $x_8$ & 0.0013 & 0.0831 & -0.0260 & -0.0226 & -0.0504 & -0.0069 & -0.0329 & -0.0030 \\
    $x_9$ & 0.0313 & 0.0504 & -0.0153 & -0.0194 & -0.0343 & -0.0175 & -0.0386 & -0.0220 \\ \hline
    $y_1$ & 0.0246 & 0.0196 & 0.0301 & -0.0453 & 0.0120 & 0.0226 & 0.0048 & -0.0338 \\
    $y_2$ & 0.0430 & 0.0384 & 0.0594 & -0.0752 & 0.0181 & 0.0381 & -0.0016 & -0.0091 \\
    $y_3$ & 0.0495 & 0.0380 & 0.0912 & -0.0665 & 0.0193 & 0.0086 & 0.0002 & 0.0125 \\
    $y_4$ & 0.0406 & 0.0351 & 0.1118 & -0.0353 & 0.0140 & -0.0200 & 0.0006 & 0.0181 \\
    $y_5$ & 0.0053 & 0.0325 & 0.1199 & -0.0011 & 0.0044 & -0.0310 & 0.0016 & 0.0086 \\
    $y_5$ & -0.0349 & 0.0289 & 0.1208 & 0.0310 & -0.0112 & -0.0228 & 0.0013 & -0.0105 \\
    $y_7$ & -0.0551 & 0.0195 & 0.1090 & 0.0577 & -0.0221 & -0.0015 & -0.0010 & -0.0156 \\
    $y_8$ & -0.0582 & 0.0155 & 0.0773 & 0.0614 & -0.0272 & 0.0318 & -0.0043 & -0.0019 \\
    $y_9$ & -0.0292 & 0.0140 & 0.0426 & 0.0379 & -0.0226 & 0.0553 & -0.0034 & 0.0133 \\
    \end{tabular}
    \caption{Displacements of each cell center in the $x$- and $y$-direction for the mean geometry and each deformation mode as identified in~\cite{georg_uncertainty_2019}. 
    The data are given in \si{mm}.}
    \label{tab:displacements}
\end{table}

The deformation of the mean geometry, as well as the deformation modes, are illustrated in \cref{fig:xy_deformations}, where for each deformation mode, the deformed cavity axis is displayed in two planes, the $yz$- and $xz$-plane, respectively and the original cell centers on the ideal cavity axis are marked with red crosses.
These axis shapes are obtained by spline interpolation of the individual cell center displacements as given in \cref{tab:displacements}.
The spline interpolation is also part of the Matlab code in~\cite{Repo_Demo}.
{\def\mywidth{\linewidth}
\def\myheight{0.5\linewidth}

\pgfplotsset{
    ymin = -9e-5,
    ymax = 1.5e-4,
    ytick = {-5e-5, 0, 5e-5, 1e-4, 1.5e-4},
    width=\mywidth, height=\myheight,
    xlabel={$z$ axis (\si{\meter})}, 
    grid=both,
	minor grid style={gray!25},
	major grid style={gray!25},
	ylabel={$x/y$ axis, resp.,  (\si{\meter})},
	ylabel style = {yshift=-0.6em},
	xlabel style = {yshift=.5em},
    }

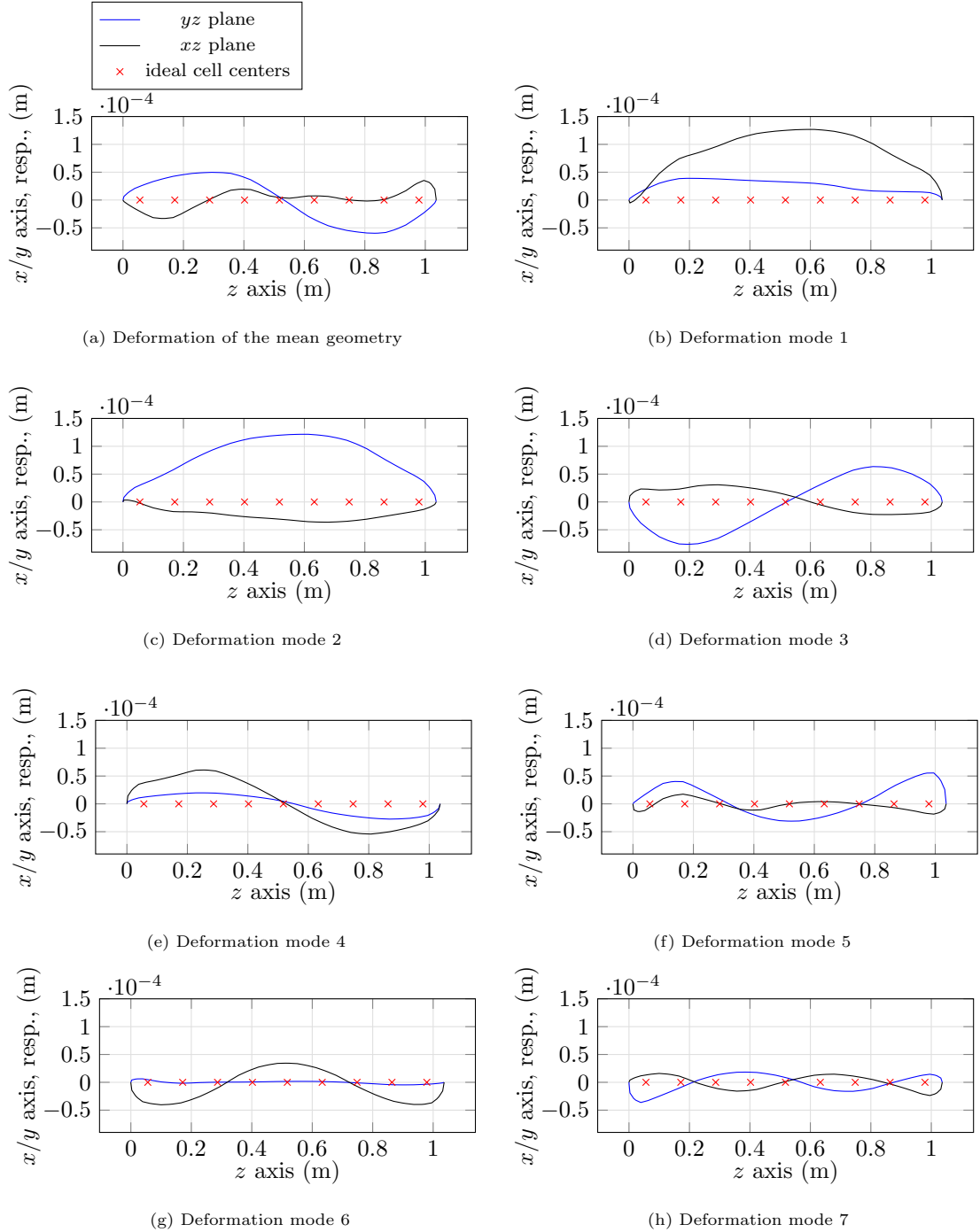
\begin{figure}
    \centering
\begin{subfigure}[b]{0.48\textwidth}
\begin{tikzpicture}
\begin{axis}[
	legend style={font=\footnotesize, at={(0,1.2)},anchor=south west},
	]
\addplot[color=blue] table[x=z,y=y,col sep=comma]{data/Def_mu_yz.csv};
\addplot[color=black] table[x=z,y=x,col sep=comma]{data/Def_mu_xz.csv};
\addplot[color = red, only marks, mark = x] table[x=z,y=axis,col sep=comma]{data/cellcenters.csv};
\legend{$yz$ plane, $xz$ plane, ideal cell centers}
\end{axis}
\end{tikzpicture}%
\caption{Deformation of the mean geometry}
\end{subfigure}
~
\begin{subfigure}[b]{0.48\textwidth}
\begin{tikzpicture}
\begin{axis}%
\addplot[color=blue] table[x=z,y=y,col sep=comma]{data/Def_mode_1_yz.csv};
\addplot[color=black] table[x=z,y=x,col sep=comma]{data/Def_mode_1_xz.csv};
\addplot[color = red, only marks, mark = x] table[x=z,y=axis,col sep=comma]{data/cellcenters.csv};
\end{axis}
\end{tikzpicture}%
\caption{Deformation mode 1}
\end{subfigure}
\vspace{1em}

\begin{subfigure}[t]{0.48\textwidth}
\begin{tikzpicture}
\begin{axis}
\addplot[color=blue] table[x=z,y=y,col sep=comma]{data/Def_mode_2_yz.csv};
\addplot[color=black] table[x=z,y=x,col sep=comma]{data/Def_mode_2_xz.csv};
\addplot[color = red, only marks, mark = x] table[x=z,y=axis,col sep=comma]{data/cellcenters.csv};
\end{axis}
\end{tikzpicture}%
\caption{Deformation mode 2}
\end{subfigure}
~
\begin{subfigure}[t]{0.48\textwidth}
\begin{tikzpicture}
\begin{axis}%
\addplot[color=blue] table[x=z,y=y,col sep=comma]{data/Def_mode_3_yz.csv};
\addplot[color=black] table[x=z,y=x,col sep=comma]{data/Def_mode_3_xz.csv};
\addplot[color = red, only marks, mark = x] table[x=z,y=axis,col sep=comma]{data/cellcenters.csv};
\end{axis}
\end{tikzpicture}%
\caption{Deformation mode 3}
\end{subfigure}
\vspace{1em}

\begin{subfigure}[t]{0.48\textwidth}
\begin{tikzpicture}
\begin{axis}%
\addplot[color=blue] table[x=z,y=y,col sep=comma]{data/Def_mode_4_yz.csv};
\addplot[color=black] table[x=z,y=x,col sep=comma]{data/Def_mode_4_xz.csv};
\addplot[color = red, only marks, mark = x] table[x=z,y=axis,col sep=comma]{data/cellcenters.csv};
\end{axis}
\end{tikzpicture}%
\caption{Deformation mode 4}
\end{subfigure}
~
\begin{subfigure}[t]{0.48\textwidth}
\begin{tikzpicture}
\begin{axis}%
\addplot[color=blue] table[x=z,y=y,col sep=comma]{data/Def_mode_5_yz.csv};
\addplot[color=black] table[x=z,y=x,col sep=comma]{data/Def_mode_5_xz.csv};
\addplot[color = red, only marks, mark = x] table[x=z,y=axis,col sep=comma]{data/cellcenters.csv};
\end{axis}
\end{tikzpicture}%
\caption{Deformation mode 5}
\end{subfigure}

\begin{subfigure}[t]{0.48\textwidth}
\begin{tikzpicture}
\begin{axis}%
\addplot[color=blue] table[x=z,y=y,col sep=comma]{data/Def_mode_6_yz.csv};
\addplot[color=black] table[x=z,y=x,col sep=comma]{data/Def_mode_6_xz.csv};
\addplot[color = red, only marks, mark = x] table[x=z,y=axis,col sep=comma]{data/cellcenters.csv};
\end{axis}
\end{tikzpicture}%
\caption{Deformation mode 6}
\end{subfigure}%
~
\begin{subfigure}[t]{0.48\textwidth}
\begin{tikzpicture}
\begin{axis}%
\addplot[color=blue] table[x=z,y=y,col sep=comma]{data/Def_mode_7_yz.csv};
\addplot[color=black] table[x=z,y=x,col sep=comma]{data/Def_mode_7_xz.csv};
\addplot[color = red, only marks, mark = x] table[x=z,y=axis,col sep=comma]{data/cellcenters.csv};
\end{axis}
\end{tikzpicture}%
\caption{Deformation mode 7}
\end{subfigure}
\caption{The mean deformation and the seven deformation modes as identified in the Karhunen-Loève decomposition displayed for the cavity axis. The red crosses mark the cell centers when the axis is ideal and undeformed. 
The blue and black lines indicate the deformed axes in the mean geometry and for the deformation modes. %
The blue lines correspond to the deformation seen in the $yz$ plane and the black lines to the deformation seen in the $xz$ plane.}
\label{fig:xy_deformations}
\end{figure}
}

\subsection{Computational Setup}
All computations are carried out in MATLAB\textsuperscript{\textregistered}, using the GeoPDEs \cite{Vazquez_2016aa} package for the discretization with IGA.
We perform all computations on the Bonna cluster hosted by the University of Bonn with two Intel\textsuperscript{\textregistered} Xeon\textsuperscript{\textregistered} Gold 6130 CPUs with sixteen \SI{2.10}{\giga \hertz} cores each, hyper-threading enabled and \SI{187}{\giga \byte} RAM.

The 9-cell TESLA cavity with attached beampipes was discretized by splines of degree two, yielding $25,744$ degrees of freedom.
For the numerical examples we want to consider the 27 eigenpairs of smallest magnitude.
The first 9 of these are nondegenerate eigenpairs, while the next 9 eigenvalues are all degenerate with multiplicity $m=2$.
To evaluate the efficiency of our derivative-based approach, we compare it to a polynomial chaos expansion (PCE) using the UQLab software package (version 2.0) \cite{MS2014}. 
To this end, we choose 17 different perturbation amplitudes (i.e., $\vt = \{2^{-5}, 2^{-4.5}, 2^{-4},\ldots,2^2, 5, 6\}$), for each of which we calculate a PCE approximation of the mean and variance.
The sampling process is parallelized using MATLAB\textsuperscript{\textregistered} \texttt{parfor} loops.

\subsection{Verification of the perturbation approach}

For the calculation of the eigenpair derivatives we first solve for the unperturbed eigenvalue.
We then use the derivatives of the coefficients \eqref{eq:ddt_CA0} to calculate the derivatives of the eigenpairs with respect to each deformation mode using equation \eqref{eq:der_eigpair:disc}. Equation \eqref{eq:KLE:tensor_lu} then yields a representation of the covariances of the first derivatives.
These are then used to calculate the approximations of the variances of the eigenvalues as well as the variances of the associated modes.

To validate our perturbation approach, we compare our approximations against mean and variance of a surrogate PCE model in UQLab, choosing polynomial degree of $4$. The coefficients are obtained from $1,000$ Latin hypercube samples with the provided Least Angle Regression (LARS) method and the hyperbolic-norm of $q = 0.75$. 
For the eigenvalues the Euclidean norm of the error over $\vec{\lambda}$ is calculated, while for the eigenmodes we consider the $L^2$-norm of the error of the variances of the eigenspaces.
Due to the construction of the stochastic deformation, Lemma~\ref{lem:Gtoeig}, and \cref{lem:order4} we predict an order of convergence of two for the mean and four for the variance.
We plot the resulting errors of the eigenvalues and eigenspaces for the nine smallest eigenvalues, which are non-degenerate and the following eighteen eigenvalues which are degenerate with multiplicity $m = 2$. 
The results found in \cref{fig:convNondeg} and \cref{fig:convDeg} confirm the predicted order of convergence.

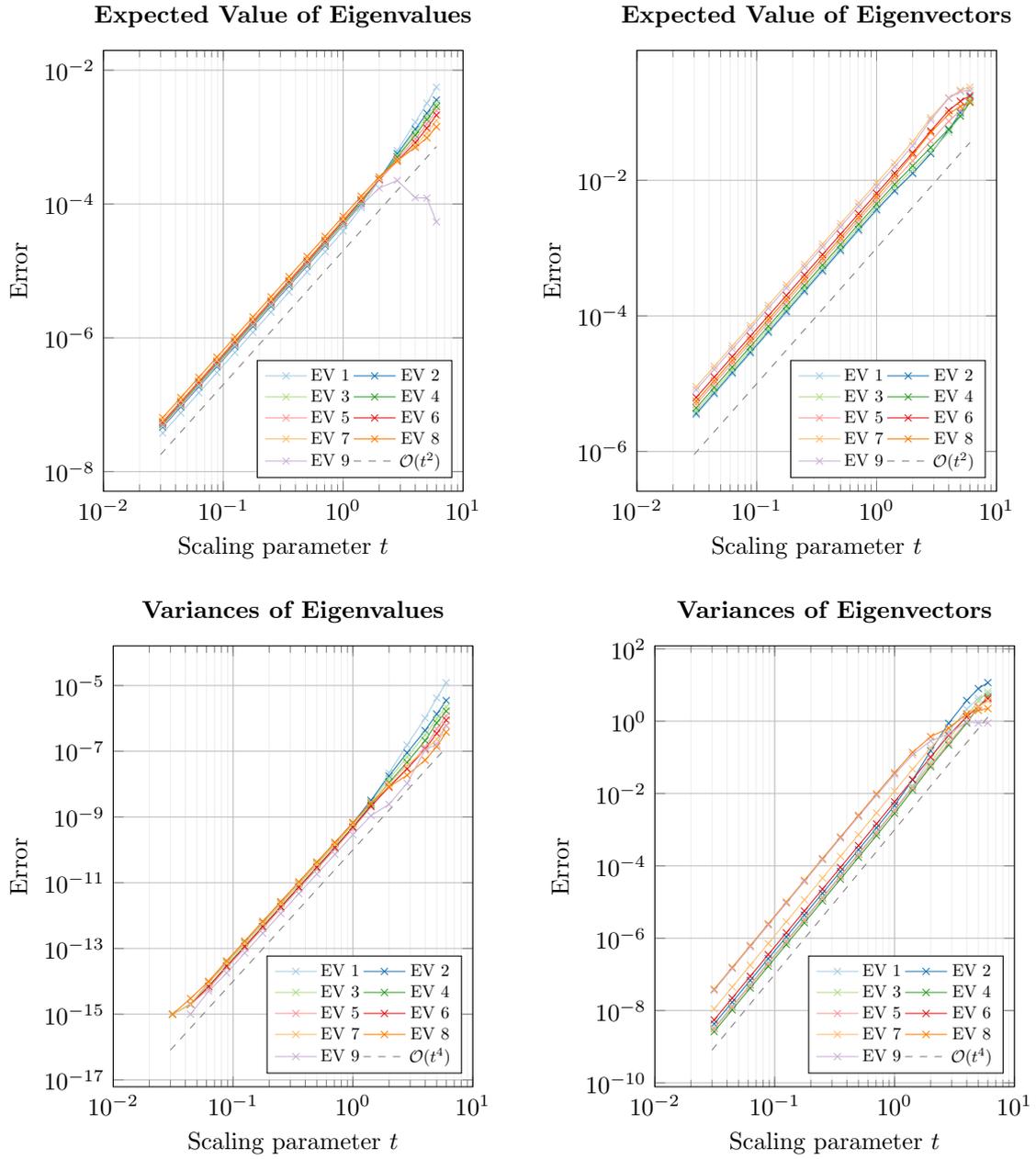
\begin{figure}
\centering
\begin{subfigure}[t]{0.48\textwidth}
\begin{tikzpicture}
\begin{axis}[
    xlabel={Scaling parameter $t$}, ylabel={Error}, %
	xminorticks=true, yminorticks=true, title style={font=\bfseries},
    xmin=0.01,
    xmax=10, 
    width = 0.94\textwidth, height = 0.35\textheight,
    grid = both, 
    major grid style = {lightgray}, minor grid style = {lightgray!25},
    ytick distance = 1e2,
    legend pos=south east,
    legend columns=2, 
    legend style={nodes={scale=0.7, transform shape}},
    ymode=log, xmode=log,
    title = {Expected Value of Eigenvalues},
    cycle list/Paired %
    ]

\pgfplotsforeachungrouped \i in {1,2,...,9}{
  \addplot +[mark = x] table[x=t,y=f,col sep=comma]{convDataPCE_plots9/E_Lambda_group\i.csv};
    \addlegendentry{EV\i}
}
\addplot [dashed, gray, domain=3e-2:6] {2E-5*x^2};

\legend{EV 1, EV 2, EV 3, EV 4, EV 5, EV 6, EV 7, EV 8, EV 9,  $\mathcal{O}(t^2)$};
\end{axis}
\end{tikzpicture}
\end{subfigure}\hfill
\begin{subfigure}[t]{0.48\textwidth}
\begin{tikzpicture}
\begin{axis}[
    xlabel={Scaling parameter $t$}, ylabel={Error}, %
	xminorticks=true, yminorticks=true, title style={font=\bfseries},
    xmin=0.01,
    xmax=10, 
     width = 0.94\textwidth, height = 0.35\textheight,
    grid = both, 
    major grid style = {lightgray}, minor grid style = {lightgray!25},
    ytick distance = 1e2,
    legend pos=south east,
    legend columns=2, 
    legend style={nodes={scale=0.7, transform shape}},
    ymode=log, xmode=log,
    title = {Expected Value of Eigenvectors},
    cycle list/Paired %
    ]

\pgfplotsforeachungrouped \i in {1,2,...,9}{
  \addplot +[mark = x] table[x=t,y=f,col sep=comma]{convDataPCE_plots9/E_u_group\i.csv};
}
\addplot [dashed, gray, domain=3e-2:6] {1E-3*x^2};

\legend{EV 1, EV 2, EV 3, EV 4, EV 5, EV 6, EV 7, EV 8, EV 9, $\mathcal{O}(t^2)$};
\end{axis}
\end{tikzpicture}
\end{subfigure} 
\vspace{1em} 

\begin{subfigure}[t]{0.48\textwidth}
\begin{tikzpicture}
\begin{axis}[
    xlabel={Scaling parameter $t$}, ylabel={Error}, %
	xminorticks=true, yminorticks=true, title style={font=\bfseries},
    xmin=0.01,
    xmax=10, 
     width = 0.94\textwidth, height = 0.35\textheight,
    grid = both, 
    major grid style = {lightgray}, minor grid style = {lightgray!25},
    ytick distance = 1e2,
    legend pos=south east,
    legend columns=2, 
    legend style={nodes={scale=0.7, transform shape}},
    ymode=log, xmode=log,
    title = {Variances of Eigenvalues},
    cycle list/Paired %
    ]

\pgfplotsforeachungrouped \i in {1,2,...,9}{
  \addplot +[mark = x] table[x=t,y=f,col sep=comma]{convDataPCE_plots9/Cov_Lambda_group\i.csv};
}
\addplot [dashed, gray, domain=3e-2:6] {1E-10*x^4};

\legend{EV 1, EV 2, EV 3, EV 4, EV 5, EV 6, EV 7, EV 8, EV 9, $\mathcal{O}(t^4)$};
\end{axis}
\end{tikzpicture}
\end{subfigure}\hfill
\begin{subfigure}[t]{0.48\textwidth}
\begin{tikzpicture}
\begin{axis}[
    xlabel={Scaling parameter $t$}, ylabel={Error}, %
	xminorticks=true, yminorticks=true, title style={font=\bfseries},
    xmin=0.01,
    xmax=10, 
     width = 0.94\textwidth, height = 0.35\textheight,
    grid = both, 
    major grid style = {lightgray}, minor grid style = {lightgray!25},
    ytick distance = 1e2,
    legend pos=south east,
    legend columns=2, 
    legend style={nodes={scale=0.7, transform shape}},
    ymode=log, xmode=log,
    title = {Variances of Eigenvectors},
    cycle list/Paired %
    ]

\pgfplotsforeachungrouped \i in {1,2,...,9}{
  \addplot +[mark = x] table[x=t,y=f,col sep=comma]{convDataPCE_plots9/Cov_u_group\i.csv};
}
\addplot [dashed, gray, domain=3e-2:6] {1E-3*x^4};

\legend{EV 1, EV 2, EV 3, EV 4, EV 5, EV 6, EV 7, EV 8, EV 9, $\mathcal{O}(t^4)$};
\end{axis}
\end{tikzpicture}
\end{subfigure}
\caption{Convergence of error between approximation and PCE estimate for mean and variance, computed for the first nine eigenmodes, belonging to eigenvalues of simple multiplicity.}
\label{fig:convNondeg}
\end{figure}

\begin{figure}
\centering
\begin{subfigure}[t]{0.48\textwidth}
\begin{tikzpicture}
\begin{axis}[
    xlabel={Scaling parameter $t$}, ylabel={Error}, %
	xminorticks=true, yminorticks=true, title style={font=\bfseries},
    xmin=0.01,
    xmax=10, 
    width = 0.94\textwidth, height = 0.35\textheight,
    grid = both, 
    major grid style = {lightgray}, minor grid style = {lightgray!25},
    ytick distance = 1e2,
    legend pos=north west,
    legend columns=2, 
    legend style={nodes={scale=0.7, transform shape}},
    ymode=log, xmode=log,
    title = {Expected Value of Eigenvalues},
    cycle list/Paired %
    ]

\pgfplotsforeachungrouped \i in {1,2,...,9}{
  \addplot +[mark = x] table[x=t,y=f,col sep=comma]{convDataPCE_plots18/E_Lambda_group\i.csv};
}
\addplot [dashed, gray, domain=3e-2:6] {5E-5*x^2};

\legend{ES 10, ES 11, ES 12, ES 13, ES 14, ES 15, ES 16, ES 17, ES 18,  $\mathcal{O}(t^2)$};
\end{axis}
\end{tikzpicture}
\end{subfigure}\hfill
\begin{subfigure}[t]{0.48\textwidth}
\begin{tikzpicture}
\begin{axis}[
    xlabel={Scaling parameter $t$}, ylabel={Error}, %
	xminorticks=true, yminorticks=true, title style={font=\bfseries},
    xmin=0.01,
    xmax=10, 
     width = 0.94\textwidth, height = 0.35\textheight,
    grid = both, 
    major grid style = {lightgray}, minor grid style = {lightgray!25},
    ytick distance = 1e2,
    legend pos=north west,
    legend columns=2, 
    legend style={nodes={scale=0.7, transform shape}},
    ymode=log, xmode=log,
    title = {Expected Value of Eigenspaces},
    cycle list/Paired %
    ]

\pgfplotsforeachungrouped \i in {1,2,...,9}{
  \addplot +[mark = x] table[x=t,y=f,col sep=comma]{convDataPCE_plots18/E_u_group\i.csv};
}
\addplot [dashed, gray, domain=3e-2:6] {1E-3*x^2};

\legend{ES 10, ES 11, ES 12, ES 13, ES 14, ES 15, ES 16, ES 17, ES 18, $\mathcal{O}(t^2)$};
\end{axis}
\end{tikzpicture}
\end{subfigure} 
\vspace{1em} 

\begin{subfigure}[t]{0.48\textwidth}
\begin{tikzpicture}
\begin{axis}[
    xlabel={Scaling parameter $t$}, ylabel={Error}, %
	xminorticks=true, yminorticks=true, title style={font=\bfseries},
    xmin=0.01,
    xmax=10, 
     width = 0.94\textwidth, height = 0.35\textheight,
    grid = both, 
    major grid style = {lightgray}, minor grid style = {lightgray!25},
    ytick distance = 1e2,
    legend pos=north west,
    legend columns=2, 
    legend style={nodes={scale=0.7, transform shape}},
    ymode=log, xmode=log,
    title = {Variances of Eigenvalues},
    cycle list/Paired %
    ]

\pgfplotsforeachungrouped \i in {1,2,...,9}{
  \addplot +[mark = x] table[x=t,y=f,col sep=comma]{convDataPCE_plots18/Cov_Lambda_group\i.csv};
}
\addplot [dashed, gray, domain=3e-2:6] {9E-10*x^4};

\legend{ES 10, ES 11, ES 12, ES 13, ES 14, ES 15, ES 16, ES 17, ES 18, $\mathcal{O}(t^4)$};
\end{axis}
\end{tikzpicture}
\end{subfigure}\hfill
\begin{subfigure}[t]{0.48\textwidth}
\begin{tikzpicture}
\begin{axis}[
    xlabel={Scaling parameter $t$}, ylabel={Error}, %
	xminorticks=true, yminorticks=true, title style={font=\bfseries},
    xmin=0.01,
    xmax=10, 
     width = 0.94\textwidth, height = 0.35\textheight,
    grid = both, 
    major grid style = {lightgray}, minor grid style = {lightgray!25},
    ytick distance = 1e2,
    legend pos=north west,
    legend columns=2, 
    legend style={nodes={scale=0.7, transform shape}},
    ymode=log, xmode=log,
    title = {Variances of Eigenspaces},
    cycle list/Paired %
    ]
\pgfplotsforeachungrouped \i in {1,2,...,9}{
  \addplot +[mark = x] table[x=t,y=f,col sep=comma]{convDataPCE_plots18/Cov_u_group\i.csv};
}
\addplot [dashed, gray, domain=3e-2:6] {1E-3*x^4};

\legend{ES 10, ES 11, ES 12, ES 13, ES 14, ES 15, ES 16, ES 17, ES 18, $\mathcal{O}(t^4)$};
\end{axis}
\end{tikzpicture}
\end{subfigure}
\caption{Convergence of error between approximation and PCE estimate for mean and variance, computed for the eigenspaces 10 to 18, each belonging to eigenvalues of multiplicity~2.}
\label{fig:convDeg}
\end{figure}
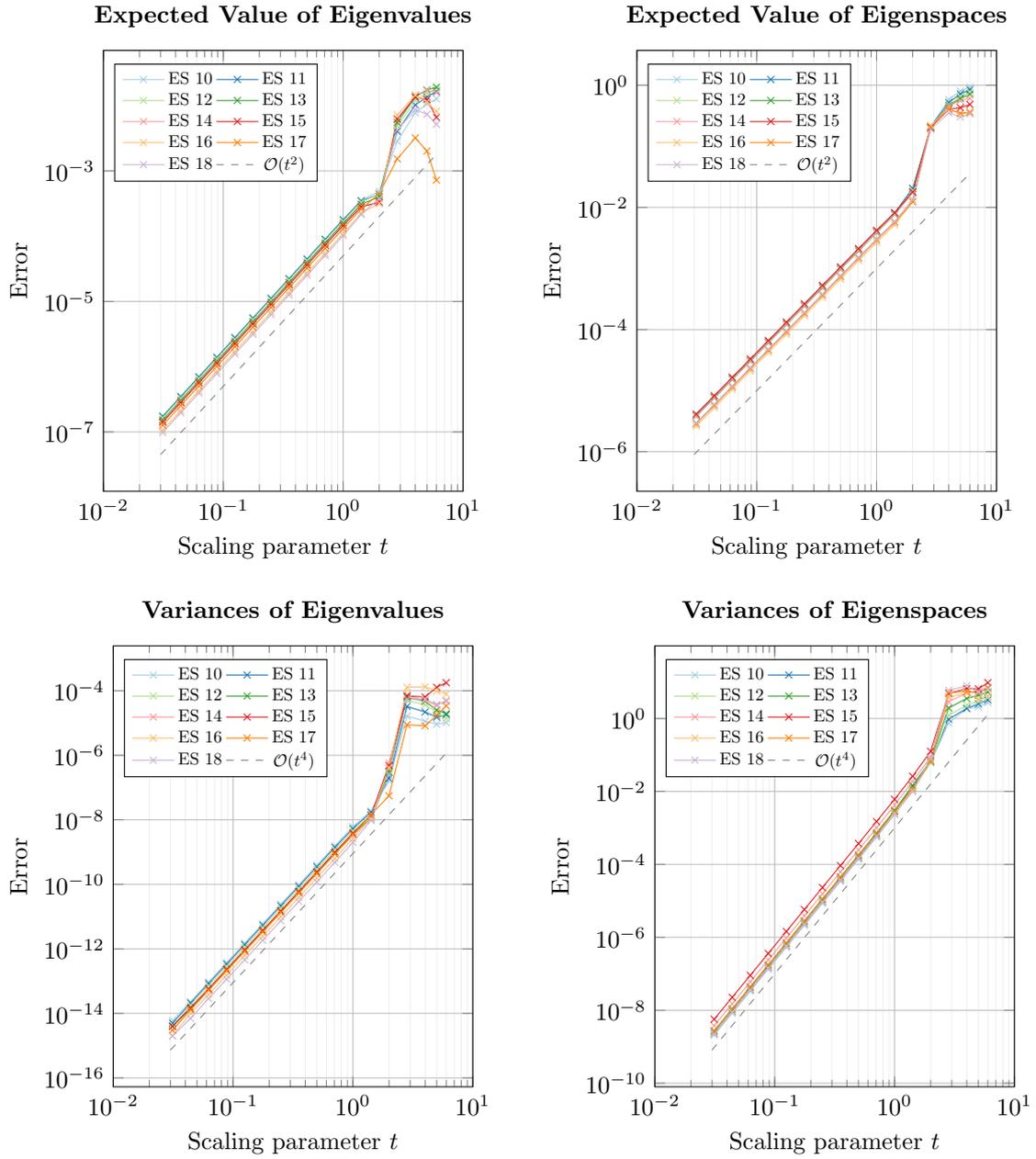

\subsection{Shape Uncertainty Quantification for TESLA Cavities}

The resulting variance maps of the first nine modes are shown in \cref{fig:modeE9var}. 
Based on these plots, we can make some observations: 
First of all, we note a strong connection between the variance and the distribution of the electric field strength, see \cref{fig:E-modes}.
When also considering the orientation of the lines indicating the variances, we note that the $z$-component is particularly prominent. 
This makes sense, also in regard to the first observation, as we can see in \cref{fig:modeE9_componentDef}, where we display the electric field strength, exemplary for the case of the accelerating mode.
Here, we plot the components and the magnitude of the electric field strength of the ninth mode.
Considering the transversal $x$- and $y$-components, we notice that they are close to zero in most parts of the domain and that this also does not change significantly when we deform the cavity.
In contrast, the longitudinal $z$-component, which has the characteristic alternating field pattern, changes vastly during the deformation.
Notably, the most significant field deformations occur in the center of the cells rather than in regions distant from the central axis.
These observations are clearly reflected in the variance map \cref{fig:var_emode9} of the accelerating mode.
We also observe that the variance of the accelerating mode in the outer cells is much larger than in the mid-cells. 
This also is in accordance with the fact that when deforming the cavity, the field strength of the accelerating mode changes most significantly in the outer cells, see \cref{fig:modeE9_componentDef}.

From the variance maps in \cref{fig:modeE9var}, we conclude, that the first and second eigenmode are the most sensitive to the eccentricity, as here, the variance attains the largest values.
Regarding the relevant accelerating mode in \cref{fig:var_emode9}, we observe the strongest variance in the two end-cells, i.e. the E-field, particularly the longitudinal component, is the most sensitive in these cells.
Since the particle bunch travels along the central axis, we focus on this area.
We observe, that in the inner cells, the variance along the central axis is comparatively small, i.e., the eccentricity does not have a major influence on the field. 
In contrast, the field in the two end-cells is more sensitive to our shape deformations, as we can see from the higher variance values, also around the central axis. This is plausible since the end-cells are known to be particularly sensitive, e.g. for the field flatness, see \cite{Shemelin2009}.

\begin{figure}
    \def\imgwidth{0.6\textwidth}
    \centering
    \begin{subfigure}[c]{\textwidth}\centering
    \includegraphics[trim = 100 280 100 280, clip, width = \imgwidth]{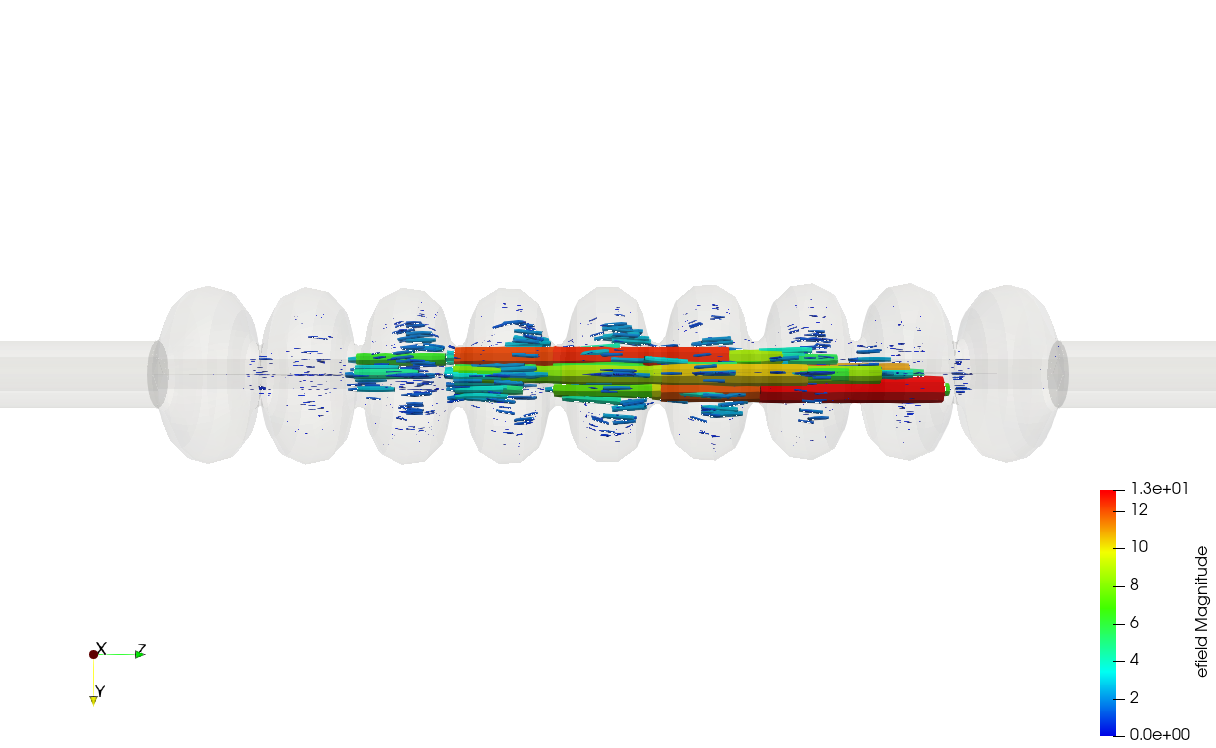}
    \caption{}
    \end{subfigure}\vspace{0.1em}

    \begin{subfigure}[c]{\textwidth}\centering
    \includegraphics[trim = 100 280 100 280, clip, width = \imgwidth]{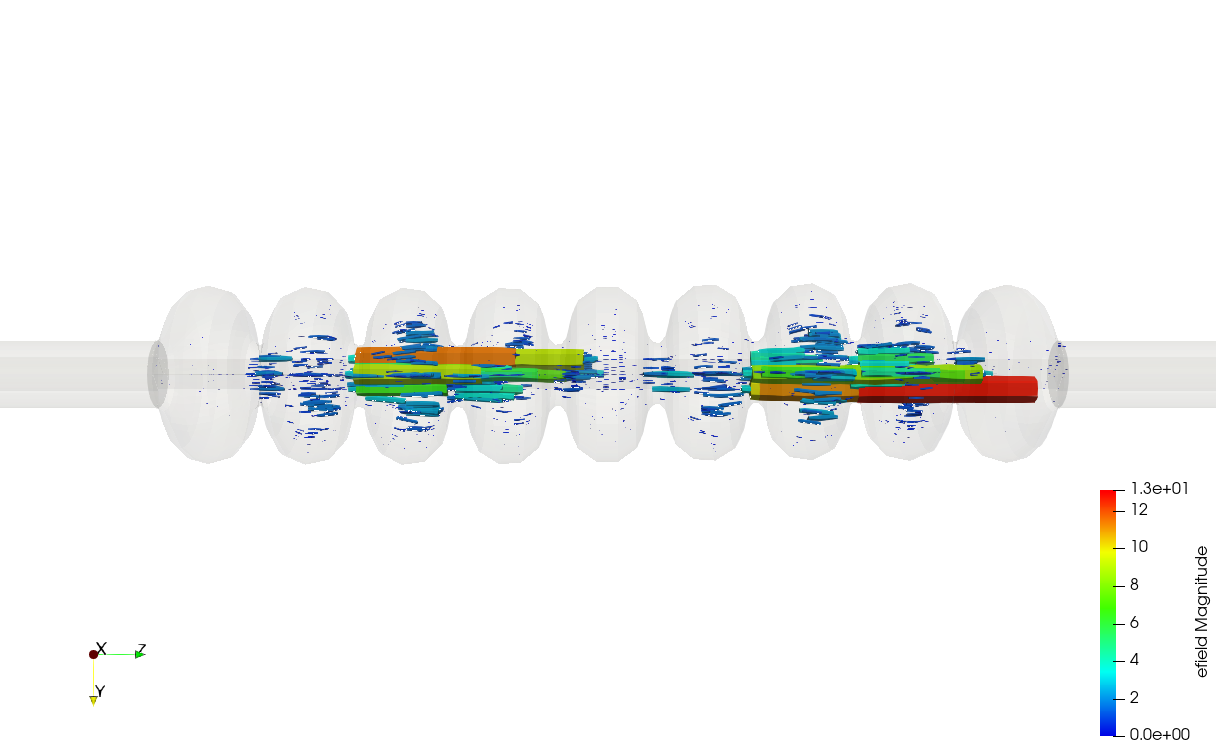}
    \caption{}
    \end{subfigure}\vspace{0.1em}

    \begin{subfigure}[c]{\textwidth}\centering
    \includegraphics[trim = 100 280 100 280, clip, width = \imgwidth]{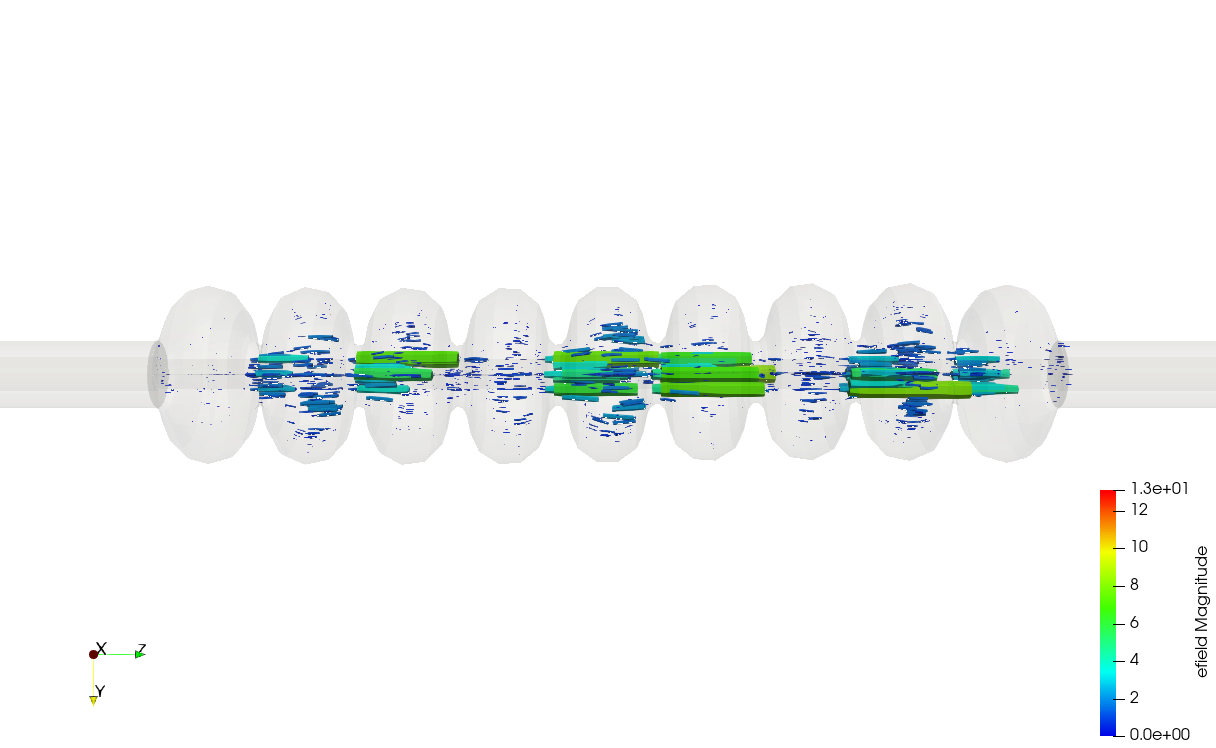}
    \caption{}
    \end{subfigure}\vspace{0.1em}

    \begin{subfigure}[c]{\textwidth}\centering
    \includegraphics[trim = 100 280 100 280, clip, width = \imgwidth]{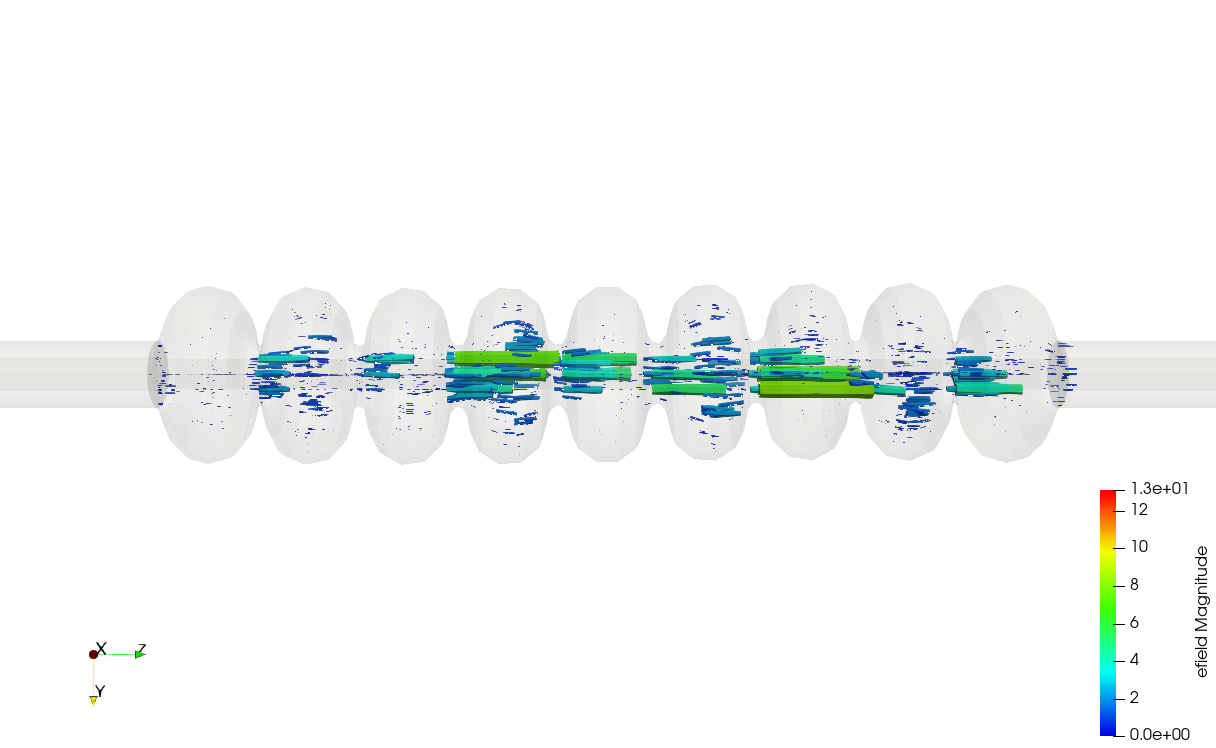}
    \caption{}
    \end{subfigure}\vspace{0.1em}

    \begin{subfigure}[c]{\textwidth}\centering
    \includegraphics[trim = 100 280 100 280, clip, width = \imgwidth]{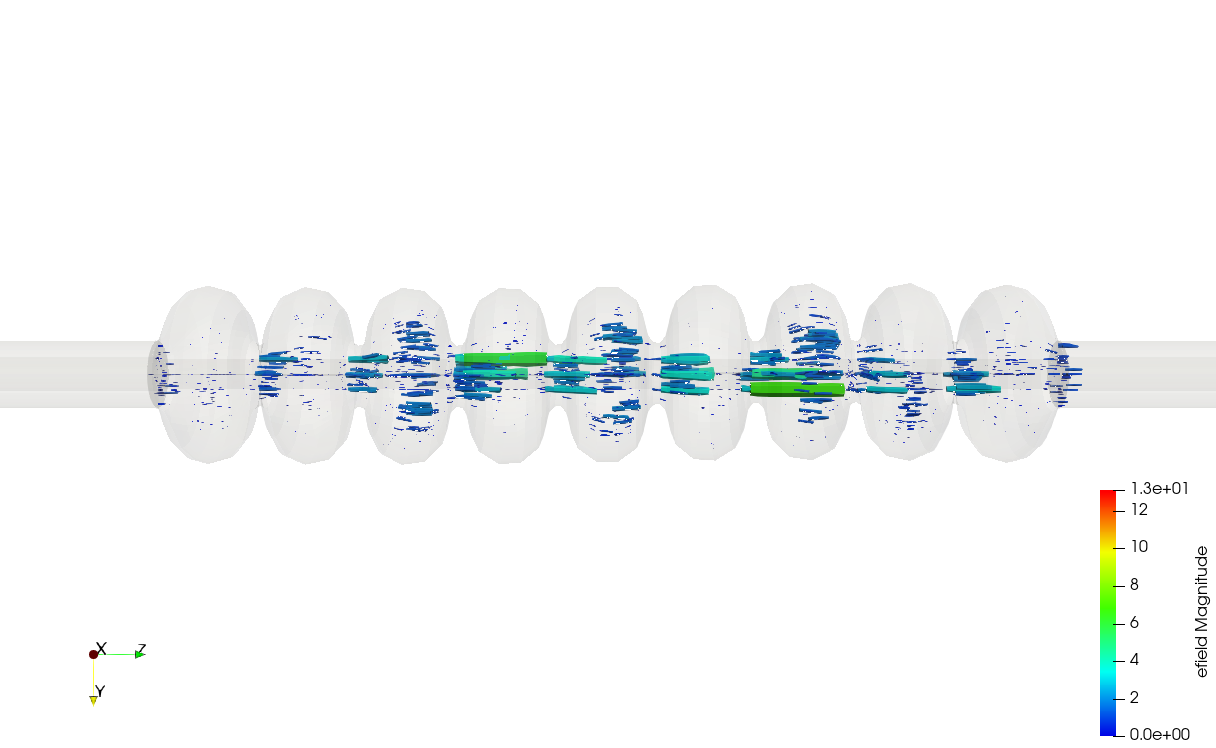}
    \caption{}
    \end{subfigure}\vspace{0.1em}

    \begin{subfigure}[c]{\textwidth}\centering
    \includegraphics[trim = 100 280 100 280, clip, width = \imgwidth]{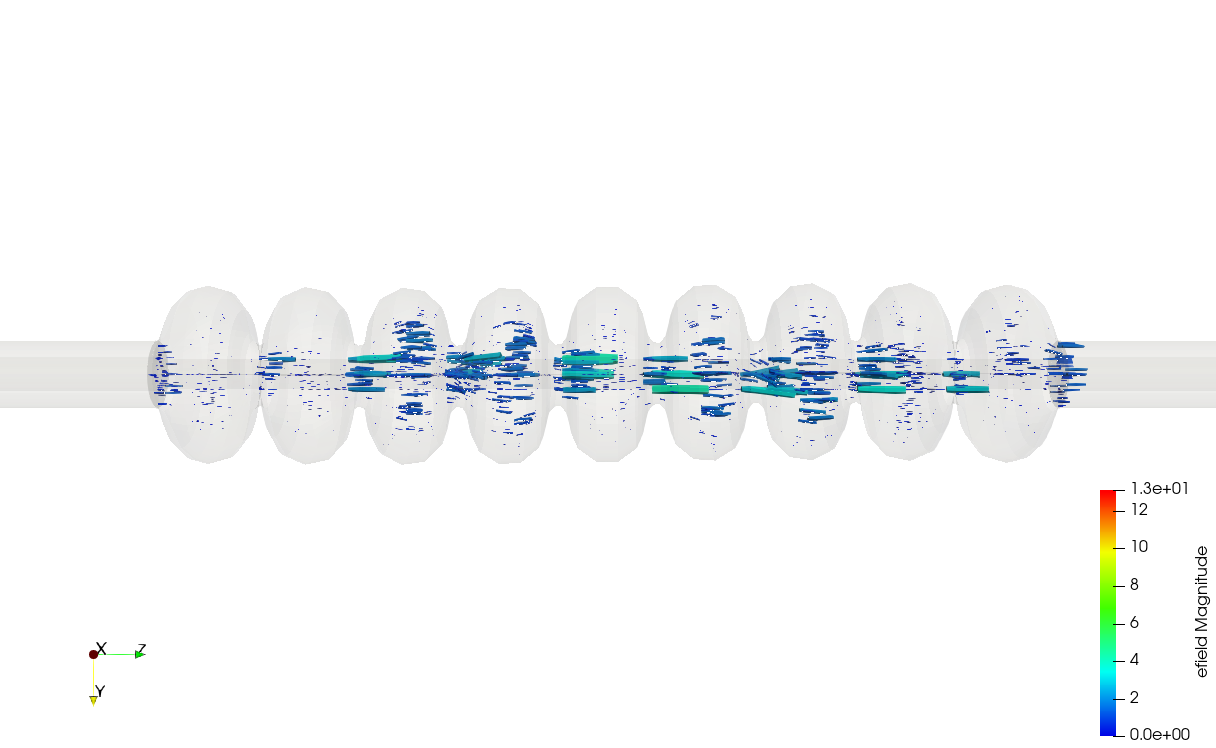}
    \caption{}
    \end{subfigure}\vspace{0.1em}

    \begin{subfigure}[c]{\textwidth}\centering
    \includegraphics[trim = 100 280 100 280, clip, width = \imgwidth]{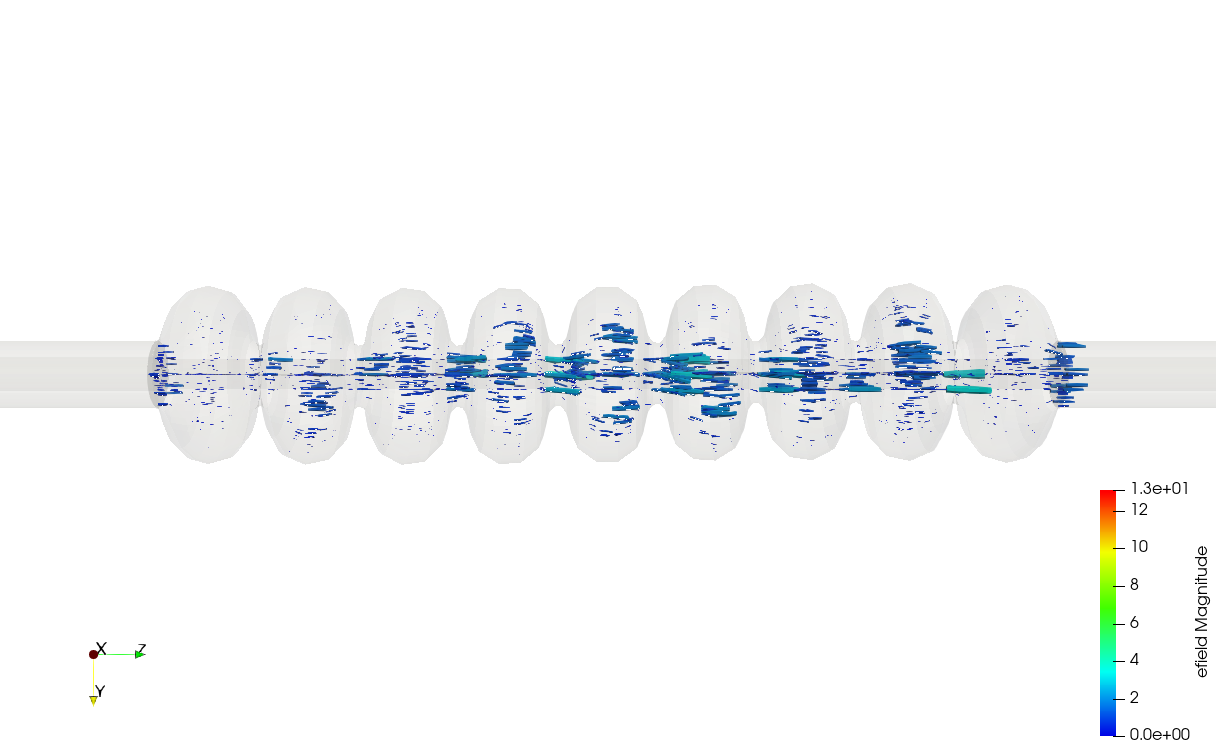}
    \caption{}
    \end{subfigure}\vspace{0.1em}

    \begin{subfigure}[c]{\textwidth}\centering
    \includegraphics[trim = 100 280 100 280, clip, width = \imgwidth]{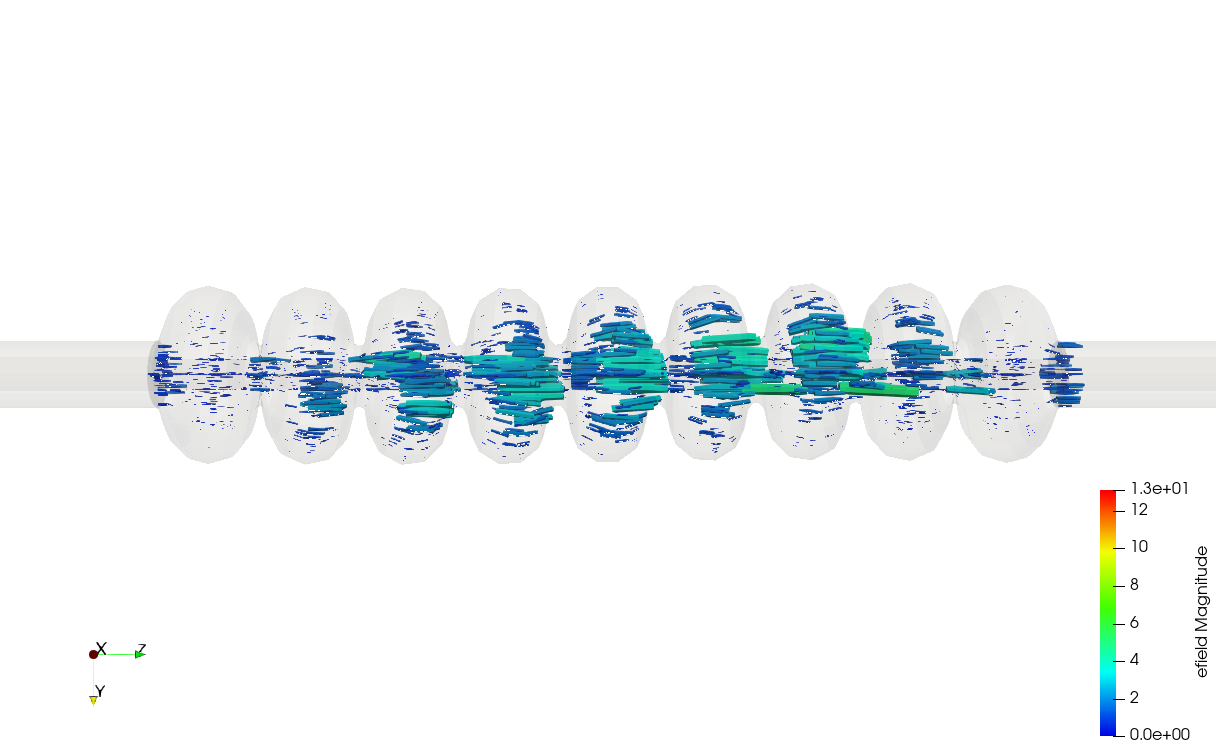}
    \caption{}
    \end{subfigure}\vspace{0.1em}

    \begin{subfigure}[c]{\textwidth}\centering
    \includegraphics[trim = 100 280 100 280, clip, width = \imgwidth]{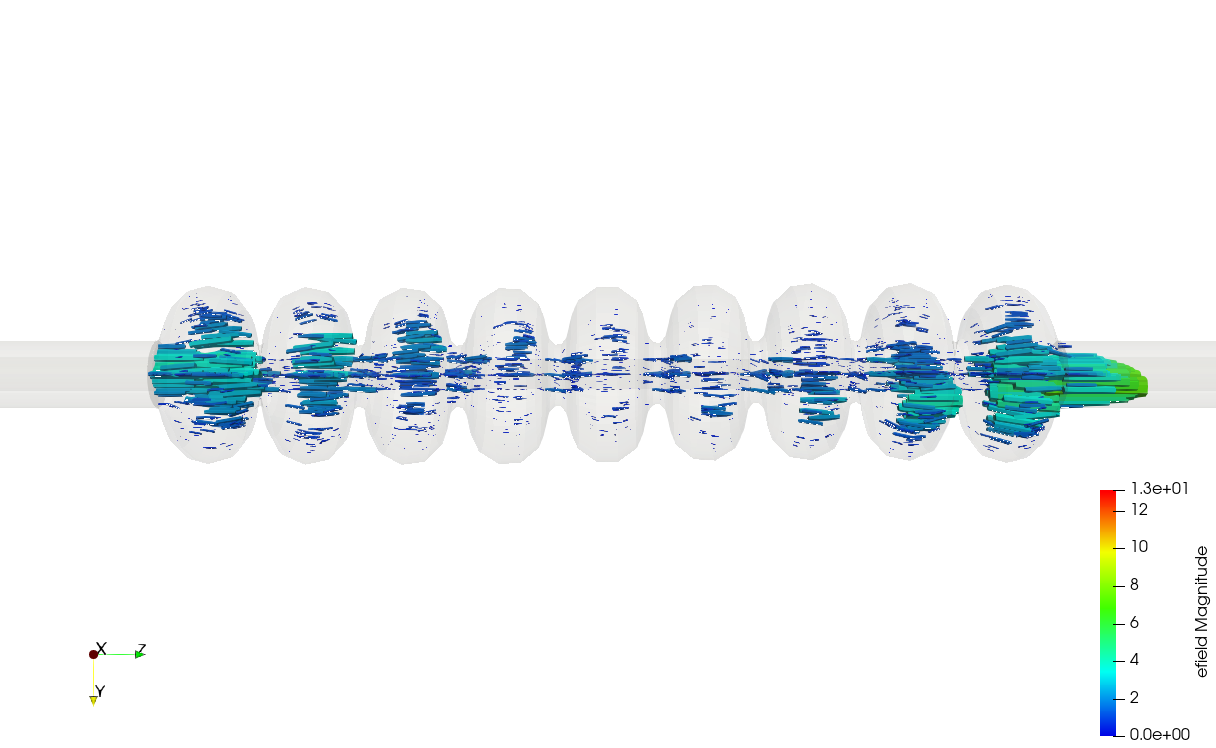}
    \caption{}
    \label{fig:var_emode9}
    \end{subfigure}\vspace{0.1em}
    
    \begin{tikzpicture}[overlay, xshift = 0.4\textwidth, yshift = 0.125\textwidth]
     \def\x{0.15}
     \def\h{-1.52}
     \def\d{2.99}
     \node[anchor=south] at (0,-1.65cm) {\includegraphics [width=0.4cm, height = 3cm] {fig/colorbar}};
    \foreach \y/\val in {0/0, 0.1538/2, 0.3077/4, 0.4615/6, 0.6154/8, 0.7692/10, 0.9231/12, 1/13} \draw (\x,\h+\d*\y)--(\x+0.1,\h+\d*\y) node [anchor=west] {\footnotesize $\val$};
     \node[] at (0.1,1.8) {\footnotesize $\Var[\vec{E}]$};%
\end{tikzpicture}

    \caption{Variance of the $\vec{E}$-field. %
    }
    \label{fig:modeE9var}
\end{figure}

\begin{figure}
    \def\imgwidth{0.48\textwidth}
    \centering
    
    \begin{subfigure}[c]{\imgwidth}%
     {\includegraphics[trim = 110 10 110 10, clip,  height = 0.2\textwidth]{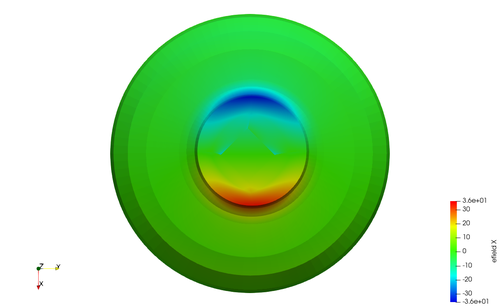}}
     \end{subfigure}~
    \begin{subfigure}[c]{\imgwidth}%
     {\includegraphics[trim = 110 10 110 10, clip,  height = 0.2\textwidth]{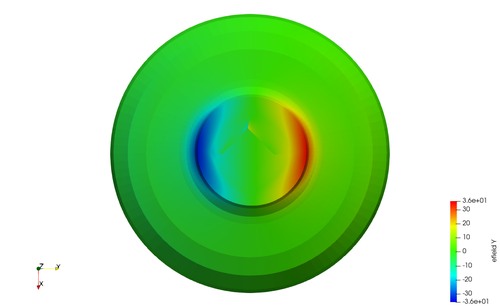}}
    \end{subfigure}

    \begin{subfigure}[c]{\imgwidth}\centering
     {\includegraphics[trim = 322 290 322 290, clip,  width = \textwidth]{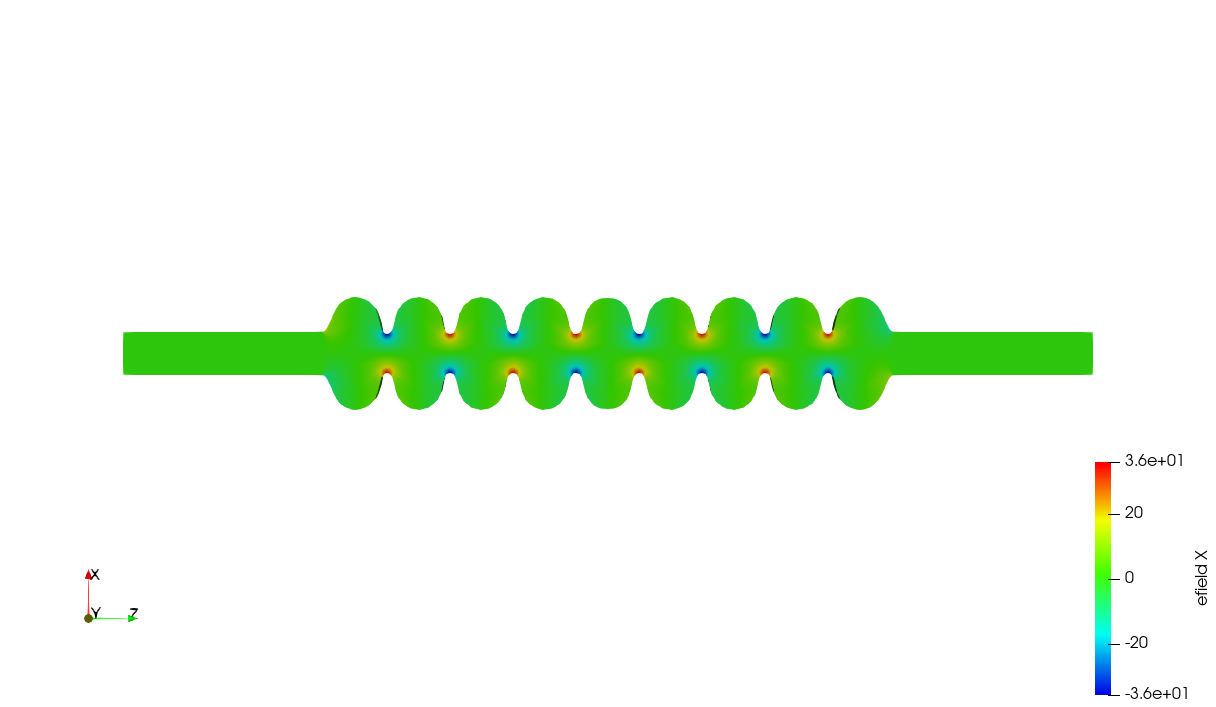}}
     \end{subfigure}~
    \begin{subfigure}[c]{\imgwidth}\centering
     {\includegraphics[trim = 322 290 322 290, clip,  width = \textwidth]{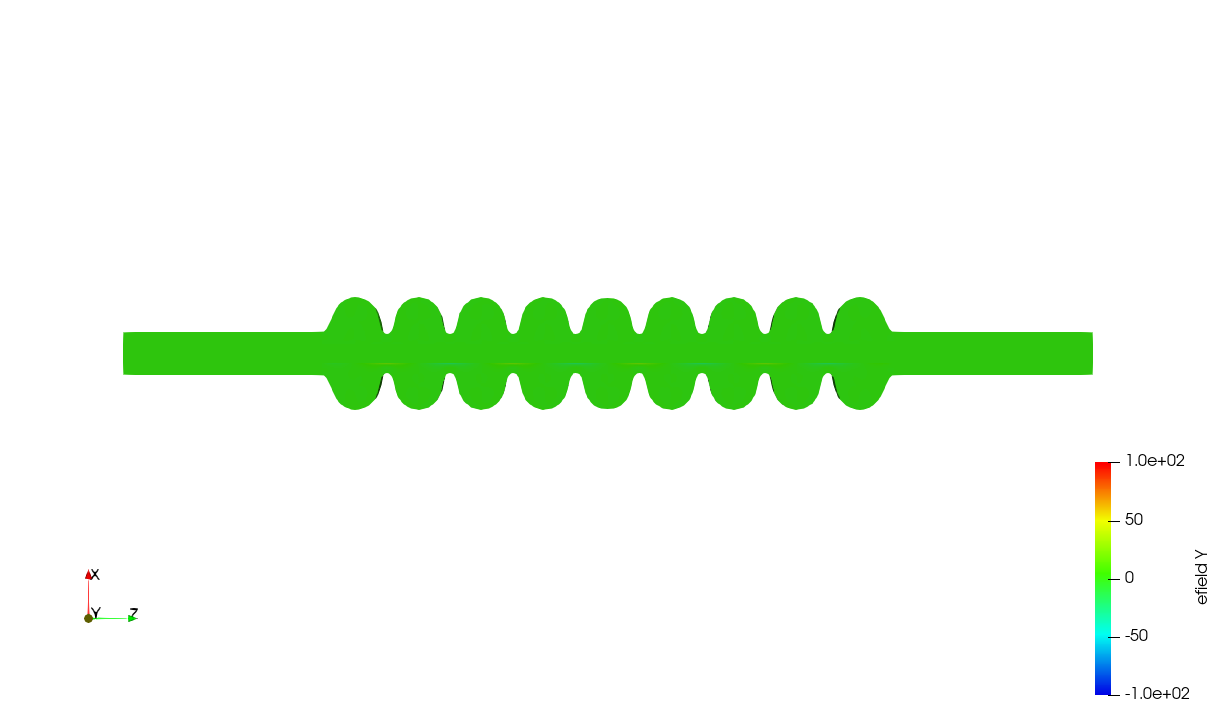}}
    \end{subfigure}
    
    \begin{subfigure}[c]{\imgwidth}\centering
     {\includegraphics[trim = 286 310 286 310, clip,  width = \textwidth]{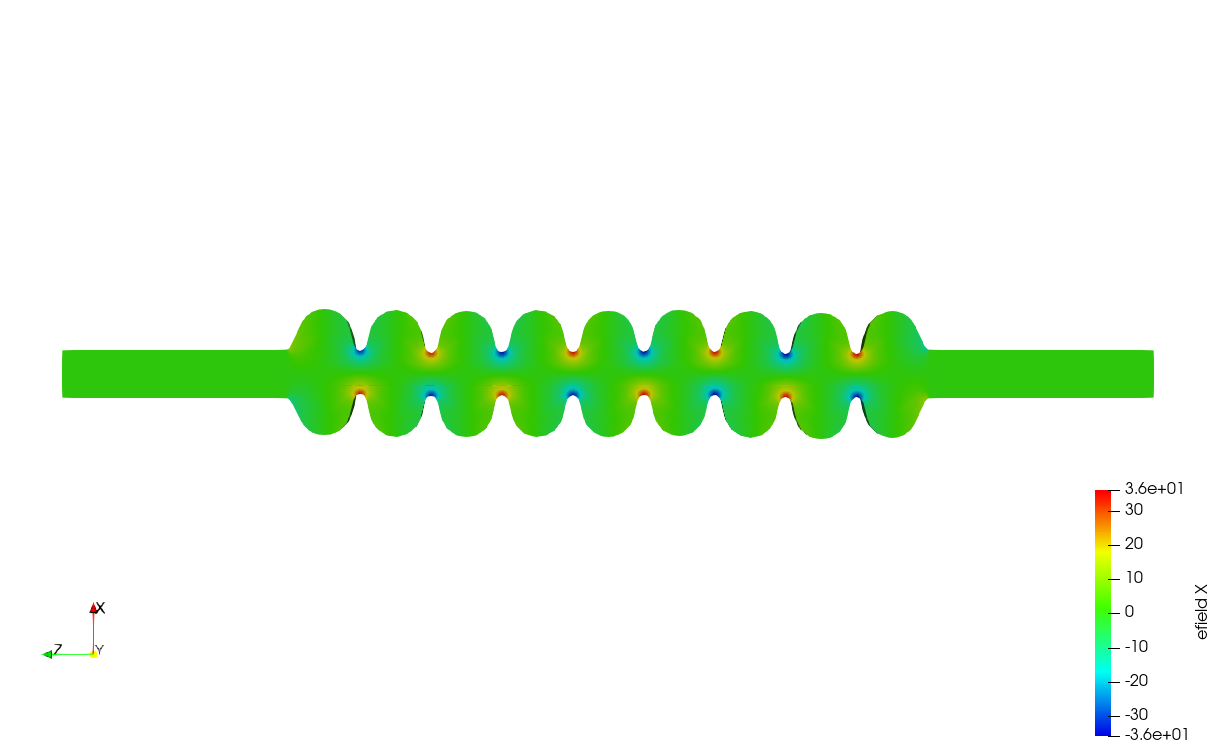}}
     \end{subfigure}~
    \begin{subfigure}[c]{\imgwidth}\centering
     {\includegraphics[trim = 286 310 286 310, clip, width = \textwidth]{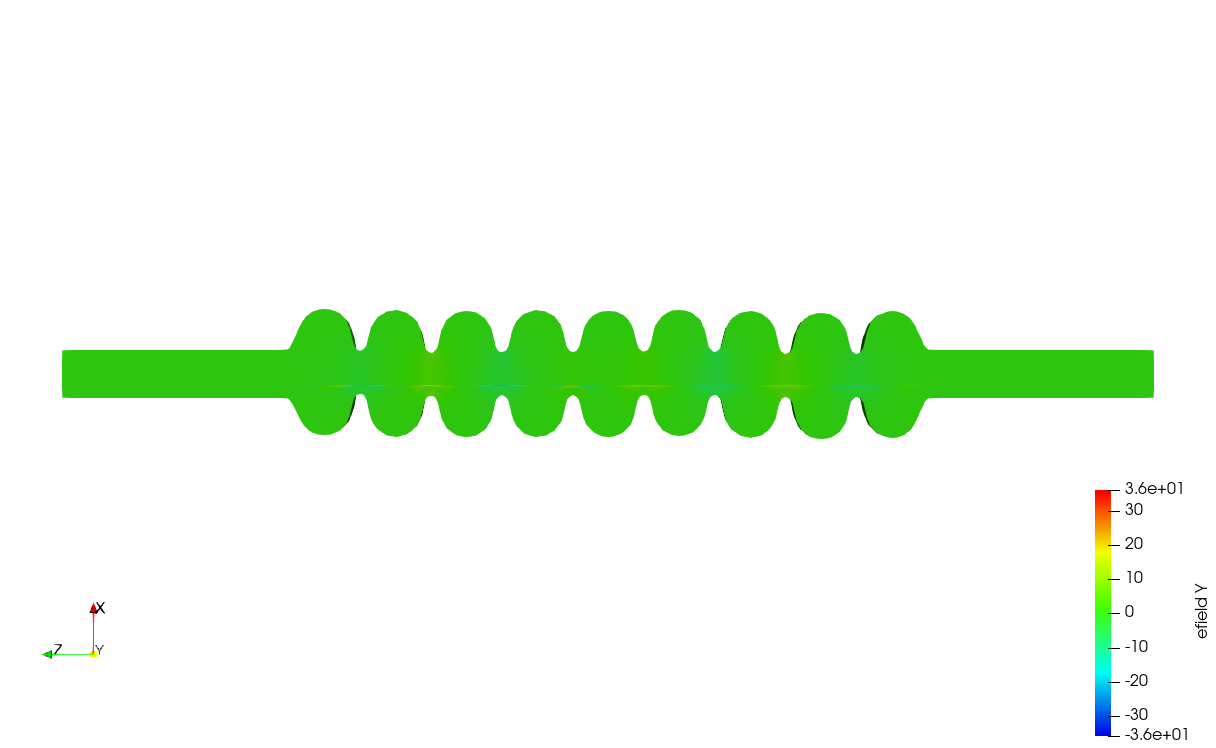}}
    \end{subfigure}
    
    \begin{subfigure}[c]{\imgwidth}\centering
     {\includegraphics[trim = 286 300 286 300, clip,  width = \textwidth]{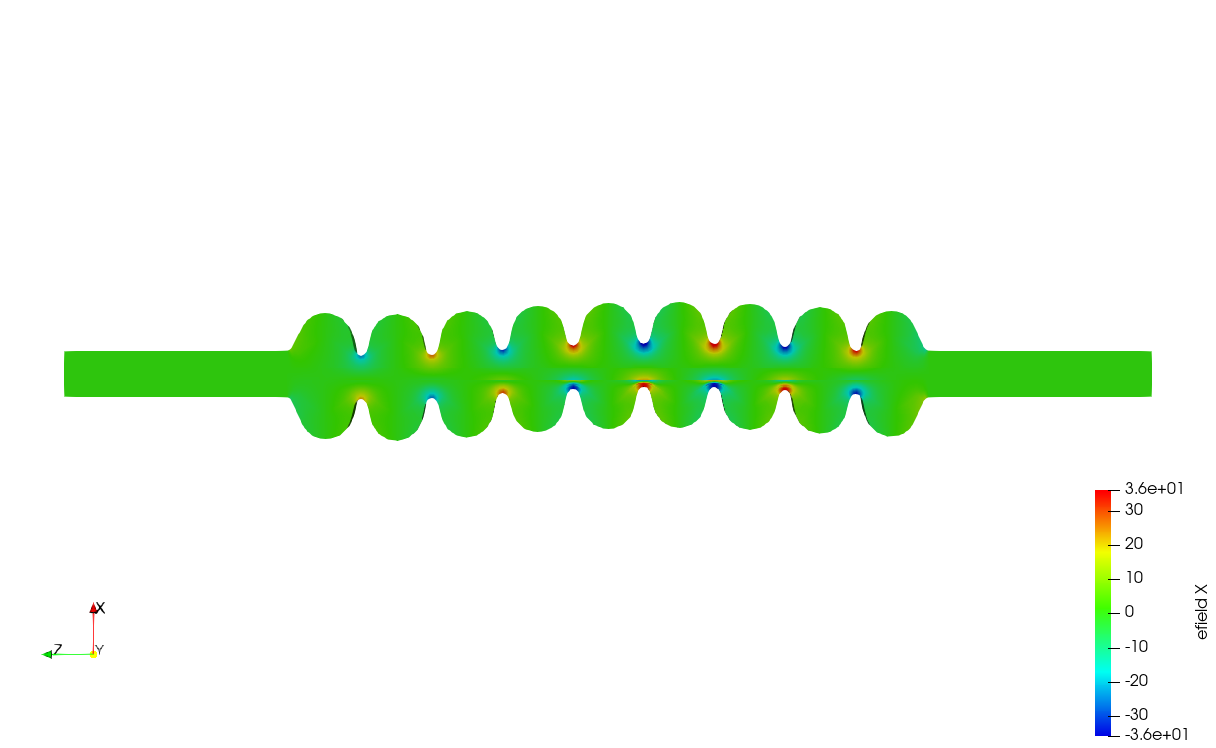}}
     \caption{$\vec{E}_x$, normalized to the $[-1,1]$-range}
     \end{subfigure}~
    \begin{subfigure}[c]{\imgwidth}\centering
     {\includegraphics[trim = 286 300 286 300, clip,  width = \textwidth]{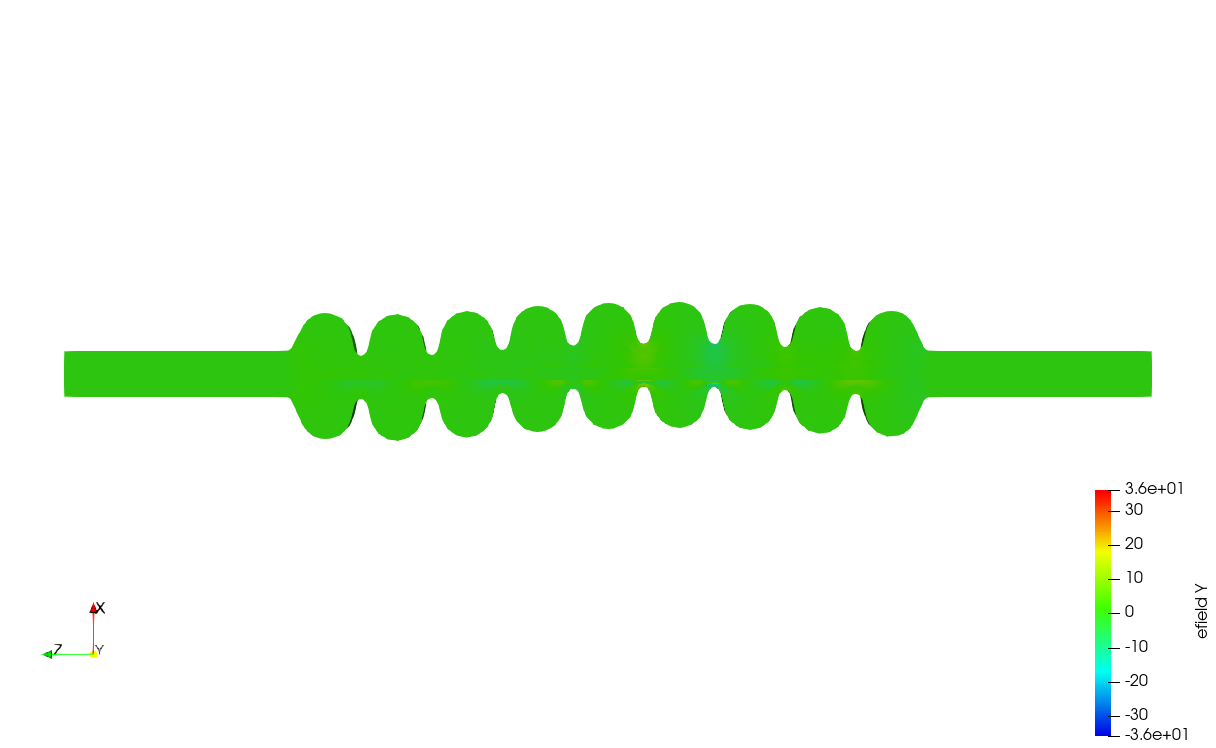}}
     \caption{$\vec{E}_y$, normalized to the $[-1,1]$-range}
    \end{subfigure}\vspace{1em}

    \begin{subfigure}[c]{\imgwidth}\centering
     {\includegraphics[trim = 292 310 292 310, clip,  width = \textwidth]{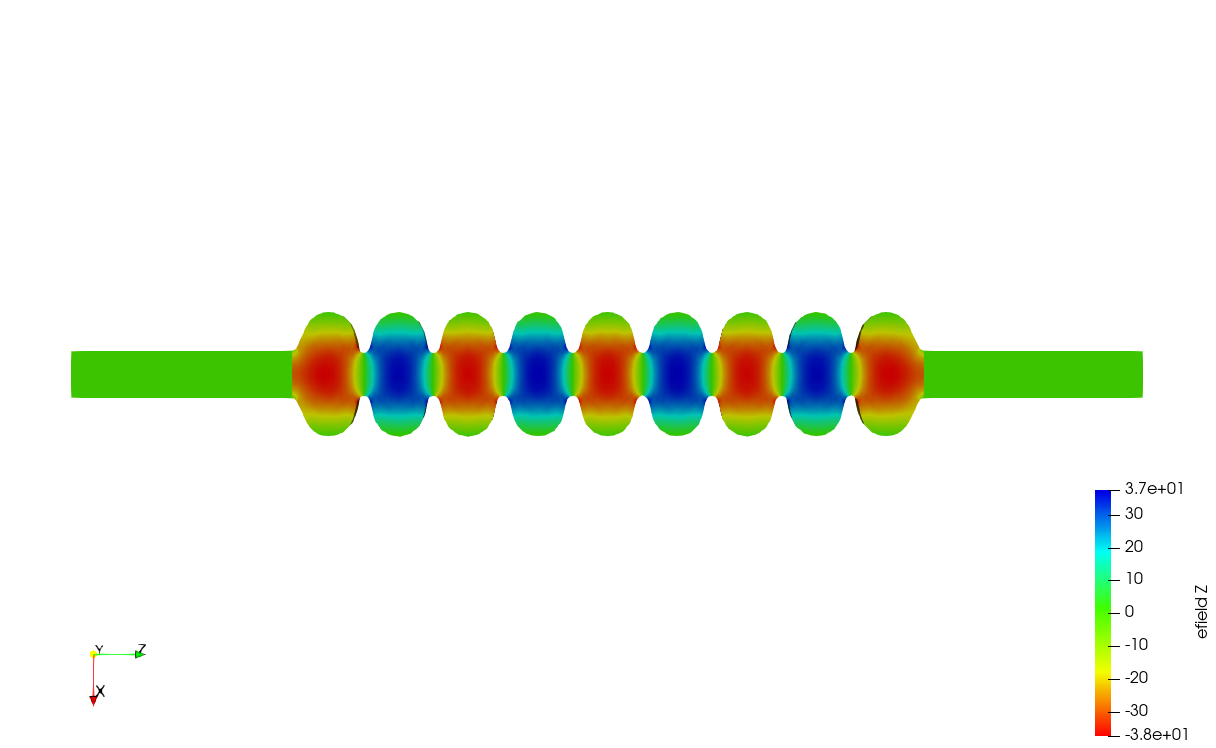}}
     \end{subfigure}~
    \begin{subfigure}[c]{\imgwidth}\centering
     {\includegraphics[trim = 286 310 286 310, clip,  width = \textwidth]{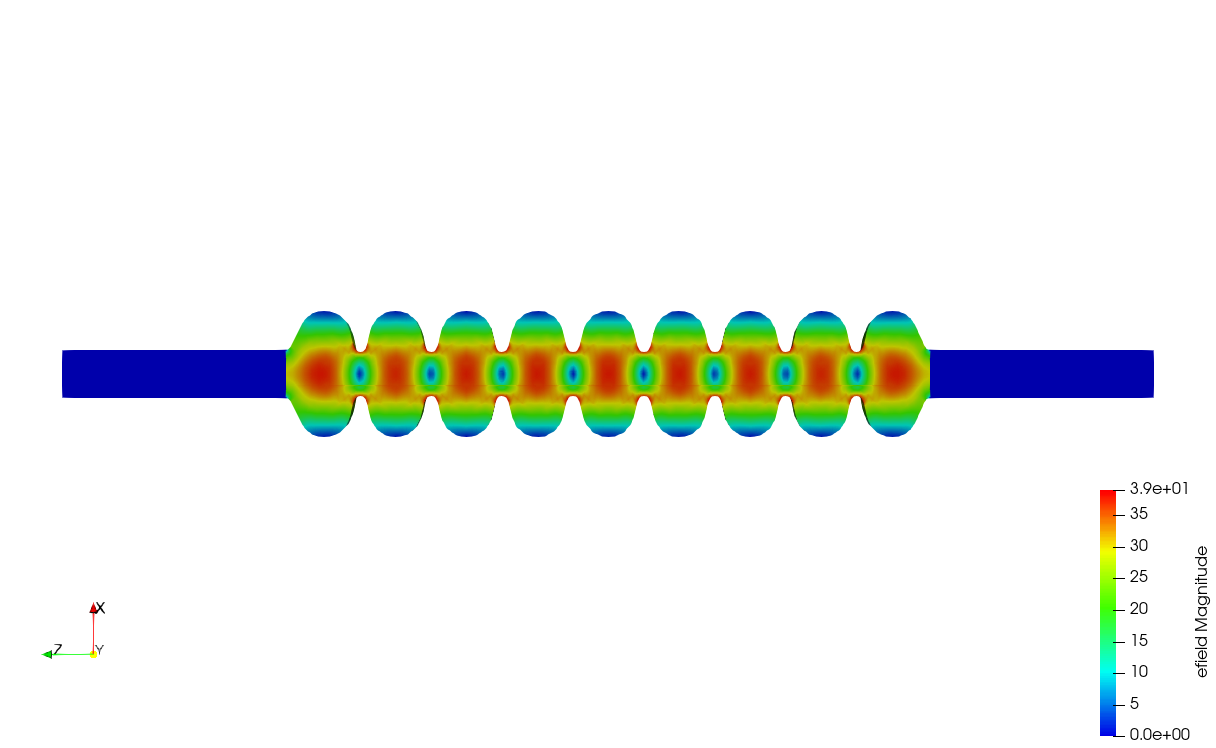}}
    \end{subfigure}\vspace{0.1em}
    
    \begin{subfigure}[c]{\imgwidth}\centering
     {\includegraphics[trim = 292 310 292 310, clip,  width = \textwidth]{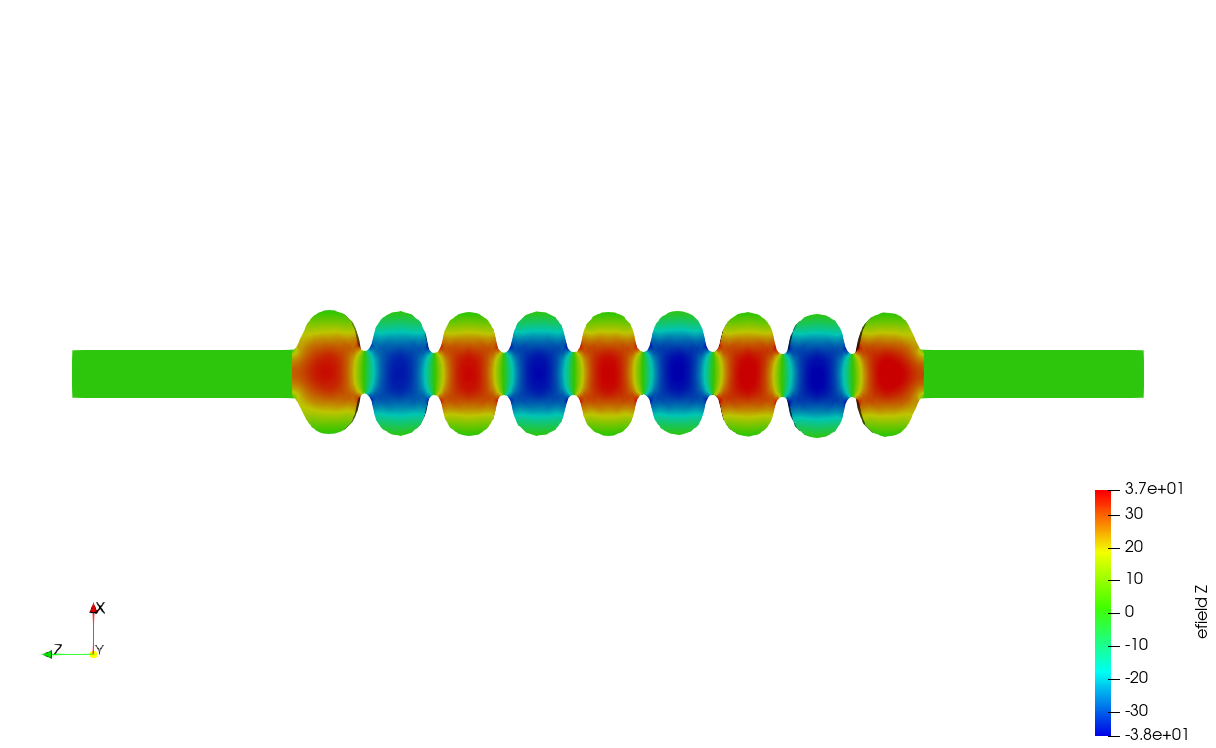}}
     \end{subfigure}~
    \begin{subfigure}[c]{\imgwidth}\centering
     {\includegraphics[trim = 286 310 286 310, clip,  width = \textwidth]{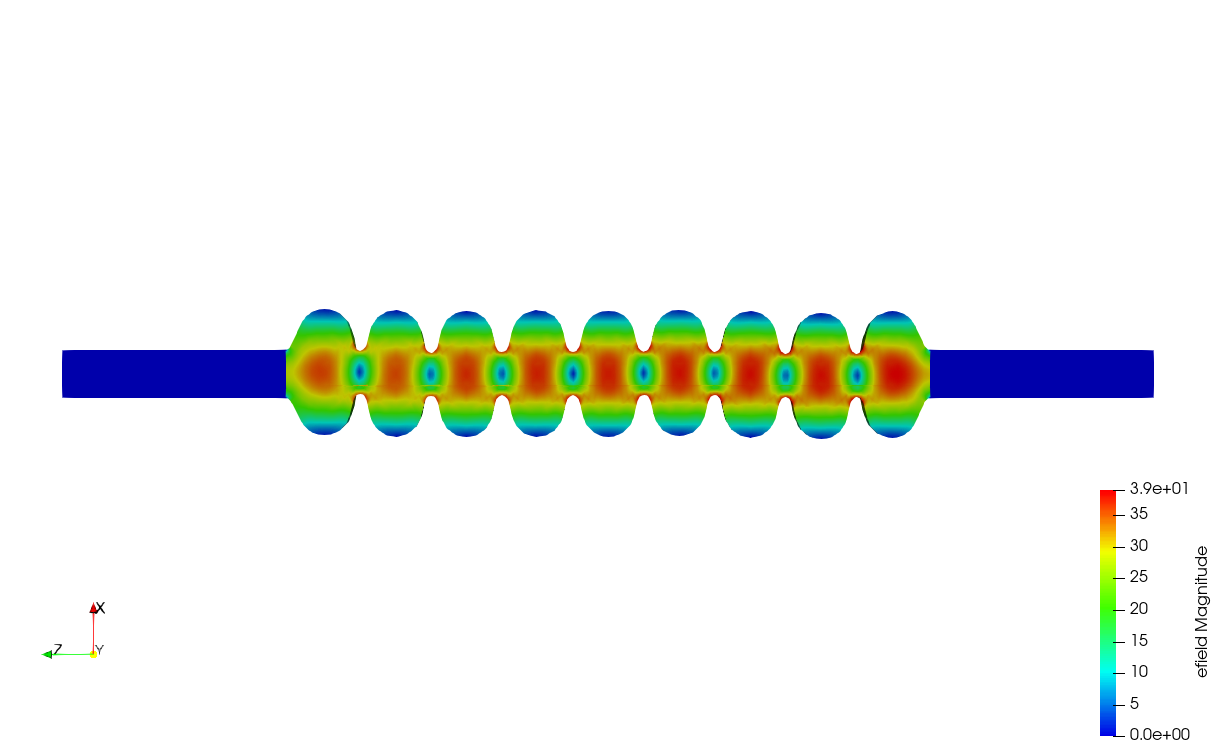}}
    \end{subfigure}
    
    \begin{subfigure}[c]{\imgwidth}\centering
     {\includegraphics[trim = 292 300 292 300, clip,  width = \textwidth]{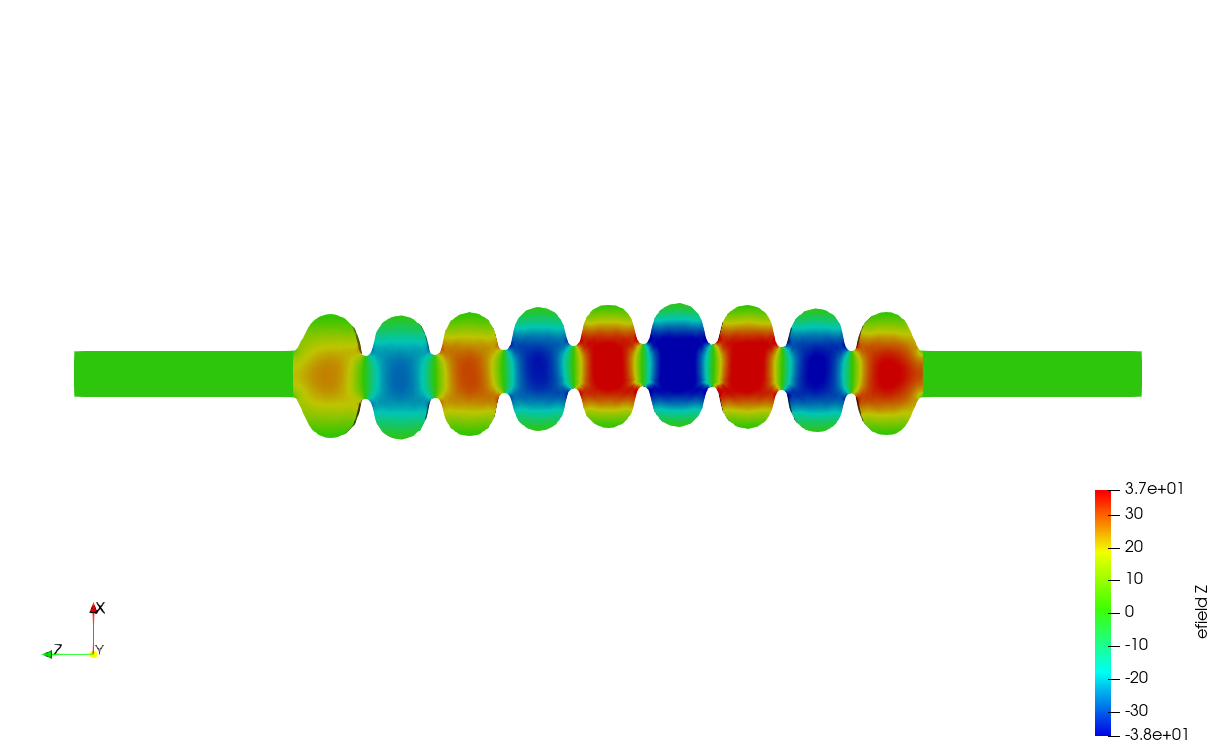}}
     \caption{$\vec{E}_z$, normalized to the $[-1,1]$-range}
     \end{subfigure}~
    \begin{subfigure}[c]{\imgwidth}\centering
     {\includegraphics[trim = 286 300 286 300, clip,  width = \textwidth]{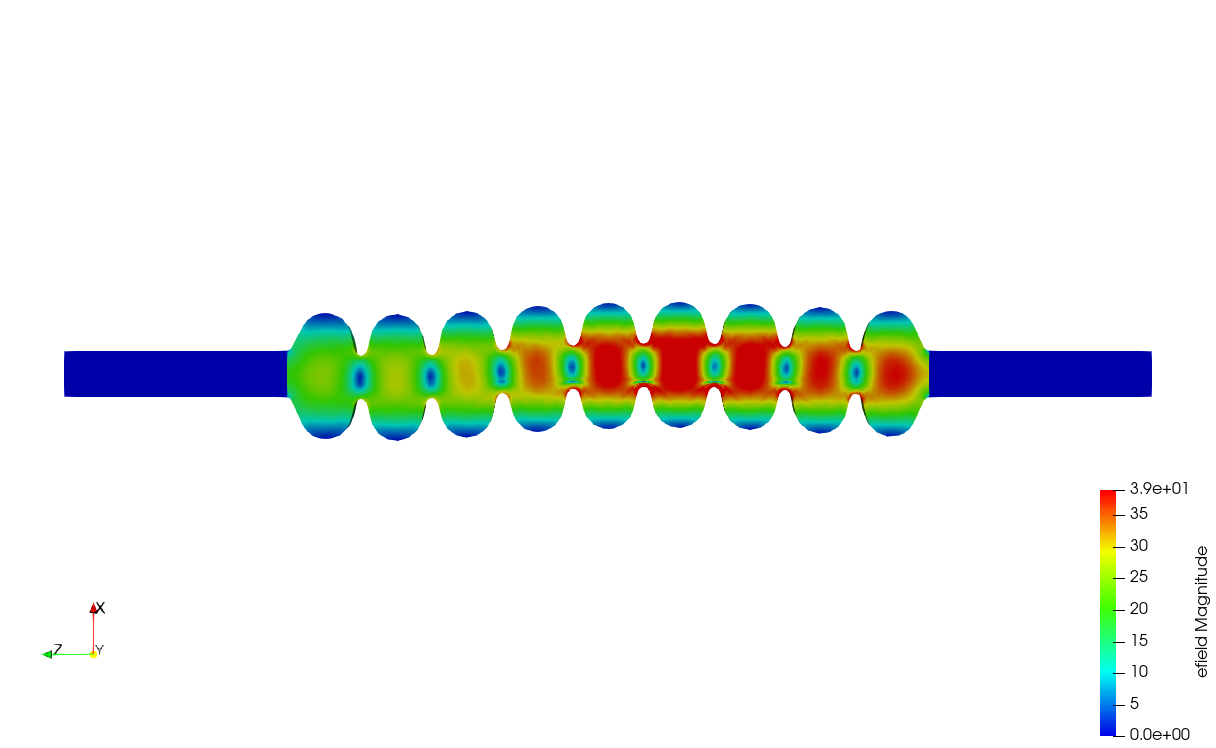}}
    \caption{$|\vec{E}|$, normalized to the $[0,1]$-range}
    \end{subfigure}\vspace{0.1em}
    
    \begin{tikzpicture}[overlay, xshift = 0.4\textwidth, yshift = .9\textwidth]
     \def\x{0.15}
     \def\h{-1.52}
     \def\d{1.486}
     \node[anchor=south] at (0,-1.65cm) {\includegraphics [width=0.4cm, height = 1.5cm] {fig/colorbar}};
     \foreach \y in {0,0.2,0.4,0.6,0.8,1} \draw (\x,\h+\d*\y)--(\x+0.1,\h+\d*\y) node [anchor=west] {\footnotesize $\y$};
     \node[] at (0.5,0.35) {\footnotesize $|\vec{E}|$};%
     \foreach \y in {0,0.25,0.5,0.75,1} \pgfmathsetmacro\z{(\y*2-1)} \draw (-\x,\h+\d*\y)--(-\x-0.1,\h+\d*\y) node [anchor=east] {\footnotesize $\z$};
     \node[] at (-0.5,0.35) {\footnotesize $\vec{E}_i$};%
\end{tikzpicture}

    \begin{tikzpicture}[overlay, xshift = -0.5\textwidth, yshift = 0.554\textheight]
     \draw[red, very thick] (0.067\textwidth, -0.9) -- (0.067\textwidth, -2.3);
     \draw[red, very thick] (0.5625\textwidth, -0.9) -- (0.5625\textwidth, -2.3);
    \end{tikzpicture}

    \caption{Normalized components and magnitude of the $\vec{E}$-field of the ninth mode (i.e. the accelerating mode). 
    At the top we see the cross-sections in the $z$-plane at the cutting planes marked in red. 
    In each subfigure, the field plot of the nominal design (top), the mean geometry $\domain_0$ (middle; cf. \cref{fig:xy_deformations}, subplot (a)) and a sampled deformed cavity $\domain_\vt$ (bottom).}
    \label{fig:modeE9_componentDef}
\end{figure}

\subsection{Computational Wall Clock Time}

For the perturbation approach and sampling for the PCEs, we initialize the computations by assembling the system matrices of the unperturbed cavity and solving the corresponding eigenvalue problem.
Next, when applying our novel approach, we compute the derivatives of system matrices with respect to the deformation using the software published in~\cite{Ziegler_github} for each of the seven deformation modes. 
This step can be fully parallelized and then takes \SI{2}{\min}, averaged over $10$ runs. 
Then, for each deformation mode, we assemble and solve system~\eqref{eq:der_eigpair:disc} and, given the derivatives~$\left[\domain_\vt \vec{E}_{\vt}\right]_i$ and $\left[\domain_\vt \vec{\lambda}_{\vt}\right]_i$ for deformation modes $i = 1, \ldots, M$, compute the covariance~\eqref{eq:KLE:tensor_lu}.
Since again, the contribution of each deformation mode can be computed independently, this task can be fully parallelized and is completed after \SI{25}{\s}, on average.

In contrast, for building the PCE surrogate model, we loop over $17$ iterations.
In each, we evaluate $1,000$ samples, i.e. we assemble the system matrices at $1,000$ Latin hypercube samples, solve the eigenvalue problem to find the $27$ eigenpairs described above.
The samples are computed in a \texttt{parfor} loop.
When distributing the task among $32$ workers, the computations per iteration take on average \SI{17}{\min}. 
After calculating the samples, we use UQLab on a single core with the above settings to estimate the weights of PCEs for each entry of the sampled vectors containing eigenvalue and eigenfunction data, which takes around \SI{30}{\min} per eigenpair and iteration.
Therefore, the derivative-based uncertainty quantification, where the full computation over all $27$ eigenpairs is completed in under \SI{3}{\min}, is dramatically more efficient.
The precise timing values are however subject to the prototype character of the code and are therefore to be understood rather as points of reference.

\section{Conclusion} \label{sec:conclusion}
We considered the uncertainty quantification of the Maxwell eigenvalue problem on random domains by means of a perturbation approach. The considered eigenvalues can be of higher but finite multiplicity and can thus also deal with crossings or bifurcations in the eigenvalue trajectories. To this end, we used the domain mapping approach to convert the eigenvalue problem on random domains into an eigenvalue problem with uncertain coefficients. The uncertainty in this eigenvalue problem was quantified using an extension of the approach of \cite{DE2024}, for which the derivatives of the domain transformation coefficients were required. These derivatives were taken from \cite{ziegler_mode_2022}. Using isogeometric analysis for the representation of the geometry allows for a straightforward modelling of the deformations and computation of the required derivatives.
We applied our theory for the uncertainty quantification of a three-dimensional 9-cell TESLA cavity with a random deformation model obtained from real-world data from the DESY database.

\bibliographystyle{elsarticle-num}

\end{document}